\documentclass[10pt]{amsart}

\usepackage{amsmath}
\usepackage{graphicx}

\usepackage{color} 
\usepackage[dvipsnames]{xcolor}

\usepackage{amssymb}

\usepackage{euscript,color}

\definecolor{red}{rgb}{1,0,0}

\newtheorem{thm}{Theorem}[section]
\newtheorem{prop}[thm]{Proposition}
\newtheorem{cor}[thm]{Corollary}

\theoremstyle{definition}
\newtheorem{defn}[thm]{Definition}
\newtheorem{ex}[thm]{Example}

\theoremstyle{remark}
\newtheorem{remark}[thm]{Remark}

\numberwithin{equation}{section}

\def\bending{\kappa}
\def\R{\mathbb{R}}
\def\curvature{{\rm k}}
\def\SL{{\rm SL}(2,\R)}
\def\Aut{{\rm Aut}(\R^{2,2})}
\def\A{{\rm Aut}^{\uparrow}_+(\R^{2,2})}
\def\Cartan{\mathcal{C}^{\uparrow}_+(\R^{2,2})}
\def\AdS{{\rm AdS}}

\def\arctanh{{\rm arctanh}}

\begin{document}
\title[Integrable Flows on Null Curves in the Anti-de Sitter $3$-Space]{Integrable Flows on Null Curves in the Anti-de Sitter $3$-Space}

\author{Emilio Musso}
\address{(E. Musso) Dipartimento di Matematica, Politecnico di Torino,
Corso Duca degli Abruzzi 24, I-10129 Torino, Italy}
\email{emilio.musso@polito.it}

\author{\'Alvaro P\'ampano}
\address{(A. P\'ampano) Department of Mathematics and Statistics, Texas Tech University, Lubbock, TX, 79409, USA}
\email{alvaro.pampano@ttu.edu}

\thanks{Authors partially supported by PRIN projects  PRIN 2017 ``Real and Complex Manifolds: Topology, Geometry and Holomorphic Dynamics" (protocollo 2017JZ2SW5-004) and PRIN 2022 ``Real and Complex Manifolds: Geometry and Holomorphic Dynamics" (protocollo 2022AP8HZ9) and by the GNSAGA of INDAM. The authors gratefully acknowledge the warm hospitality of the Department of Mathematics at the Politecnico di Torino and of the Texas Tech University Center in Sevilla.}

\subjclass[2010]{}

\date{\today}

\maketitle

\begin{abstract}
We formulate integrable flows related to the KdV hierarchy on null curves in the anti-de Sitter $3$-space ($\AdS$). Exploiting the specific properties of the geometry of $\AdS$, we analyze their interrelationships with Pinkall flows in centro-affine geometry. We show that closed stationary solutions of the lower order flow can be explicitly found in terms of periodic solutions of a Lam\'e equation. In addition,  we study the evolution of non-stationary curves arising from a 3-parameter family of periodic solutions of the KdV equation.\\

\end{abstract}

\section{Introduction} 

The Korteweg-De Vries (KdV) equation is a partial differential equation (PDE) which has a long history and a great amount of literature about it (see, for instance, \cite{GGKM,L1,L2,ZF}). Originally, the equation appeared as a model to understand the propagation of waves on shallow water surfaces (\cite{KV}). This equation belongs to a hierarchy of higher order evolution PDEs and is a prototype of completely integrable evolution equations (\cite{CIB1,CIB2,CQ1,CQ,CQ3,CQ4,Di,FK1,FK2,TW1,ZF}, among others). 

Throughout its history, the KdV equation has appeared in several geometric contexts. For instance, integrable geometric flows governed by the KdV equation and its hierarchy appear in the following topics: in the centro-affine geometry with the Pinkall flows on star-shaped curves (\cite{CP, P}), in the study of nondegenerate isothermic submanifolds in real and complex projective spaces (see the last pages of \cite{BDFP}) and, closely related to the topic of this paper, in the evolution of null curves in the Lorentz-Minkowski $(2+1)$-space (\cite{AGL,MN1Bis,MN2}). In particular, the work \cite{MN2} clearly suggests the existence of analogue evolution equations for null curves in the de Sitter and in the anti-de Sitter $3$-spaces. The specific case of the anti-de Sitter $3$-space ($\AdS$, for short) is special since for this ambient space the theory of null curves can be analyzed by combining together pairs of star-shaped curves in the centro-affine plane. This allows us to find a relation between the evolution of null curves and the Pinkall flows.

On the ground of these considerations, in this paper we investigate integrable geometric flows related to the KdV hierarchy on null curves in $\AdS$. More precisely, the paper is organized as follows: Section 2 collects the basic information about the geometry of $\AdS$ in the $\SL$ model. A peculiarity of $\AdS$ is that its automorphism group is not simple and its $2:1$ spin-cover is $\SL\times\SL$. Section 3 is devoted to examine the basic properties of null curves without inflection points in $\AdS$, which we denote by $\gamma:J\subseteq\R\longrightarrow\AdS$. Such curves possess canonical parameterizations and a third order differential invariant $\bending$, the \emph{bending} (often known as curvature or torsion). In addition, they also have a canonical \emph{Cartan frame field} $\mathcal{F}=\{\gamma,T,N,B\}$ defined along them, which originates a lift $(F_+,F_-):J\subseteq\R\longrightarrow\SL\times\SL$, the \emph{spinor frame field} along $\gamma$. The components $F_\pm$ of the spinor frame field along $\gamma$ are, precisely, the canonical central affine frame fields of two star-shaped curves $\eta_\pm$ in the centro-affine plane $\dot{\R}^2$ with central affine curvatures $\curvature_\pm=\bending\pm 1$, respectively (see Theorem \ref{relation}). The pair of star-shaped curves $(\eta_+,\eta_-)$ is referred as to the \emph{pair of cousins} associated with the null curve $\gamma$. 

In Section 4 we consider the evolution equation for null curves in $\AdS$ given by
\begin{equation*}
	\partial_t\gamma=-2\sqrt{2}\,\left(\bending\, T+2B\right),
\end{equation*}
where $\gamma(s,t)=\gamma_t(s)$, $t\in I\subseteq\R$, is a one parameter family of null curves parameterized by the ``proper time'' $s\in J$. For brevity and in analogy with \cite{AGL}, we call the above equation the \emph{LIEN flow}. We prove, in Theorem \ref{induced}, that the induced evolution equation on the bending $\bending(s,t)$ of $\gamma(s,t)$ is the KdV equation in the form
\begin{equation*}
	\partial_t\bending+\partial_{s}^3\bending-6\bending\,\partial_s\bending=0\,.
\end{equation*}
Moreover, we discuss how to build solutions of the LIEN flow beginning with solutions of the KdV. The starting point is the Lax pair formulation of the KdV equation given by Zakharov and Faddeev (\cite{ZF}). From this result, a one parameter family of maps $E_\lambda:J\times I\subseteq\R^2\longrightarrow\SL\times\SL$, $\lambda\in\R$, can be associated to a solution $\bending(s,t)$ of the KdV equation. Using the terminology of \cite{TU}, we call $E_\lambda$ the \emph{extended frame} of $\bending$ with spectral parameter $\lambda\in\R$. In Theorem \ref{induced}, we also show that given a solution $\bending(s,t)$ of the KdV equation, the map $\gamma=E_1 E_{-1}^{-1}:J\times I\subseteq\R^2\longrightarrow\AdS$ is a solution of the LIEN flow with bending $\bending(s,t)$. 

Subsequently, in Section 5, we focus on the stationary curves of the LIEN flow. As it is the case of the flows in the Lorentz-Minkowski $(2+1)$-space, stationary curves of the flow in $\AdS$ are critical points of functionals depending linearly on the first two conserved quantities of the KdV hierarchy. Such a variational problem has been considered in the last couple of decades by several authors (see for instance \cite{AGL,ABG,BFJL,MN1,MN1Bis} and references therein). The bending $\bending$ of these stationary curves is a solution of the third order ordinary differential equation
\begin{equation*}\label{ode}
	\bending'''+2\ell\bending'-6\bending\bending'=0\,,
\end{equation*}
where $\ell\in\R$ and $\left(\,\right)'$ denotes the derivative with respect to the proper time $s$. The bending of the evolution of stationary curves by the LIEN flow is the traveling wave solution $\bending(s+2\ell t)$ of the KdV equation. The general solutions of the above ordinary differential equation can be expressed in terms of elliptic functions. In particular, the periodic ones can be expressed in terms of the square of the Jacobi's ${\rm sn}$ function. In Theorem \ref{stationarycurve} we prove that the stationary curves with nonconstant periodic bending of the LIEN flow in $\AdS$ can be explicitly integrated employing the fundamental solutions of the Lam\'e equation of order one, in the case that these solutions are both periodic. The Floquet eigenvalue problem for the Lam\'e differential equation (see for instance \cite{V}) plays an essential role in constructing closed stationary curves of the flow and their time evolution. In fact, these solutions of the LIEN flow involve either Lam\'e functions of order one or else functions which involve the Jacobi's ${\rm sn}$ function and the local Heun functions (\cite{OLBC}).

In Section 6 we consider the 3-parameter family of periodic solutions of the KdV equation studied in \cite{KKSH} and investigate the corresponding evolution of the LIEN flow. Unlike the stationary case, the integration of these curves relies on solutions of Hill's equations that cannot be explicitly written in terms of known special functions. Therefore, our analysis is essentially based on numerical solutions of such equations.

The LIEN flow belongs to an infinite hierarchy of evolution equations on null curves, all of the form $\partial_t\gamma=\mathfrak{a}_nT+\mathfrak{b}_nN+\mathfrak{c}_nB$, where $\mathfrak{a}_n,\mathfrak{b}_n,\mathfrak{c}_n$ are functions (polynomial) of the bending $\bending$ and higher order derivatives with respect to the proper time $s$. The induced evolution equation on $\bending$ is the whole KdV hierarchy (the explicit construction of the hierarchy is given in Appendix B).

The results of this paper could be extended to other hierarchies of integrable evolution equations such as the mKdV, the Kaup-Kupershmidt, and Sawada-Kotera hierarchies (see, for instance, \cite{CIM,CQ1,CQ,CQ3,CQ4}). These hierarchies appear as the integrability conditions of geometric flows on curves as well: the mKdV hierarchy is related to the flows on curves in $2$-dimensional Riemannian space forms and flows on Legendrian curves of  $\mathbb{S}^3$ with its standard pseudo-Hermitian structure (\cite{Web}), the Kaup-Kupershmidt hierarchy appears in the context of flows on curves in $\mathbb{RP}^2$ and in the $3$-dimensional centro-affine space as well as flows on Legendrian curves in $\mathbb{S}^3$ with its standard Cauchy-Riemann structure (\cite{ChM}),  and the Sawada-Kotera hierarchy is related to the integrable flows on curves in the affine plane.

\section{Anti-de Sitter $3$-Space}

In this section we will introduce the model for the anti-de Sitter $3$-space ($\AdS$) given by the special linear group $\SL$ and describe its basic features.

Consider the vector space of $2\times 2$ real matrices $\R^{2,2}$ equipped with the quadratic form ${\rm q}$ of signature $(-,-,+,+)$ defined by
$${\rm q}(X)=-{\rm det}(X)=x_1^2x_2^1-x_1^1x_2^2\,,$$
for each $X=(x_i^j)\in\R^{2,2}$. The corresponding inner product will be denoted by $\langle\cdot,\cdot\rangle$ and we will consider the orientation of $\R^{2,2}$ determined by the volume form $\Omega=dx_1^1\wedge dx_2^2\wedge dx_1^2\wedge dx_2^1$.

On $\Lambda^2(\R^{2,2})$ we define an inner product $\langle\langle\cdot,\cdot\rangle\rangle$ by
$$\langle\langle U\wedge V,W\wedge Z\rangle\rangle={\rm det}\begin{pmatrix} \langle U,W\rangle & \langle U,Z\rangle \\ \langle V,W\rangle & \langle V,Z\rangle \end{pmatrix}.$$
A bivector $U\wedge V\in\Lambda^2(\R^{2,2})$ is of type $(-,-)$ if the restriction of the above inner product to ${\rm span}\{U\wedge V\}$ is negative definite, and of type $(-,0)$ if this restriction is nonzero semi-negative definite and degenerate. It can be shown that if $U\wedge V\in\Lambda^2(\R^{2,2})$ is of type $(-,-)$ and $W\wedge Z\in\Lambda^2(\R^{2,2})$ is of type $(-,0)$, then $\langle\langle U\wedge V,W\wedge Z\rangle\rangle\neq 0$. We fix the bivector of type $(-,-)$
$$U\wedge V=\begin{pmatrix} 1 & 0 \\ 0 & 1 \end{pmatrix}\wedge\begin{pmatrix} 0 & 1 \\ -1 & 0 \end{pmatrix}\in\Lambda^2(\R^{2,2})\,,$$ 
and we say that a bivector $W\wedge Z$ of type $(-,0)$ is \emph{positive} if $\langle\langle U\wedge V,W\wedge Z\rangle\rangle>0$. The set of all positive bivectors of type $(-,0)$ is denoted by $\mathcal{N}_+$. 
This choice defines the notion of time orientation in $\R^{2,2}$.


Let $\Aut$ be the $6$-dimensional Lie group consisting of all linear isometries $L:\R^{2,2}\longrightarrow\R^{2,2}$. This Lie group has four connected components. The connected component of the identity is
$$\A=\{L\in\Aut\,\lvert\,L^*(\Omega)=\Omega,\,L\wedge L(\mathcal{N}_+)=\mathcal{N}_+\}\,,$$
and its elements are \emph{causal automorphisms} of $\R^{2,2}$ (ie. automorphisms that preserve the choice of time orientation) that also preserve the orientation.

A \emph{Cartan basis} of $\R^{2,2}$ is a basis $C=\{C_1,C_2,C_3,C_4\}$ of $\R^{2,2}$ such that
$$g=\left(\langle C_i,C_j\rangle\right)_{i,j=1,...,4}=\begin{pmatrix} -1 & 0 & 0 & 0 \\ 0 & 0 & 0 & 1 \\ 0 & 0 & 1 & 0 \\ 0 & 1 & 0 & 0 \end{pmatrix}.$$
We say that the basis $C$ is \emph{positively oriented} if $\Omega(C)=1$ and \emph{future-directed} if $C_1\wedge C_2\in\mathcal{N}_+$. We denote by $\Cartan$ the set of all positively oriented and future-directed Cartan bases of $\R^{2,2}$. The group $\A$ acts simply transitively on the left of $\Cartan$ and, hence, $\Cartan$ carries the differentiable structure inherited from $\A$. 

With some abuse of notation, denote by $C_j$, $j=1,...,4$, the map that associates to each Cartan basis $C$ the $j$-th element of $C$. Differentiating these maps we have\footnote{Throughout this paper, the Einstein summation convention will be used.}
\begin{equation}\label{dC}
	dC_j=\omega_{j}^i C_i\,,
\end{equation} 
for every $j=1,...,4$, where $\omega_j^i$, $i,j=1,...,4$, are exterior differential $1$-forms on $\Cartan$. Differentiating \eqref{dC}, we get
\begin{equation}\label{domega}
	d\omega_j^i=-\omega_k^i\wedge\omega_j^k\,,
\end{equation}
$i,j=1,...,4$, which are the \emph{structure equations} of the frame manifold $\Cartan$. Differentiating now $g_{ij}=\langle C_i,C_j\rangle$ and using \eqref{dC}, we obtain 
\begin{equation}\label{homega}
	g_{ik}\omega_j^k+g_{jk}\omega_i^k=0\,,
\end{equation}
for all $i,j=1,...,4$. From \eqref{homega} we conclude that the matrix-valued $1$-form $\omega=\left(\omega_j^i\right)$ takes values in the Lie algebra $\mathfrak{g}=\{X\in\R^{2,2}\,\lvert\,X^t g+g X=0\}$. Thus,
\begin{equation}\label{omega}
	\omega=\begin{pmatrix} 0 & \omega_1^4 & \omega_1^3 & \omega_1^2 \\ \omega_1^2 & \omega_2^2 & -\omega_4^3 & 0 \\ \omega_1^3 & \omega_2^3 & 0 & \omega_4^3 \\ \omega_1^4 & 0 & -\omega_2^3 & -\omega_2^2 \end{pmatrix}.
\end{equation}
The $1$-forms $\omega_1^2$, $\omega_1^3$, $\omega_1^4$, $\omega_2^2$, $\omega_2^3$ and $\omega_4^3$ define an $\A$-invariant parallelization on $\Cartan$.

Let $\SL$ be the \emph{special linear group} of degree $2$ over $\R$ (ie. the group consisting of the $2\times 2$ real matrices of determinant one with the ordinary matrix multiplication) and consider the Lie group $\SL\times\SL$. For $(A,B)\in\SL\times\SL$, define the linear map
$$L_{(A,B)}:X\in\R^{2,2}\longmapsto A X B^{-1}\in\R^{2,2}\,.$$
Then, $\widehat{L}:(A,B)\in\SL\times\SL\longmapsto L_{(A,B)}\in\A$ is the \emph{$2:1$ spin-covering homomorphism} of $\A$. If we choose the Cartan basis $\{P_1,P_2,P_3,P_4\}\in\Cartan$ defined by
\begin{equation}\label{bases}
	P_1=\begin{pmatrix} 1 & 0 \\ 0 & 1 \end{pmatrix},\quad P_2=\begin{pmatrix} 0 & \sqrt{2} \\ 0 & 0 \end{pmatrix}, \quad P_3=\begin{pmatrix} -1 & 0 \\ 0 & 1 \end{pmatrix}, \quad P_4=\begin{pmatrix} 0 & 0 \\ \sqrt{2} & 0 \end{pmatrix},
\end{equation}
then
$$\pi_s:(A,B)\in\SL\times\SL\longmapsto\{A P_j B^{-1}\}_{j=1,2,3,4}\in\Cartan\,,$$
is a \emph{$2:1$ spin-covering map} such that
$$\pi_s^*(\omega)=\begin{pmatrix} 0 & \frac{1}{\sqrt{2}}\left(\alpha_1^2-\beta_1^2\right) & -\alpha_1^1+\beta_1^1 & \frac{1}{\sqrt{2}}\left(\alpha_2^1-\beta_2^1\right) \\ \frac{1}{\sqrt{2}}\left(\alpha_2^1-\beta_2^1\right) & \alpha_1^1+\beta_1^1 & \frac{1}{\sqrt{2}}\left(\alpha^1_2+\beta^1_2\right) & 0 \\ -\alpha_1^1+\beta_1^1 & \frac{1}{\sqrt{2}}\left(\alpha_1^2+\beta_1^2\right) & 0 & -\frac{1}{\sqrt{2}}\left(\alpha_2^1+\beta_2^1\right) \\ \frac{1}{\sqrt{2}}\left(\alpha_1^2-\beta_1^2\right) & 0 & -\frac{1}{\sqrt{2}}\left(\alpha^2_1+\beta^2_1\right) & -\left(\alpha_1^1+\beta_1^1\right) \end{pmatrix},$$
where $\omega$ is the matrix-valued $1$-form \eqref{omega} and $(\alpha=A^{-1}dA,\beta=B^{-1}dB)$ is the $\mathfrak{sl}(2,\R)\times\mathfrak{sl}(2,\R)$-valued Maurer-Cartan form. Here, $\mathfrak{sl}(2,\R)$ represents the Lie algebra of $\SL$.

\begin{remark} \emph{For every $j=1,...,4$, differentiating $A P_j B^{-1}$ and evaluating it at the identity element, we obtain
		$$\alpha\, P_j-P_j \, \beta=\pi_s^*(\omega_j^i)P_i\,.$$
Then, expanding the computation we conclude with the expression of $\pi_s^*(\omega)$.}
\end{remark}


The restriction of the inner product $\langle\cdot,\cdot\rangle$ of $\R^{2,2}$ to the special linear group $\SL$ gives a Lorentzian metric of constant sectional curvature $-1$. The special linear group $\SL$ endowed with the Lorentzian metric induced by $\langle\cdot,\cdot\rangle$ is a model for the \emph{anti-de Sitter $3$-space} ($\AdS$).

\begin{remark}\label{H} \emph{Throughout this paper, it will be implicitly assumed that the model for $\AdS$ is the special linear group $\SL$. Topologically, $\AdS\cong\mathbb{D}^2\times\mathbb{S}^1$, which in turn can be identified with the open solid torus in $\mathbb{R}^3$ swept out by the rotation of the (open) unit disc of the $Oxz$-plane centered at $(2,0,0)$ around the $Oz$-axis, known as the torical model for $\AdS$ (see Figure \ref{torical}). This identification will be used throughout the paper to visualize the geometric properties of null curves.}
\end{remark}

\begin{figure}[h]
	\begin{center}
		\makebox[\textwidth][c]{
			\includegraphics[height=5cm,width=5cm]{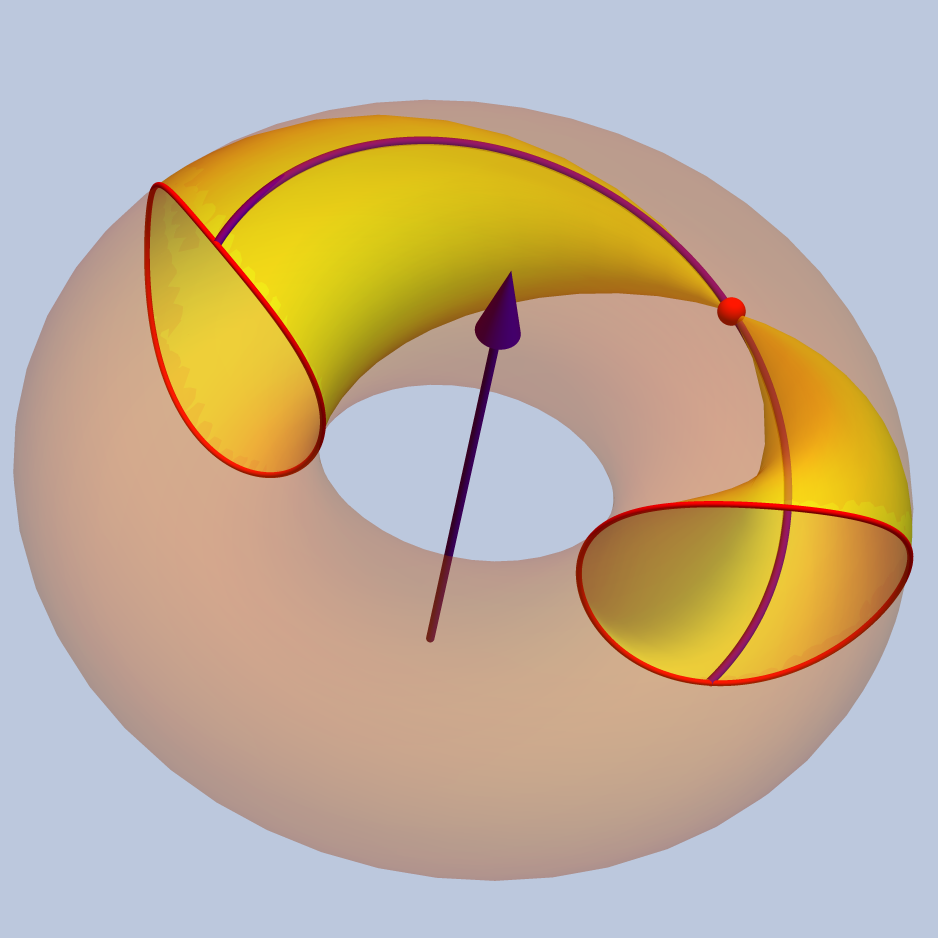}
		}
		\caption{\small{The torical model for $\AdS$ together with a light cone (in yellow) and a null geodesic (in purple).  At infinity, the light cone intersects the ideal boundary into two space-like curves (in red).}} \label{torical}
	\end{center}
\end{figure}

Consider the normal vector field $Q:X\in\AdS\longmapsto-2X\in\R^{2,2}$ and the orientation in $\AdS$ determined by the volume form $i_Q\Omega$. Given a null tangent vector $T\in T_X\AdS$, the bivector $X\wedge T$ is of type $(-,0)$. We define a time-orientation on $\AdS$ by declaring $T$ to be future-directed if $X\wedge T\in\mathcal{N}_+$. The group $\Aut$ acts transitively and effectively on the left of $\SL$ and can be viewed as the isometry group of $\AdS$. Then, $\A$ is the restricted isometry group consisting of all causal isometries of $\AdS$ that also preserve the orientation.

\begin{remark}\label{CE} \emph{The Minkowski, the de Sitter, and the anti-de Sitter $(n+1)$-dimensional spaces can be conformally embedded as open sets of the $(n+1)$-dimensional Einstein universe, that is $\mathbb{S}^1\times \mathbb{S}^n$ equipped with the Lorentzian metric $-ds^2_{\mathbb{S}^1}+ds^2_{\mathbb{S}^n}$. One of the differences of the anti-de Sitter space with respect to the other Lorentzian space forms is that, when considered as an open domain of $\mathbb{S}^1\times \mathbb{S}^n$, its boundary is smooth, diffeomorphic to the Cartesian product of $\mathbb{S}^1$ with an $(n-1)$-dimensional disc and the restriction of the Einstein pseudo-metric induces a conformally flat Lorentzian metric on the boundary. In the case of the Minkowski space, the boundary is singular and the restriction of the Einstein metric induces on the smooth locus, a degenerate quadratic form. The boundary of the de Sitter space is smooth and the restriction of the Einstein metric induces a conformally flat positive definite quadratic form (see for instance \cite{DMN}). This specificity of the anti-de Sitter space has been used in theoretical physics to establish a correspondence (known as the Maldacena correspondence) between string theory in the anti-de Sitter space and conformal field theory in its ideal boundary (\cite{Ma,Wi}).}
\end{remark}

\section{Geometry of Null Curves}

In this section we will collect the basic geometric properties of null curves in $\AdS$ and, employing the spinor frame field, relate them with a suitable pair of star-shaped curves in the centro-affine plane (Subsection 3.1). This relation will be illustrated for the case of null curves with constant bending (Subsection 3.2 and Appendix A).

Let $J\subseteq\R$ be an open interval. A smooth immersed curve $\gamma:J\subseteq \mathbb{R}\longrightarrow \AdS$ is \emph{null} if its velocity vector $\gamma'(s)$ is a null (or, light-like) vector for each $s\in J$. In other words, if $\langle \gamma'(s),\gamma'(s)\rangle=0$ for all $s\in J$. A null curve is \emph{future-directed} if the bivector $\gamma\wedge\gamma'$ of type $(-,0)$ is positive, that is, if $\gamma\wedge\gamma'\in\mathcal{N}_+$.

Let $\gamma:J\subseteq\R\longrightarrow\AdS$ be a future-directed null curve without inflection points (ie. such that $\gamma'\wedge\gamma''\neq 0$ holds). Since $\gamma'\wedge\gamma''\neq 0$, then $\gamma''$ is a space-like vector. We say that $\gamma$ is parameterized by the \emph{proper time}\footnote{The term ``proper time" is usually employed only for time-like curves. However, for convenience, we will use it throughout this paper for null curves.} if $\langle \gamma''(s),\gamma''(s)\rangle=4$ for every $s\in J$. The proper time $s$ is defined up to an additive constant. 

Assume that $\gamma$ is parameterized by the proper time and define the \emph{bending}\footnote{This function is sometimes called the \emph{curvature} (or, \emph{torsion}) of the null curve. Here, we have opted to introduce a different name to distinguish it from the curvature of star-shaped curves in the centro-affine plane that will be used later on.} of $\gamma$ as the function $\bending:J\subseteq\R\longrightarrow\R$ given by
$$\bending(s)=-\frac{1}{16}\langle\gamma'''(s),\gamma'''(s)\rangle\,.$$
It is clear that two null curves $\gamma$ and $\widetilde{\gamma}$ are \emph{equivalent to each other} if and only if $\bending(s)=\widetilde{\bending}(s+c)$ for some constant $c\in\R$.

We define the \emph{tangent} $T$, the \emph{normal} $N$, and the \emph{binormal}\footnote{Although we preserve the names, the vector fields $N(s)$ and $B(s)$ are not the (Frenet) normal and (Frenet) binormal defined for space-like and time-like curves. They are just vector fields defined along a null curve by the above equations.} $B$ vector fields along $\gamma$ by
\begin{eqnarray*}
	T(s)&=&\frac{1}{\sqrt{2}}\,\gamma'(s)\,,\\
	N(s)&=&\frac{1}{2}\,\gamma''(s)\,,\\
	B(s)&=&\frac{1}{\sqrt{2}}\,\bending(s)\,\gamma'(s)-\frac{1}{2\sqrt{2}}\,\gamma'''(s)\,.
\end{eqnarray*}
It then follows that $\{T,N,B\}$ is a moving frame along $\gamma$ and, for every $s\in J$, $\mathcal{F}(s)=\{\gamma(s),T(s),N(s),B(s)\}$ is a future-directed Cartan basis of $\mathbb{R}^{2,2}$ satisfying ${\rm det}(\mathcal{F}(s))=\pm 1$.

\begin{remark}\label{assumptions} \emph{From now on we assume that all our curves $\gamma:J\subseteq\R\longrightarrow\AdS$ are null, future-directed, parameterized by their proper time $s\in J$, and have no inflection points. Moreover, possibly acting on $\gamma$ with an orientation-reversing causal isometry of $\R^{2,2}$, we may assume that $\mathcal{F}(s)$ is positively oriented for every $s\in J$ (ie. ${\rm det}(\mathcal{F}(s))=1$ for every $s\in J$).}
\end{remark}

The map $\mathcal{F}: J\subseteq\R\longrightarrow\Cartan$ is said the \emph{Cartan frame field} along $\gamma$ and it satisfies the following \emph{Frenet-type equations}
\begin{equation}\label{dF}
	d\mathcal{F}=\mathcal{F}\,\mathcal{K}\,ds\,,
\end{equation}
where $\mathcal{K}$ is the matrix given by
\begin{equation}\label{K}
	\mathcal{K}=\begin{pmatrix} 0 & 0 & 0 & \sqrt{2} \\ \sqrt{2} & 0 & \sqrt{2}\,\bending & 0 \\ 0 & \sqrt{2} & 0 & -\sqrt{2}\,\bending \\ 0 & 0 & -\sqrt{2} & 0 \end{pmatrix}.
\end{equation}


Let $\pi_s:\SL\times\SL\longrightarrow\Cartan$ be the $2:1$ spin-covering map introduced in Section 2. Then, the \emph{spinor frame field} along $\gamma$ is a lift $(F_+,F_-)$ of the Cartan frame field $\mathcal{F}$ to $\SL\times\SL$ via $\pi_s$. The spinor frame field is defined up to a sign.

Consider the Cartan basis $\{P_1,P_2,P_3,P_4\}\in\Cartan$ given by \eqref{bases}. Then, the Cartan frame field $\mathcal{F}$ along $\gamma$ is given by
$$\gamma=F_+ F_-^{-1},\quad T=F_+ P_2 F_-^{-1},\quad N=F_+ P_3 F_-^{-1},\quad B=F_+ P_4 F_-^{-1}.$$
Consequently, the spinor frame field $(F_+,F_-)$ along $\gamma$ satisfies the linear systems
\begin{eqnarray}
	dF_+&=&F_+ \begin{pmatrix} 0 & \bending+1 \\ 1 & 0 \end{pmatrix}ds\,,\label{dF+}\\
	dF_-&=&F_- \begin{pmatrix} 0 & \bending-1 \\ 1 & 0 \end{pmatrix}ds\,,\label{dF-}
\end{eqnarray}
where $\bending$ is the bending of $\gamma$. These equations are the spinorial counterpart of the Frenet-type equations \eqref{dF} and, hence, we will refer to them as the \emph{spinorial Frenet-type equations} of $\gamma$.

The spin-covering map $\pi_s$ provides us with the ideal approach to relate null curves in $\AdS$ with a suitable pair of star-shaped curves in the centro-affine plane. 

\subsection{Relation to Star-Shaped Curves}

Let $\dot{\R}^2=\R^2\setminus\{(0,0)\}$ be the centro-affine plane (ie. the once-punctured Euclidean plane $\R^2$). 

A smooth immersed curve $\eta:J\subseteq\R\longrightarrow\dot{\R}^2$ is \emph{star-shaped} if $\eta\wedge\eta'\neq 0$. Any star-shaped curve $\eta$ can be parameterized by the \emph{central affine arc length}, which is defined so that ${\rm det}(\eta,\eta')=1$. The function $\curvature=-{\rm det}(\eta',\eta'')$ is the \emph{central affine curvature}\footnote{Our definition of the central affine curvature coincides with that of Terng-Wu (\cite{TW1}) and it has the opposite sign of that of Pinkall (\cite{P}).} of $\eta$. The map $F=(\eta,\eta'):J\subseteq\R\longrightarrow\SL$ is the \emph{canonical central affine frame field} along $\eta$. Differentiating ${\rm det}(\eta,\eta')=1$, we get that the canonical central affine frame field $F$ satisfies the Frenet-type equations,
\begin{equation}\label{canonicaldF}
	dF=F \begin{pmatrix} 0 & \curvature \\ 1 & 0 \end{pmatrix}ds\,.
\end{equation}

\begin{defn} Let $\eta$ and $\overline{\eta}$ be two star-shaped curves parameterized by the central affine arc length. The pair $(\eta,\overline{\eta})$ is a \emph{pair of cousins} if the central affine curvatures $\curvature$ and $\overline{\curvature}$ of $\eta$ and $\overline{\eta}$, respectively, are related by $(\curvature-\overline{\curvature})/2=1$.
\end{defn} 

We will next explain the relation between pairs of cousins of star-shaped curves in $\dot{\R}^2$ and null curves in $\AdS$.

Let $\gamma:J\subseteq\R\longrightarrow\AdS$ be a null curve and $(F_+,F_-)$ be the spinor frame field along $\gamma$. Denote by $\eta_+$ and $\eta_-$, respectively, the first column vectors of the components of $(F_+,F_-)$. It follows from \eqref{dF+} and \eqref{dF-} that $\eta_\pm:J\subseteq\R\longrightarrow\dot{\R}^2$ are two star-shaped curves parameterized by the central affine arc length, with central affine curvatures $\curvature_+=\bending+1$ and $\curvature_-=\bending-1$, respectively. Consequently, $(\eta_+,\eta_-)$ is a pair of cousins. 

Conversely, given a pair of star-shaped cousins $(\eta,\overline{\eta})$ with canonical central affine frame fields $F$ and $\overline{F}$, respectively, the curve 
$$\gamma=F\overline{F}^{-1}$$ 
is a null curve in $\AdS$ parameterized by the proper time and without inflection points such that the bending of $\gamma$ is given by $\bending=(\curvature+\overline{\curvature})/2$, where $\curvature$ and $\overline{\curvature}$ are the central affine curvatures of $\eta$ and $\overline{\eta}$, respectively. In addition, $(F,\overline{F})$ is the spinor frame field along $\gamma$.

We summarize this characterization in the following result.

\begin{thm}\label{relation} Let $\gamma:J\subseteq\R\longrightarrow\AdS$ be a null curve with bending $\bending$ and spinor frame field $(F_+,F_-)$ along $\gamma$. Then, the first column vectors of $F_\pm$ form a pair of star-shaped cousins $(\eta_+,\eta_-)$ in $\dot{\R}^2$ with central affine curvatures
	$$\curvature_+=\bending+1\,,\quad\quad\quad\curvature_-=\bending-1\,,$$
respectively.

Conversely, let $(\eta_+,\eta_-)$ be a pair of star-shaped cousins with central affine curvatures $\curvature_+$ and $\curvature_-$, and canonical central affine frame fields $F_+$ and $F_-$, respectively. Then,
	$$\gamma=F_+F_-^{-1}\,,$$
is a null curve in $\AdS$ with bending $\bending=(\curvature_++\curvature_-)/2$ and spinor frame field $(F_+,F_-)$ along it. 
\end{thm}

\begin{defn} The pair of star-shaped cousins $(\eta_+,\eta_-)$ is called the \emph{pair of cousins associated with the null curve $\gamma$}.
\end{defn} 

\subsection{The Orbit-Type of a Null Curve with Periodic Bending}\label{OrbitType}

Let $\gamma : J=\R\longrightarrow \AdS$ be a null curve with nonconstant  periodic bending $\bending$ of lest period $\rho>0$ and spinor frame field $(F_+,F-)$. The matrix 
$$M=(M_+,M_-)=\left(F_+(\rho)F_+(0)^{-1}, F_-(\rho)F_-(0)^{-1}\right),$$ 
is the monodromy of $\gamma$.  

Consider the action of $\SL$ on itself by the inner automorphisms ${\rm Int}_A : X\longmapsto AXA^{-1}$. There are two fixed points, namely, $\pm {\rm Id}$, and three types of orbits: parabolic, elliptic, and hyperbolic. We say that $X$ is parabolic if it is not diagonalizable over ${\mathbb C}$, hyperbolic if it is diagonalizable over $\R$ with eigenvalues $r$ and $1/r$, $r\neq \pm 1$, and elliptic if it is diagonalizable over ${\mathbb C}$, with eigenvalues $e^{\pm i\theta}$, $\theta \in (0,\pi)$. Let ${\mathtt P}(2,\R)$ be the closed set of the parabolic elements and ${\mathtt E}(2,\R)$ and ${\mathtt H}(2,\R)$ be the open sets of the elliptic and hyperbolic elements, respectively.  

Null curves with periodic bending can be classified by the type of orbit of their monodromies (there are 5 types). We say that the type of orbit is $(E,E)$ if $F_{\pm}\in {\mathtt E}(2,\R)$, $(H,H)$ if $F_{\pm}\in {\mathtt H}(2,\R)$ and analogously for the other cases. The trajectory of $\gamma$ is invariant by the group generated by the monodromy (the monodromy group of $\gamma$). The curve is closed if and only if its monodromy group is finite. This happens if and only if $M_{\pm}$ are diagonalizable over ${\mathbb C}$, with eigenvalues $e^{\pm i2\pi q_{\pm}}$, where $q_{\pm}\in [0,1]\cap {\mathbb Q}$. If this is the case $\gamma$ is periodic with least period ${\rm lcm}(n_+,n_-)\rho/ {\mathtt s}$, where $q_{\pm}=m_{\pm}/n_{\pm}$, ${\rm gcd}(m_{\pm},n_{\pm})=1$ and ${\mathtt s}\in \{1,1/2\}$ is the spin of $\gamma$.

In the following example we will describe and illustrate the relation of Theorem \ref{relation} for the case of null curves with constant bending.

\begin{ex}\label{example} Consider a null curve $\gamma:J=\R\longrightarrow\AdS$ with constant bending $\bending$. Depending on the possible combinations of the pair of cousins associated with $\gamma$ (which, in this case, each of the star-shaped curves has constant central affine curvatures $\curvature_\pm$), there are five possible cases. In order to obtain closed null curves, we will need both canonical central affine frame fields $F_\pm$ to be periodic. This corresponds to the case $\bending<-1$, for which the pair of cousins is composed by two ellipses. For visualization purposes, we will restrict here to this case (for the other cases see Appendix A). Since $\bending<-1$ is constant, the spinorial Frenet-type equations of $\gamma$, \eqref{dF+} and \eqref{dF-}, can be analytically solved obtaining the elliptic $1$-parameter subgroups
	$$F_\pm=\begin{pmatrix} \cos\left(\sqrt{\lvert \bending\pm 1\rvert}\,s\right) & -\sqrt{\lvert\bending\pm 1\rvert}\,\sin\left(\sqrt{\lvert \bending\pm 1\rvert}\,s\right) \\ \frac{1}{\sqrt{\lvert \bending\pm 1\rvert}}\sin\left(\sqrt{\lvert \bending\pm 1\rvert}\,s\right) & \cos\left(\sqrt{\lvert \bending\pm 1\rvert}\,s\right)\end{pmatrix},$$
	respectively. Consequently, the pair of cousins associated with $\gamma$ are given, respectively, by
	$$\eta_\pm(s)=\left(\cos\left(\sqrt{\lvert \bending\pm 1\rvert}\,s\right),\frac{1}{\sqrt{\lvert \bending\pm 1\rvert}}\sin\left(\sqrt{\lvert \bending\pm 1\rvert}\,s\right)\right),$$
	while $\gamma$ itself can be computed as $\gamma=F_+F_-^{-1}$.
	
	The least periods of the components of the spinor frame field $(F_+,F_-)$ are, respectively,
	$$\rho_{\pm}=\frac{2\pi}{\sqrt{\lvert \bending\pm 1\rvert}}\,.$$
	Therefore, the null curve $\gamma$ is closed if and only if the quotient $\rho_+/\rho_-\in\mathbb{Q}$ is a rational number, say $m/n$, where $m>n$ are relatively prime natural numbers. In other words, the least periods $\rho_+$ and $\rho_-$ are commensurable. This implies that the bending of $\gamma\equiv\gamma_{m,n}$ must be
	$$\bending\equiv\bending_{m,n}=-\frac{m^2+n^2}{m^2-n^2}\,.$$
	
	In Figures \ref{F1} and \ref{F2} we show two examples of closed null curves ($\gamma_{7,3}$ and $\gamma_{8,3}$, respectively) with constant bending $\bending_{m,n}<-1$ in the torical model for $\AdS$ as well as their associated pair of star-shaped cousins.
	
	If $\gamma$ is a closed (ie. periodic) curve with least period $\rho$, then there are two possibilities for its periodic spinor frame field $(F_+,F_-)$: either the least period of $(F_+,F_-)$ is $\rho$, ie. $(F_+(0),F_-(0))=(F_+(\rho),F_-(\rho))$; or else the least period of $(F_+,F_-)$ is $2\rho$, ie. $(F_+(0),F_-(0))=(-F_+(\rho),-F_-(\rho))$. In the former case we say that $\gamma$ has spin $1$ and, in the latter, spin $1/2$. Let $\gamma_{m,n}$ be a closed null curve with constant bending $\bending_{m,n}$. If $m+n$ is even, the spin of $\gamma_{m,n}$ is $1/2$ and $\gamma_{m,n}$ is a torus knot of type $((n-m)/2,(n+m)/2)$ (see Figure \ref{F1}), while if $m+n$ is odd, the spin of $\gamma_{m,n}$ is $1$ and $\gamma_{m,n}$ is a torus knot of type $(n-m,n+m)$ (see Figure \ref{F2}).
\end{ex}

\begin{figure}[h]
	\begin{center}
		\makebox[\textwidth][c]{
			\includegraphics[height=4cm,width=4cm]{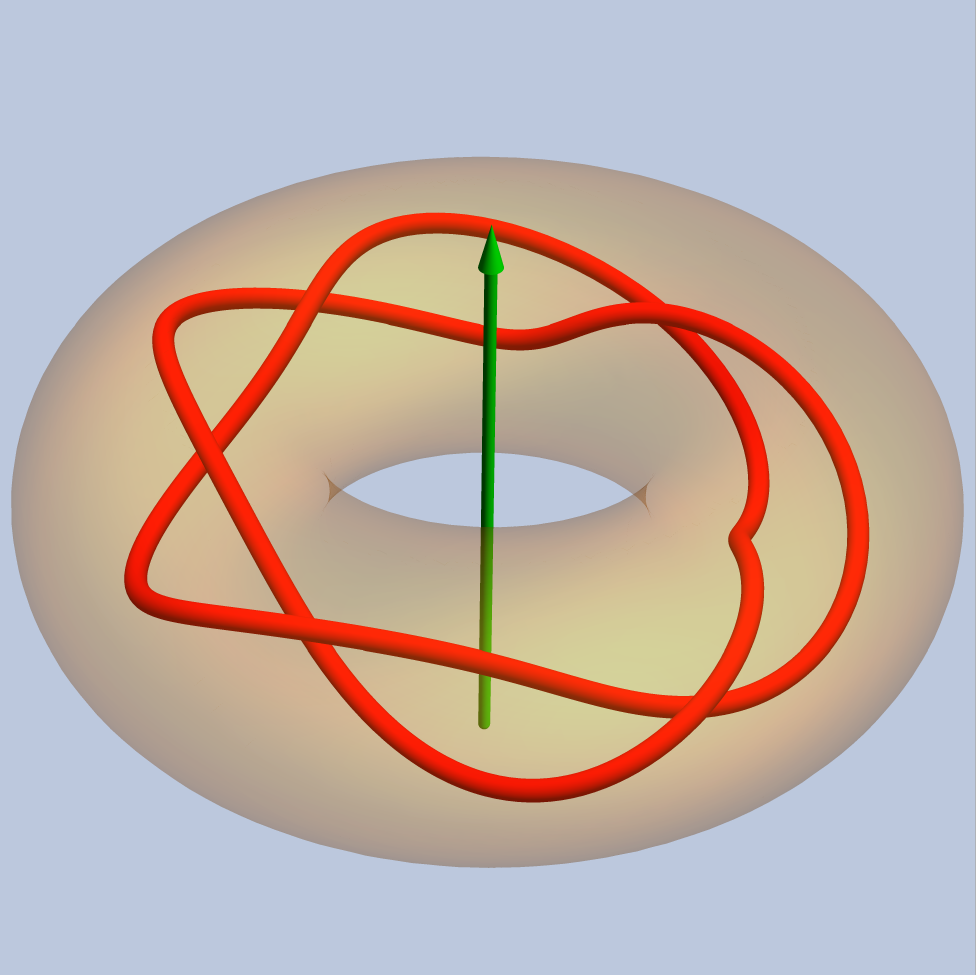}\quad\quad\quad
			\includegraphics[height=4cm,width=4cm]{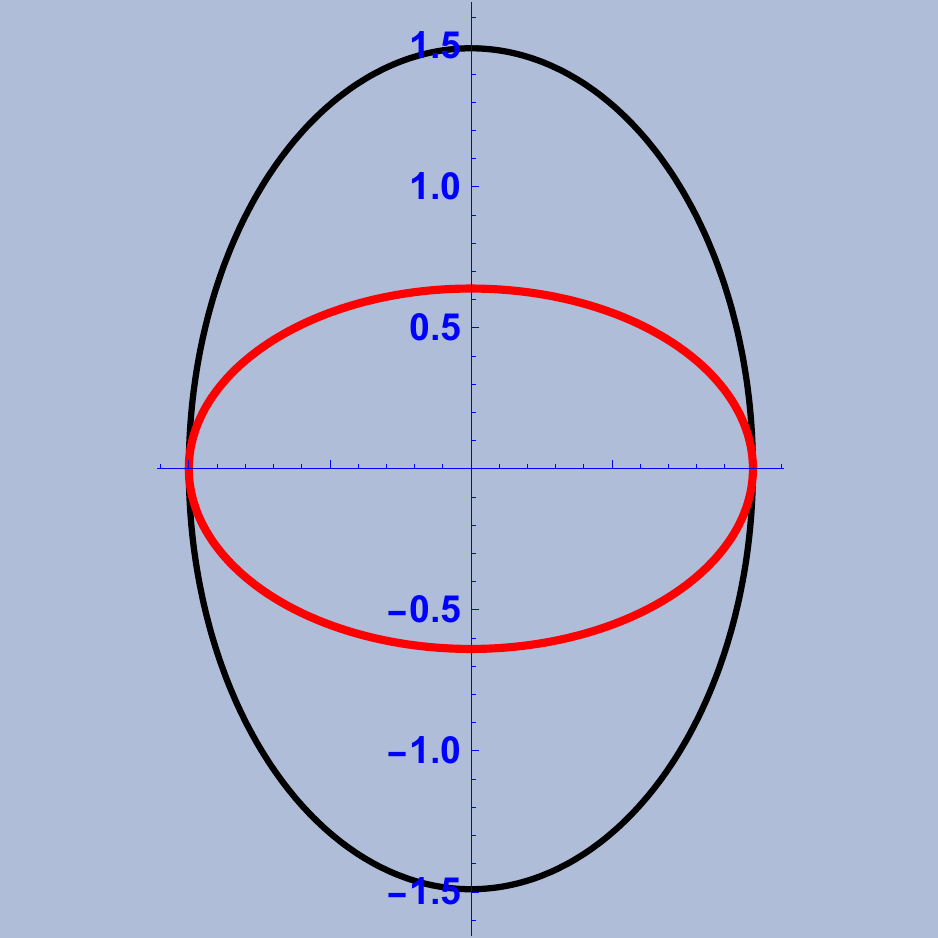}
		}
		\caption{\small{Left: The closed null curve $\gamma_{7,3}$ in the torical model for $\AdS$. Right: The pair of star-shaped cousins associated with $\gamma_{7,3}$.}} \label{F1}
	\end{center}
\end{figure}

\begin{figure}[h]
\begin{center}
		\makebox[\textwidth][c]{
			\includegraphics[height=4cm,width=4cm]{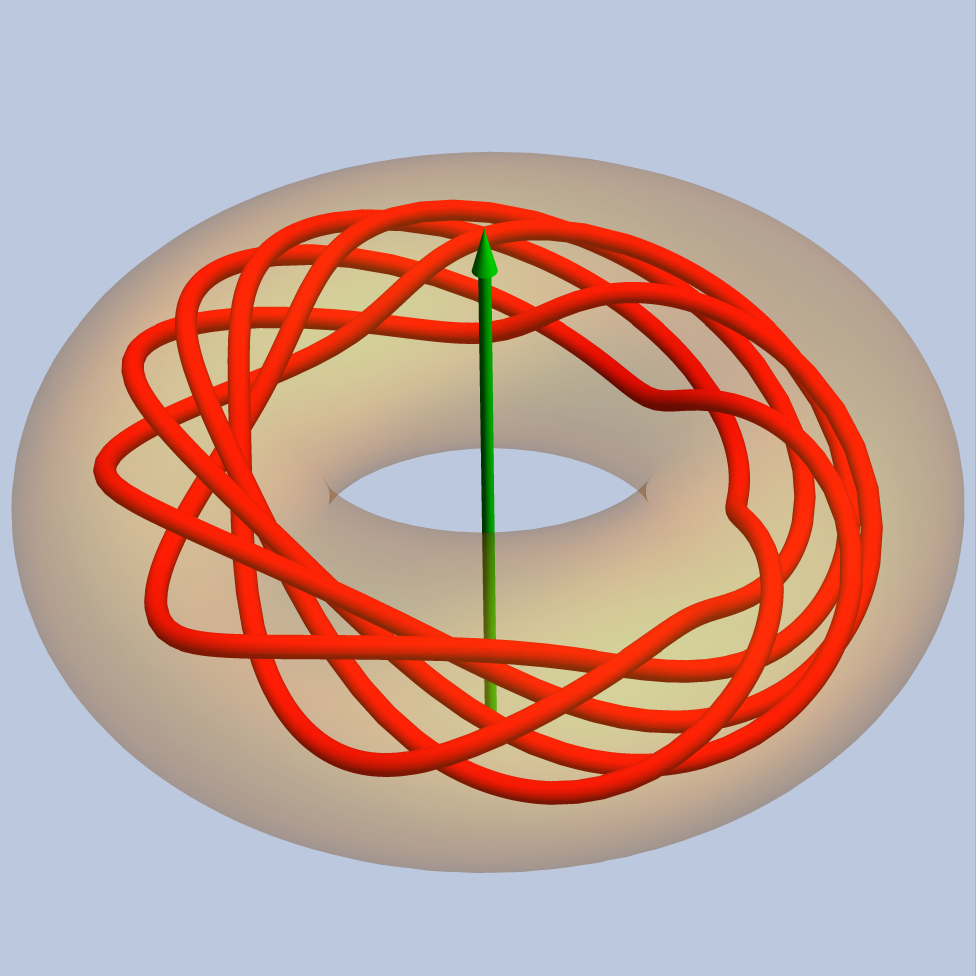}\quad\quad\quad
			\includegraphics[height=4cm,width=4cm]{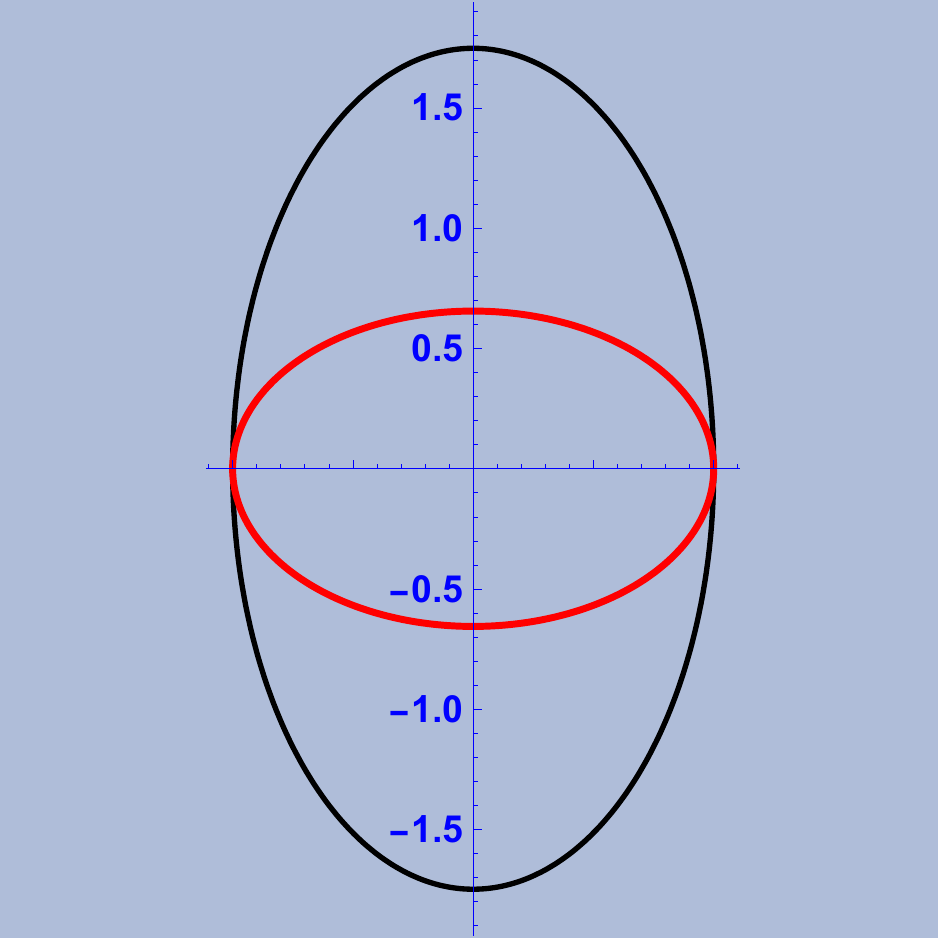}
		}
		\caption{\small{Left: The closed null curve $\gamma_{8,3}$ in the torical model for $\AdS$. Right: The pair of star-shaped cousins associated with 
		$\gamma_{8,3}$.}} \label{F2}
	\end{center}
\end{figure}

\section{Integrable Flows on Null Curves}

In this section we will introduce the LIEN flow on null curves in $\AdS$ and prove that the induced evolution equation on the bending is the KdV equation. Moreover, we will also show how to construct the solutions of the flow beginning with solutions of the KdV equation (Subsection 4.2). For the latter purpose, we will employ the Lax pair formulation of the KdV given by Zakharov and Faddeev (\cite{ZF}), hence, we begin by recalling this formulation (Subsection 4.1).

\subsection{The Lax Pair Formulation of the KdV and Extended Frames}

Given a function $\bending:(s,t)\in U\subseteq\R^2\longmapsto \bending(s,t)\in\R$ defined on a simply-connected open domain $U\subseteq\R^2$ and a constant $\lambda\in\R$, consider the $\mathfrak{sl}(2,\R)$-valued $1$-form\footnote{The $\mathfrak{sl}(2,\R)$-valued $1$-form $\Gamma_\lambda$ can be understood as the Lax connection, while the Maurer-Cartan compatibility equation is the zero-curvature equation for this connection.}
\begin{equation}\label{alpha}
	\Gamma_{\lambda}=\mathcal{K}_\lambda\,ds+\mathcal{P}_\lambda\,dt\,,
\end{equation}
where $\mathcal{K}_\lambda$ and $\mathcal{P}_\lambda$ are given by
\begin{equation}\label{KP}
	\mathcal{K}_\lambda=\begin{pmatrix} 0 & \bending+\lambda \\ 1 & 0\end{pmatrix}, \quad\quad \mathcal{P}_\lambda=\begin{pmatrix} -\partial_s\bending & -\partial_s^2\bending+2\bending^2-2\lambda\bending-4\lambda^2 \\ 2\bending-4\lambda & \partial_s\bending \end{pmatrix}.
\end{equation}
It is then a straightforward computation to check that the $1$-form $\Gamma_{\lambda}$ satisfies the Maurer-Cartan compatibility equation $d\Gamma_{\lambda}+\Gamma_{\lambda}\wedge\Gamma_{\lambda}=0$, or equivalently
$$\partial_t\mathcal{K}_\lambda-\partial_s\mathcal{P}_\lambda-\left[\mathcal{K}_\lambda,\mathcal{P}_\lambda\right]=0\,,$$
if and only if the function $\bending(s,t)$ satisfies the KdV equation
\begin{equation}\label{KdV2}
	\partial_t\bending+\partial_s^3\bending-6\bending\,\partial_s\bending=0\,.
\end{equation}
Consequently, as shown in \cite{ZF}, for every $\lambda\in\R$ there exists a map $E_\lambda:U\subseteq\R^2\longrightarrow\SL$ such that
\begin{equation}\label{dE}
	dE_\lambda=E_\lambda\,\Gamma_{\lambda}\,.
\end{equation}
The maps $E_\lambda$ depend in a real analytic fashion on $\lambda\in\R$ and are uniquely defined up to a left multiplication by an element of $\SL$. The map $E_\lambda$ is called an \emph{extended frame} of $\bending$ with spectral parameter $\lambda\in\R$ (\cite{TU})\footnote{Observe that in the paper \cite{TU}, the spectral parameter $\lambda$ is a complex number, while here we are restricting it to real values.}.

\begin{remark}\label{Pinkall} \emph{Assume that $U=J\times I$, where $J$ and $I$ are open intervals of the real line $\R$, and consider $\lambda=0$. Denote by $\eta$ the first column vector of the extended frame $E_0$. Then, $\eta$ satisfies 
		$$\partial_t\eta=-\partial_s\bending\,\eta+2\bending\,\partial_s\eta\,,$$
where $\bending=\curvature$ is the central affine curvature of $\eta$. Hence, $\eta$ is an integral curve of the Pinkall flows (\cite{P}).}
\end{remark}

\subsection{The LIEN Flow}

Let $J,I\subseteq\R$ be two open intervals such that (for convenience) $0\in I$ and consider $\gamma:(s,t)\in J\times I\subseteq\R^2\longmapsto \gamma(s,t)=\gamma_t(s)\in\AdS$ a smooth one parameter family of null curves without inflection points and parameterized by the proper time. In other words, for each $t\in I$, we have a null curve $\gamma_t:J\subseteq\R\longrightarrow\AdS$ satisfying the assumptions of Remark \ref{assumptions}.

The LIEN flow is the evolution equation for null curves in $\AdS$ given by
\begin{equation}\label{LIEN2}
	\partial_t\gamma=-2\sqrt{2}\,\left(\bending\, T+2 B\right),
\end{equation}
where $\bending=\bending(s,t)$ is the bending of $\gamma(s,t)$ and $\mathcal{F}=\{\gamma,T,N,B\}$ is the Cartan frame field along $\gamma(s,t)$ defined in Section 3. That is, $T(s,t)$ is the tangent vector field along $\gamma(s,t)$, while $B(s,t)$ is the binormal vector field.

We next show that the induced evolution equation on the bending $\bending$ of $\gamma$ is the KdV equation \eqref{KdV2}. In addition, we give a construction procedure to obtain solutions of the LIEN flow \eqref{LIEN2} beginning with solutions of \eqref{KdV2}.

\begin{thm}\label{induced}
	Let $\gamma:J\times I\subseteq\R^2\longrightarrow\AdS$ be a solution of the LIEN flow \eqref{LIEN2}. Then the bending $\bending(s,t)$ of $\gamma(s,t)$ evolves according to the KdV equation \eqref{KdV2}. 
	
	Conversely, if $\bending:J\times I\subseteq\R^2\longrightarrow\R$ is a smooth solution of the KdV equation \eqref{KdV2}, then 
	\begin{equation*}\label{gamma}
		\gamma=E_1E_{-1}^{-1}:J\times I\subseteq\R^2\longrightarrow\AdS\,,
	\end{equation*}
is a solution of the LIEN flow \eqref{LIEN2} with bending $\bending$, where $E_\lambda$, $\lambda=-1,1$, are the extended frames of $\bending$ with spectral parameters $\lambda=-1,1$, respectively.
\end{thm}
\begin{proof}
Suppose that $\gamma$ is a solution of the LIEN flow. Then, from \eqref{LIEN2} and the Frenet-type equations \eqref{dF}, the Cartan frame field $\mathcal{F}$ along $\gamma$ is a solution of the linear system 
\begin{equation}\label{LS}\mathcal{F}^{-1}d\mathcal{F}=\mathcal{K}\,ds+\mathcal{P}\,dt\,,\end{equation}
where $\mathcal{K}$ is as in \eqref{K} and 
\begin{equation}\label{P}
	{\mathcal P}=\begin{pmatrix} 0 &  -4\sqrt{2}  & 0 &	-2\sqrt{2}\,\kappa \\ 
			-2\sqrt{2}\,\kappa & p_{22} &p_{23} & 0 \\ 0 &p_{32} & 0 & -p_{23} \\ -4\sqrt{2} & 0 & -p_{32}& -p_{22} \end{pmatrix}. 
\end{equation}
Notice that the first column and first row of the matrix $\mathcal{P}$ are a consequence of the evolution equation \eqref{LIEN2}, while the unspecified entries $p_{ij}$ are completely determined by the compatibility equations
\begin{equation}\label{compatibility}
	\partial_t\mathcal{K}-\partial_s\mathcal{P}-\left[\mathcal{K},\mathcal{P}\right]=0\,.
\end{equation}
As a result we obtain the specific values
\begin{equation}\label{coef}
	p_{22}=-2\partial_s\kappa\,,\quad\quad\quad p_{32}=\frac{4}{\sqrt{2}}\,\kappa\,,\quad\quad\quad p_{23}=\frac{1}{\sqrt{2}}\left(-2\partial_{s}^2\kappa+4\kappa^2-8\right).
\end{equation}
Let $p_{22},p_{32}$ and $p_{23}$ be as above and $\mathcal{P}$ be the matrix \eqref{P}, then the compatibility equation \eqref{compatibility} is satisfied if and only if $\kappa$ is a solution of the KdV equation \eqref{KdV2}. 

Conversely, let $\bending$ be a solution of the KdV equation \eqref{KdV2}. According to \cite{ZF}, the extended frames of $\bending$ with spectral parameters $\lambda=-1$ and $\lambda=1$, $E_{-1}$ and $E_1$, respectively, do exist and they satisfy \eqref{dE}. We then define the map $\gamma=E_1 E_{-1}^{-1}$. 

For every fixed $t\in I$, it follows from \eqref{dE} and the spinorial Frenet-type equations \eqref{dF+}-\eqref{dF-} that $\gamma_t:J\subseteq\R\longrightarrow\AdS$ is a null curve without inflection points, parameterized by the proper time and with bending $\bending_t$. Moreover, $\left(E_1,E_{-1}\right)$ is the spinor frame field along $\gamma_t$. Therefore, the Cartan frame field $\mathcal{F}$ along $\gamma$ is $\mathcal{F}=\pi_s\circ (E_1,E_{-1})$. Then, using the expression of $\pi_s^*(\omega)$ given in Section 2, we deduce that
\begin{eqnarray*}
	\mathcal{F}^*\left(\omega_1^2\right)&=&\sqrt{2}\left(ds-2\bending dt\right),\\
	\mathcal{F}^*\left(\omega_1^3\right)&=&0\,,\\
	\mathcal{F}^*\left(\omega_1^4\right)&=&-4\sqrt{2}\,dt\,.
\end{eqnarray*}
(Observe that the values $\alpha_i^j$ of $\pi_s^*(\omega)$ correspond to the values of $\Gamma_{1}$, while the $\beta_i^j$ are those of $\Gamma_{-1}$.)

It then follows from
$$d\gamma=\mathcal{F}^*(\omega_1^2)T+\mathcal{F}^*(\omega_1^3)N+\mathcal{F}^*(\omega_1^4)B\,,$$
that \eqref{LIEN2} holds and, hence, $\gamma$ is a solution of the LIEN flow.
\end{proof}

\begin{remark} \emph{Let $\bending:J\times I\subseteq\R^2\longrightarrow\R$ be a smooth solution of the KdV equation \eqref{KdV2}. Employing the extended frames $E_\lambda$ and the construction of null curves of Theorem \ref{relation}, we find a relation between solutions of the LIEN flow \eqref{LIEN2} and solutions of the Pinkall flows. Indeed, as explained in Remark \ref{Pinkall}, the first column vector of $E_0$ is an integral curve $\eta$ of the Pinkall flows, while the null curve $\gamma=E_1E_{-1}^{-1}$ is a solution of the LIEN flow \eqref{LIEN2}. The solution $\bending$ of the KdV equation is the central affine curvature of $\eta$ as well as the bending of $\gamma$.}
\end{remark}

Proceeding in analogy with \cite{TW1} we next prove that closed null curves evolve by the LIEN flow \eqref{LIEN2} through closed curves, assuming the extra condition that the corresponding solution of the KdV equation \eqref{KdV2} is periodic.

\begin{prop}\label{CauchyP} Let $\gamma_o: J=\R\longrightarrow\AdS$ be a closed null curve with bending $\bending_o(s)$. Assume that the solution $\bending(s,t)$, $t\in I$, of the KdV equation \eqref{KdV2} with initial condition $\bending(s,0)=\bending_o(s)$ is periodic in $s\in \R$. Then, the solution $\gamma:\R\times I\subseteq\R^2\longrightarrow\AdS$ of the LIEN flow \eqref{LIEN2} with initial condition $\gamma(s,0)=\gamma_o(s)$ is periodic in $s\in \R$.
\end{prop}

\begin{proof} Let $\gamma_o$ be a closed null curve and assume that its periodic bending $\bending_o$ has least period $\rho$. Recall that the Cartan frame field ${\mathcal F}$ along $\gamma$ is a solution of the linear system (\ref{LS}), where $\mathcal{K}$ is as in \eqref{K} and ${\mathcal P}$ is as in  \eqref{P} for the entries given in \eqref{coef}. To prove that the evolving curves $\gamma_t:s\in\R\longmapsto\gamma(s,t)\in\AdS$ are periodic in $s\in\R$ it suffices to show that the monodromy $M: t\in I\longmapsto {\mathcal F}(\rho,t){\mathcal F}(0,t)^{-1}\in\R$ is constant. Indeed, if $M$ is constant and $\gamma_o$ is periodic in $s\in\R$, then $M$ has finite order $n\ge 0$ and $\gamma_{o}$ is periodic with least period $n\rho$.  The evolving curves $\gamma_t(s)=\gamma(s,t)$ have periodic bending of period $\rho$ and monodromies of order $n$. Consequently, $\gamma_t$ is periodic with period $n\rho$, for every $t\in I$. 
	
Differentiating $M$ with respect to $t$ we have
$$\partial_t M(t) = \partial_t{\mathcal F}(\rho,t) {\mathcal F}(0,t)^{-1}+{\mathcal F}(\rho,t)\partial_t{\mathcal F}^{-1}(0,t)\,.$$
From (\ref{LS}) we have
$$\partial_t{\mathcal F}(\rho,t)={\mathcal F}(\rho,t){\mathcal P}(\rho,t)\,,\quad\quad\quad \partial_t{\mathcal F}^{-1}(0,t) =
 -{\mathcal P}(0,t){\mathcal F}^{-1}(0,t)\,.$$
Since $\kappa(s,t)$ is periodic in $s\in\R$, we have ${\mathcal P}(\rho,t)={\mathcal P}(0,t)$ and, hence,
$$\partial_t M(t) =  {\mathcal F}(\rho,t) {\mathcal P}(\rho,t) {\mathcal F}(0,t)^{-1} - 
{\mathcal F}(\rho,t) {\mathcal P}(0,t) {\mathcal F}(0,t)^{-1}=0\,,
$$
proving the result.
\end{proof}

\begin{remark} \emph{The above proposition is a consequence of the fact that the LIEN flow \eqref{LIEN2} preserves the monodromy.}
\end{remark}

\begin{remark} \emph{The Cauchy problem for the KdV equation \eqref{KdV2} with initial data belonging to the Schwartz class ${\mathcal S}(\R,\R)$ was solved by Lax in \cite{L0}. From this it follows that a null curve whose bending $\bending$ belongs to this class, i.e. $\bending \in {\mathcal S}(\R,\R)$, evolves under the LIEN flow \eqref{LIEN2} and the bendings of the evolving curves belong to ${\mathcal S}(\R,\R)$ as well.}
\end{remark}

\section{Stationary Solutions}

Among all the possible solutions of the LIEN flow \eqref{LIEN2}, stationary solutions deserve special attention due to their interesting geometric meaning. From a geometric point of view the stationary solutions are the fixed points of the flow on the space of the equivalence classes of null curves. In this section we will focus on stationary solutions of the flow and prove that those with nonconstant periodic bending can be explicitly integrated employing the fundamental solutions of a Lam\'e equation (Subsections 5.2 and 5.3). To show this, we will first recall the Floquet eigenvalue problem for this equation (\cite{V}) which plays an essential role in the proof (Subsection 5.1).

\subsection{Floquet Spectrum for the First Order Lam\'e Equation}

The first order \emph{Lam\'e equation} is the second order ordinary differential equation
\begin{equation}\label{Lame}
	f''+\left(h-2\,\mu\,{\rm sn}^2(-,\mu)\right)f=0\,,
\end{equation}
where ${\rm sn}(-,\mu)$ is the Jacobi's ${\rm sn}$ function with \emph{elliptic parameter}\footnote{Our elliptic parameter $\mu\in(0,1)$ is the square of the elliptic modulus.} $\mu\in(0,1)$ and $h\in\R$ is a constant called the \emph{eigenvalue parameter}. Observe that ${\rm sn}^2(-,\mu)$ is an even periodic function with least period $2K(\mu)$. Here, $K$ denotes the complete elliptic integral of the first kind.  We denote by ${\rm cl}_{h,\mu}$ and ${\rm sl}_{h,\mu}$ the fundamental solutions of the equation \eqref{Lame} normalized by ${\rm cl}_{h,\mu}(0)=1$, ${\rm cl}_{h,\mu}'(0)=0$, ${\rm sl}_{h,\mu}(0)=0$, and ${\rm sl}_{h,\mu}'(0)=1$. It is clear that since ${\rm sn}^2(-,\mu)$ is an even function, ${\rm cl}_{h,\mu}$ is even while ${\rm sl}_{h,\mu}$ is odd.  For more details about elliptic integrals and Jacobian elliptic functions, we refer the reader to \cite{RW}.

\begin{defn}
	The eigenvalue parameter $h\in\R$ is a \emph{Floquet eigenvalue} (with elliptic parameter $\mu$) if both fundamental solutions ${\rm cl}_{h,\mu}$ and ${\rm sl}_{h,\mu}$ are periodic. The \emph{Floquet spectrum} with elliptic parameter $\mu$, denoted by $\mathcal{S}_\mu$, is the set of all Floquet eigenvalues with elliptic parameter $\mu$.
\end{defn}

From the Lam\'e equation \eqref{Lame} it follows that the function $\delta_{h,\mu}:\R\longrightarrow\SL$ defined by
\begin{equation}\label{beta}
	\delta_{h,\mu}(s)=\begin{pmatrix} {\rm cl}_{h,\mu}(s) & {\rm cl}_{h,\mu}'(s) \\ {\rm sl}_{h,\mu}(s) & {\rm sl}_{h,\mu}'(s) \end{pmatrix},
\end{equation}
satisfies the Cauchy problem
\begin{equation}\label{Cauchy}
	\delta_{h,\mu}'(s)=\delta_{h,\mu}(s)\begin{pmatrix} 0 & 2\mu\,{\rm sn}^2(s,\mu)-h \\1 & 0 \end{pmatrix},\quad\quad\quad \delta_{h,\mu}(0)={\rm Id}\,.
\end{equation}
The \emph{monodromy} is the matrix $M_\mu(h)=\delta_{h,\mu}(2K(\mu))\in\SL$. From \eqref{Cauchy}, it follows that $h$ is a Floquet eigenvalue if and only if the monodromy has finite order. Equivalently, if and only if the matrix $M_\mu(h)$ is diagonalizable\footnote{Observe that when $q\in(0,1)$ the monodromy $M_\mu(h)$ is always diagonalizable, since in these cases the eigenvalues are different.} over $\mathbb{C}$ with eigenvalues $e^{\pm i\pi q}$ where $q\in[0,1]\cap\mathbb{Q}$. The rational number $q$ is called the \emph{characteristic exponent} of $h$.

The set of all Floquet eigenvalues $h$ with characteristic exponent $q$ is denoted by $\mathcal{S}_{\mu,q}\subset\mathcal{S}_\mu$. The set $\mathcal{S}_{\mu,q}$ is an unbounded strictly increasing sequence (\cite{V}). In particular, $\mathcal{S}_{\mu,0}=\{b_1^{2r}(\mu)\}_{r\in\mathbb{N}}$ and $\mathcal{S}_{\mu,1}=\{b_1^{2r+1}(\mu)\}_{r\in\mathbb{N}}$, where $b_1^{2r}(\mu)$ and $b_1^{2r+1}(\mu)$ are the classical \emph{Lam\'e eigenvalues}\footnote{Due to their relevance in physics and applied mathematics, since 2020, the Lam\'e eigenvalues are implemented in the most common scientific computing software, such as \emph{Mathematica}, \emph{MatLab} or \emph{Maple}.} of order one (\cite{V0}). These eigenvalues interlace according to
$$1+\mu<b_1^{2r}(\mu)<b_1^{2r+1}(\mu)\,,$$
and $b_1^{r}(\mu)\to\infty$ as $r\to\infty$. If $q\in(0,1)$, then $\mathcal{S}_{\mu,q}=\{h\in(\mu,1)\cup(\mu+1,\infty)\,\lvert\,\tau_\mu(h)=\cos(q\pi)\}$, where $\tau_\mu$ is the real analytic function of the variable $h$ defined by half the trace of the monodromy, that is,
\begin{equation}\label{tau}
	\tau_\mu(h)=\frac{1}{2}{\rm tr}M_\mu(h)\,.
\end{equation}
In Figure \ref{trace} we illustrate the graph of a function $\tau_\mu(h)$. Assuming the corresponding restrictions, namely, $h>1+\mu$ when $q=0,1$, or $h\in(\mu,1)\cup (1+\mu,\infty)$ when $q\in(0,1)$, the values of $h$ at which $\tau_\mu(h)=\cos(q\pi)$, $q\in[0,1]$, holds are the Floquet eigenvalues with characteristic exponent $q$. One can also see from the figure that the sequence of these eigenvalues is strictly increasing and unbounded.

\begin{figure}[h]
	\begin{center}
		\makebox[\textwidth][c]{
			\includegraphics[height=4cm,width=6cm]{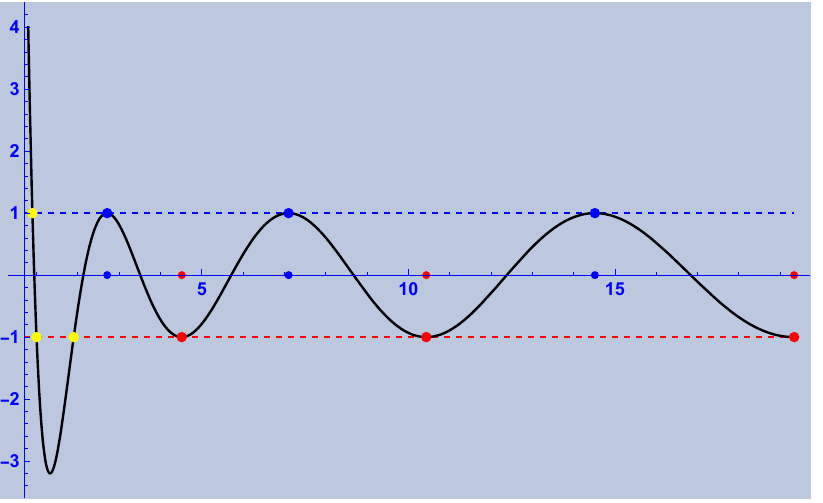}\quad\quad
			\includegraphics[height=4cm,width=6cm]{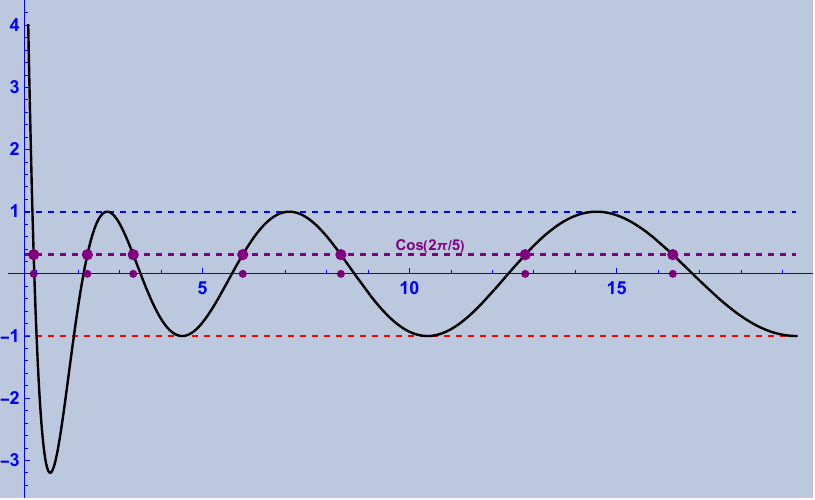}
		}
		\caption{\small{Graph of the function $\tau_\mu(h)$, \eqref{tau}, for the specific value $\mu=0.9$ (in black). Left: The first three Floquet eigenvalues with characteristic exponent $q=0$, that is, belonging to $\mathcal{S}_{\mu,0}\subset\mathcal{S}_\mu$ (the red points) and the first three Floquet eigenvalues with characteristic exponent $q=1$, which belong to $\mathcal{S}_{\mu,1}\subset\mathcal{S}_\mu$ (the blue points). Although the yellow points satisfy $\tau_\mu(h)=\cos(q\pi)$, $q=0,1$, their $h$ value does not belong to $\mathcal{S}_{\mu,q}$ since $h<1+\mu$. Right: The first seven Floquet eigenvalues with characteristic exponent $q=5/2$, that is, the first seven elements of $\mathcal{S}_{\mu,2/5}\subset\mathcal{S}_\mu$ (the purple points).}} \label{trace}
	\end{center}
\end{figure}

\subsection{Heun Functions and Periodic Solutions of the Lam\'e Equation}

In this subsection we will show how to build the fundamental solutions ${\rm cl}_{h,\mu}$ and ${\rm sl}_{h,\mu}$ of the Lam\'e equation \eqref{Lame} out of the Jacobi's ${\rm sn}$ function and a class of holomorphic functions, namely, the local Heun functions. For more details about Heun functions we refer the reader to \cite{Ro} and \cite{SK}.

The (local) \emph{Heun function} ${\mathcal H}\ell(a,q,\alpha,\beta,\gamma,\delta;z)$ with parameters $a,q,\alpha,\beta,\gamma$ and $\delta$ is the holomorphic solution of the second order ordinary differential equation
\begin{equation}\label{Heun}
	f''(z)+\left(\frac{\gamma}{z}+\frac{\delta}{z-1}+\frac{\alpha+\beta-\gamma-\delta+1}{z-a}\right)f'(z)+\frac{\alpha\beta z-q}{z(z-1)(z-a)}f(z)=0\,,
\end{equation}
with initial condition $f(0)=1$. In the present paper, we are interested in a couple of Heun functions determined by the parameters, respectively,
$$a=1/\mu\,,\quad q=(\mu-h)/4\mu\,,\quad \alpha=0\,,\quad  \beta=3/2\,,\quad  \gamma=\delta=1/2\,,$$
and 
$$a=1/\mu\,,\quad q=(1-h+4\mu)/4\mu\,,\quad \alpha=1/2\,,\quad  \beta=2\,,  \quad\gamma=3/2\,,\quad\delta=1/2\,,$$
where $\mu\in (0,1)$ is the elliptic parameter of the Jacobi's ${\rm sn}$ function and $h\in \mathcal{S}_{\mu}$ is a Floquet eigenvalue with elliptic parameter $\mu$. For short, we denote these Heun functions by $\mathcal{H}\ell_1(\mu,h;z)$ and by $\mathcal{H}\ell_2(\mu,h;z)$, respectively (an example is shown in Figures \ref{Heun1} and  \ref{Heun2}). 

\begin{figure}[h]
	\begin{center}
		\makebox[\textwidth][c]{
			\includegraphics[height=6cm,width=6cm]{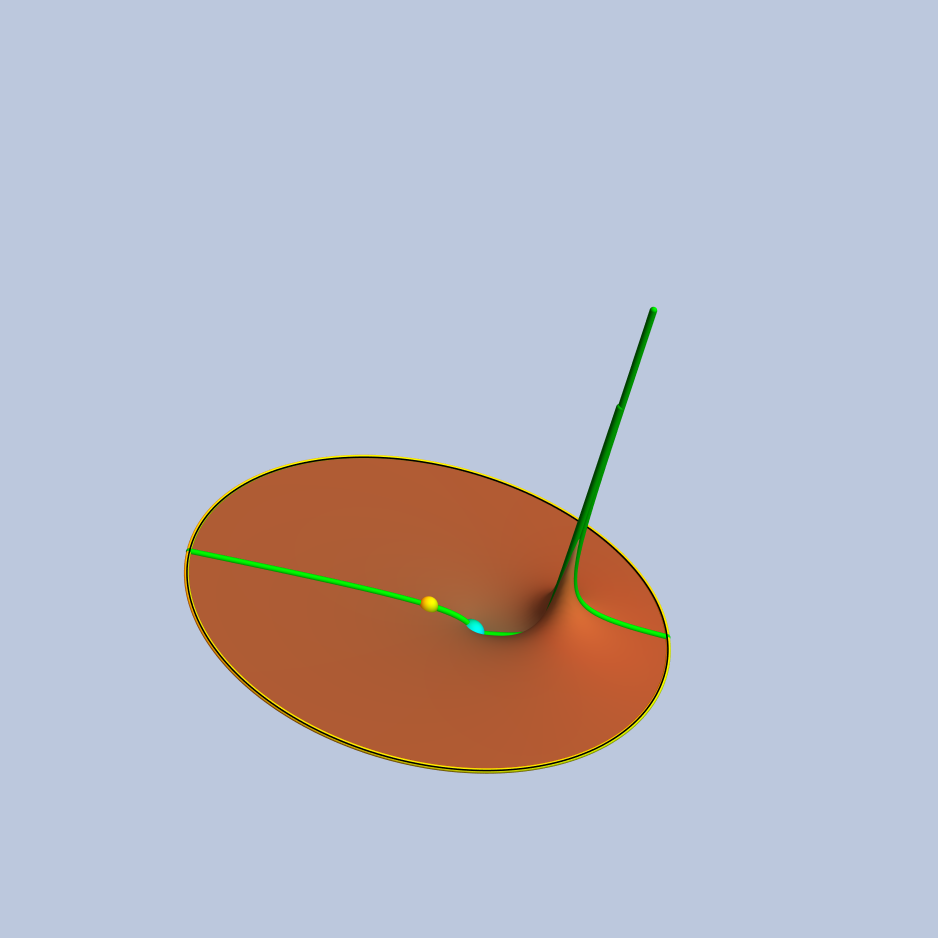}\quad\quad
			\includegraphics[height=6cm,width=6cm]{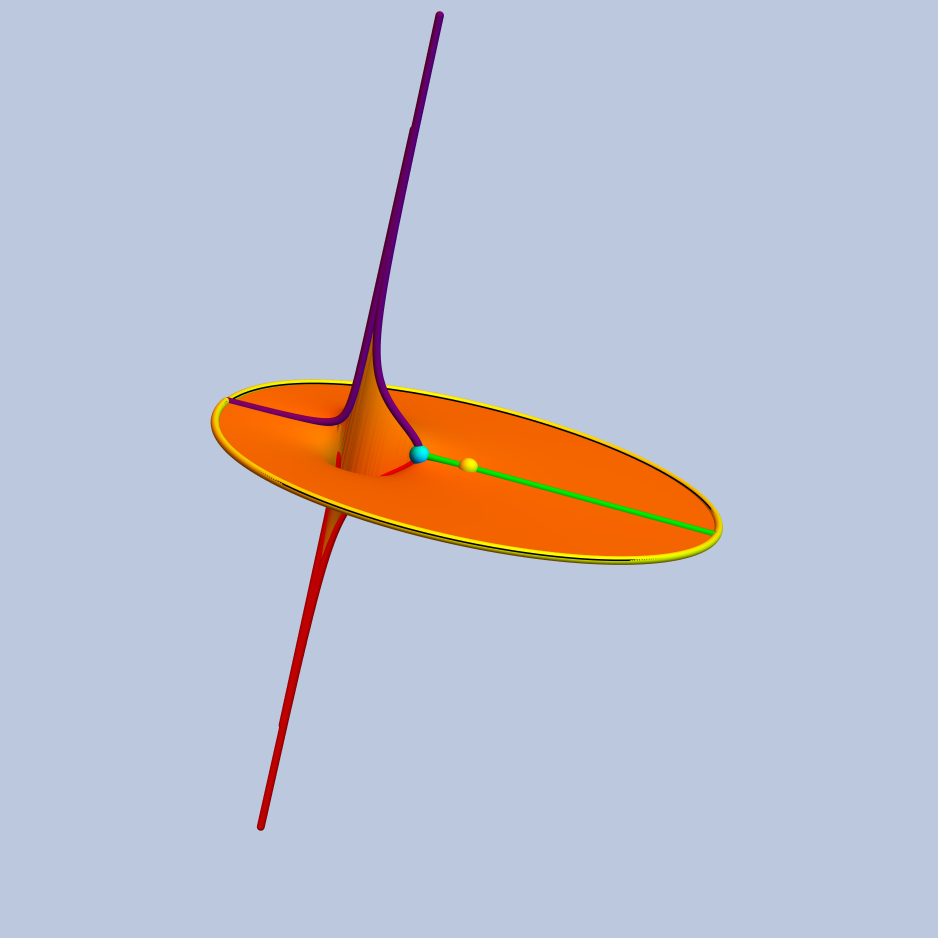}}
		\caption{\small{Graphs of the real and imaginary parts (left and right, respectively) of the Heun function  $\mathcal{H}\ell_1(\mu,h;z)$ over the disc of radius $5$ centered at the origin. In this example, the elliptic parameter is $\mu=0.4$ and the Floquet eigenvalue is $h\simeq 0.67\in\mathcal{S}_{\mu,2/5}\subset\mathcal{S}_\mu$. The yellow points are $(0,\Re(\mathcal{H}\ell_1(\mu,h;0)))$ and $(0,\Im(\mathcal{H}\ell_1(\mu,h;0)))$, while the green points are  $(1,\Re(\mathcal{H}\ell_1(\mu,h;1)))$ and $(1,\Im(\mathcal{H}\ell_1(\mu,h;1)))$.}}\label{Heun1}
	\end{center}
\end{figure}

\begin{figure}[h]
	\begin{center}
		\makebox[\textwidth][c]{
			\includegraphics[height=6cm,width=6cm]{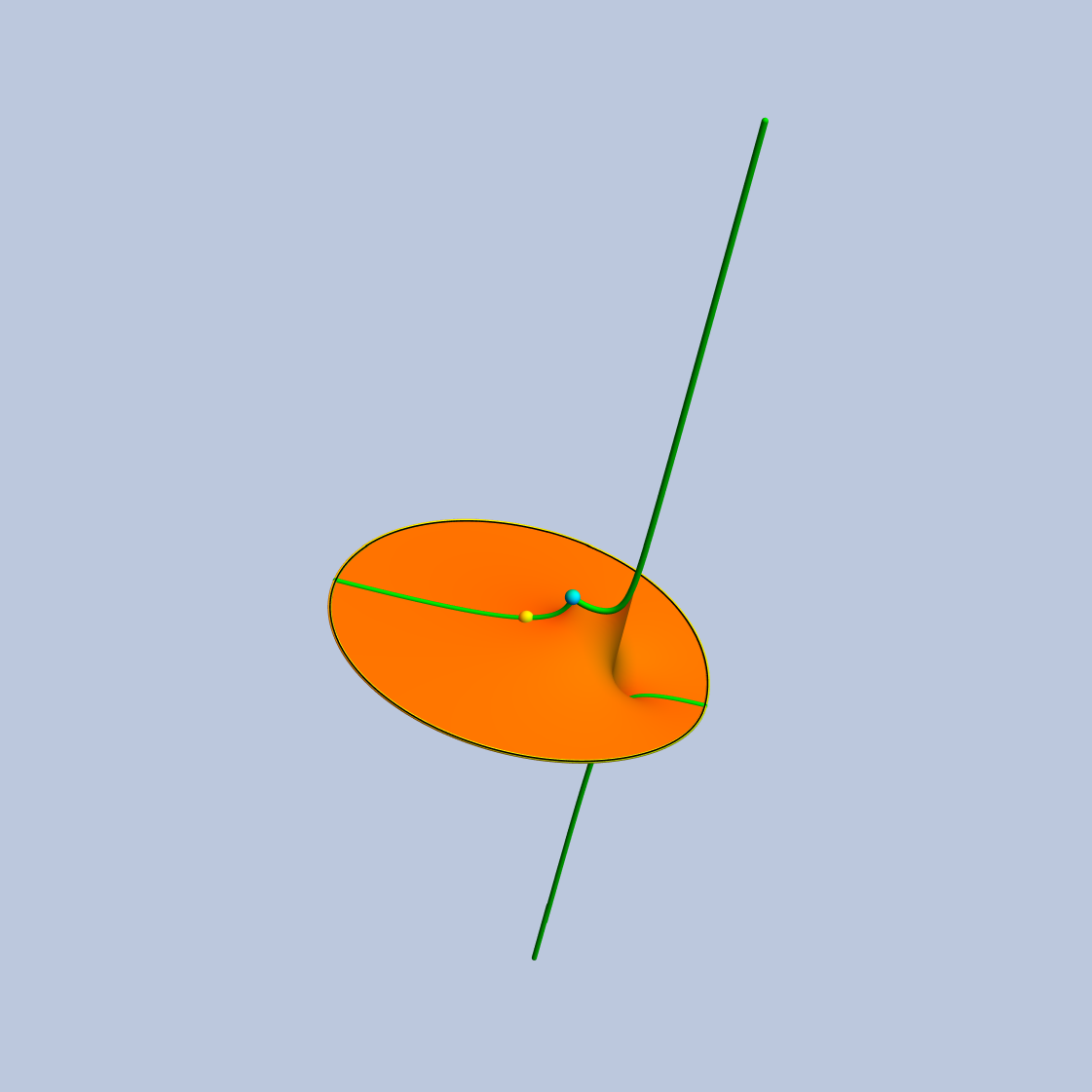}\quad\quad
			\includegraphics[height=6cm,width=6cm]{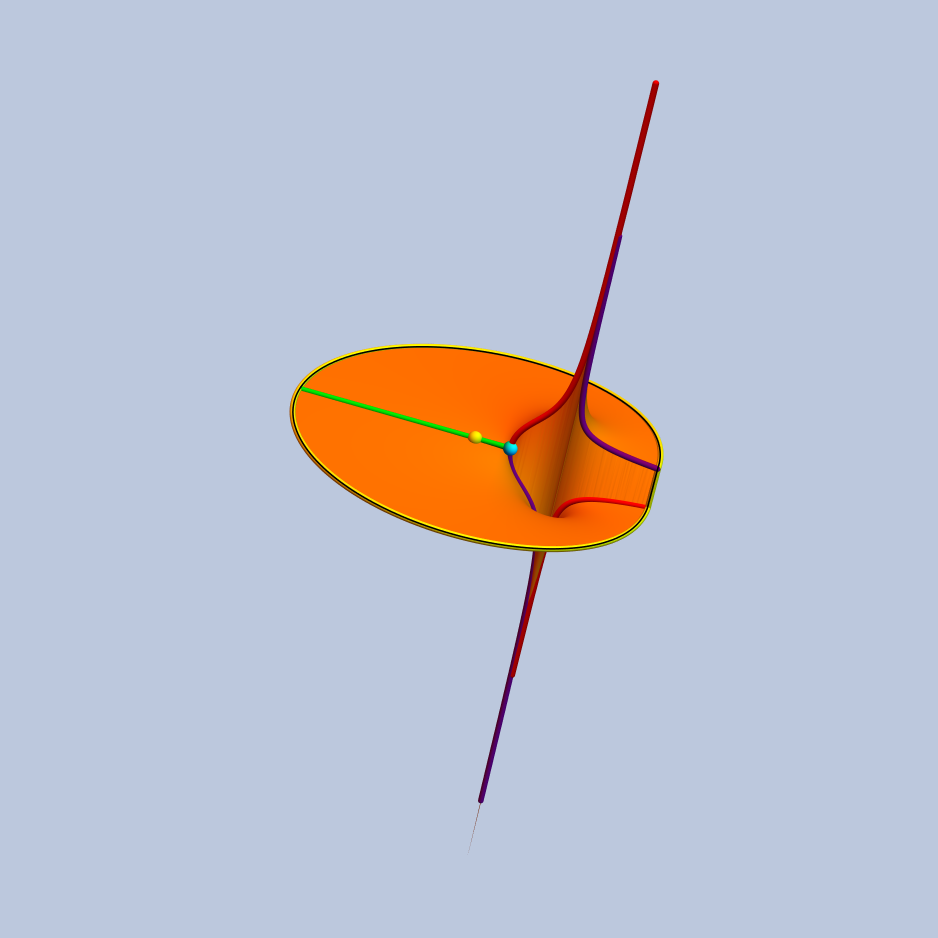}}
		\caption{\small{Graphs of the real and imaginary parts (left and right, respectively) of the Heun function  $\mathcal{H}\ell_2(\mu,h;z)$ over the disc of radius $5$ centered at the origin. In this example, the elliptic parameter is $\mu=0.4$ and the Floquet eigenvalue is $h\simeq 0.67\in\mathcal{S}_{\mu,2/5}\subset\mathcal{S}_\mu$. The yellow points are $(0,\Re(\mathcal{H}\ell_2(\mu,h;0)))$ and $(0.\Im(\mathcal{H}\ell_2(\mu,h;0)))$, while the green points are  $(1,\Re(\mathcal{H}\ell_2(\mu,h;1)))$ and $(1,\Im(\mathcal{H}\ell_2(\mu,h;1)))$.}}\label{Heun2}
	\end{center}
\end{figure}
 
\begin{remark} \emph{The functions $\mathcal{H}\ell_1(\mu,h;z)$ and $\mathcal{H}\ell_2(\mu,h;z)$ have the following main properties:
\begin{itemize}
\item The functions $\mathcal{H}\ell_1(\mu,h;z)$ and  $\mathcal{H}\ell_2(\mu,h;z)$ have a branch cut discontinuity in the complex plane running from $1$ to $+\infty$. Away from the branch cut discontinuity, they are holomorphic.
\item The real parts of $\mathcal{H}\ell_1(\mu,h;z)$ and  $\mathcal{H}\ell_2(\mu,h;z)$ are continuous on ${\mathbb R}\setminus \{1/\mu\}$ and are unbounded in any punctured open neighborhood of $1/\mu>1$. They fail to be differentiable at $z=1$.
\item The functions $s\in \R\setminus\{1/\mu\} \longmapsto\Re(\mathcal{H}\ell_1(\mu,h;s))$ and $s\in \R\setminus\{1/\mu\} \longmapsto \Re(\mathcal{H}\ell_2(\mu,h;s))$ are continuous and real analytic away from $s=1$ (see Figure \ref{Heun3}) and $\Im(\mathcal{H}\ell_1(\mu,h;s))=\Im(\mathcal{H}\ell_2(\mu,h;s))=0$ for every $s<1$.
\item On the real interval $(-\infty,1)$ the functions $\mathcal{H}\ell_1(\mu,h;z)$ and $\mathcal{H}\ell_2(\mu,h;z)$ are real-valued.
\end{itemize}}
\end{remark}

\begin{figure}[h]
	\begin{center}
		\makebox[\textwidth][c]{
			\includegraphics[height=6cm,width=6cm]{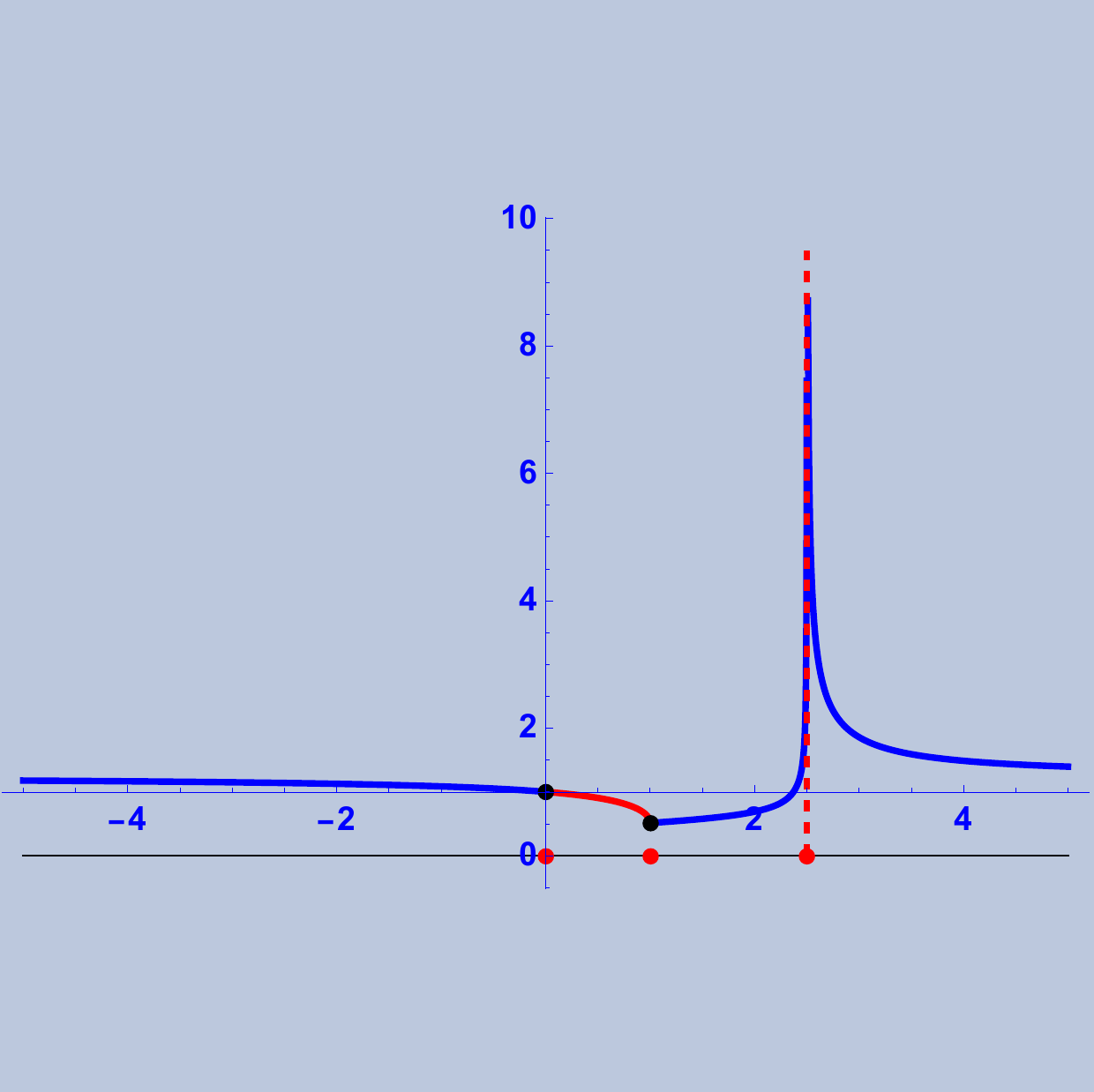}\quad\quad
			\includegraphics[height=6cm,width=6cm]{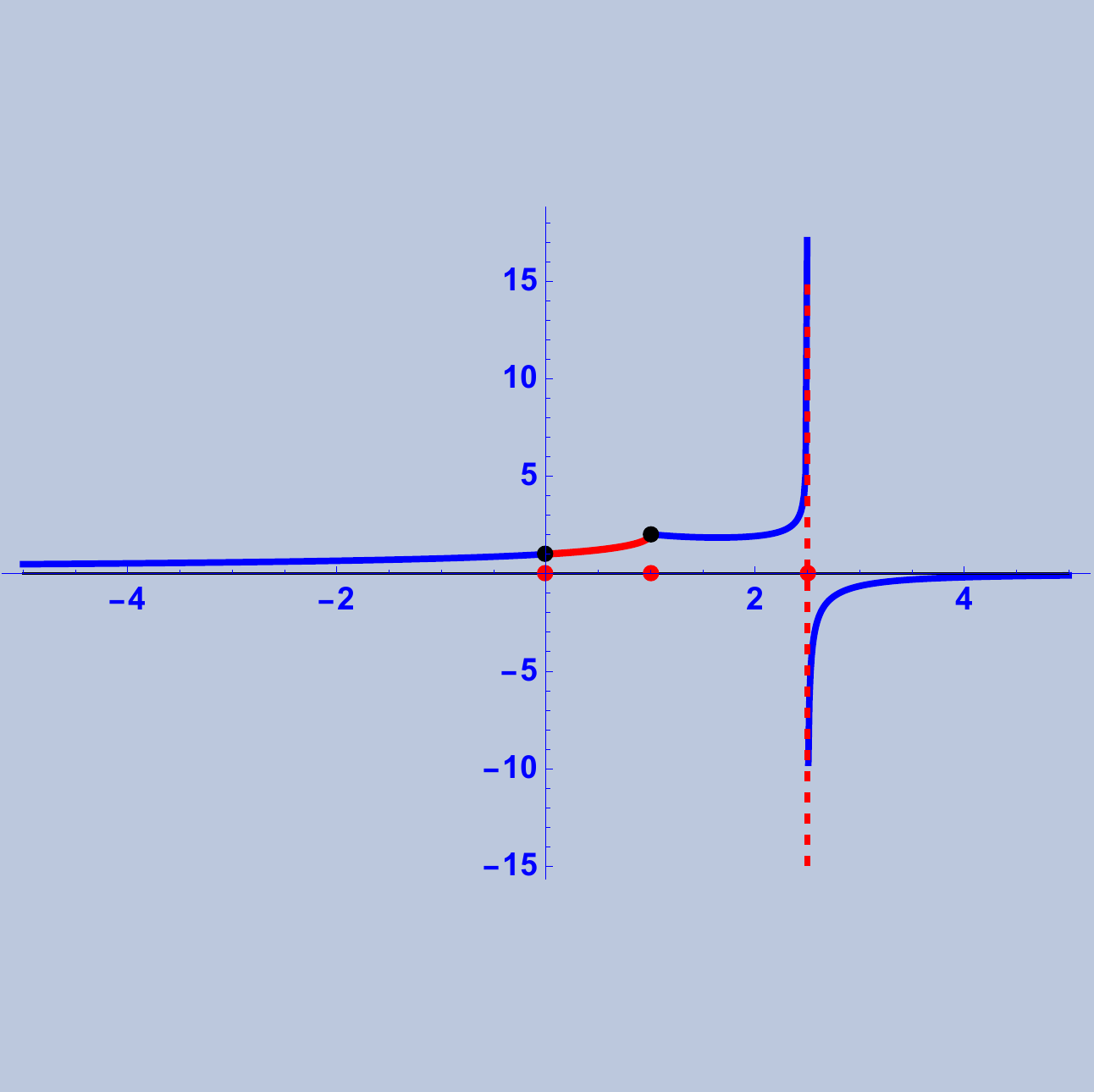}}
		\caption{\small{The graphs of the functions $s\in \R\longmapsto\Re(\mathcal{H}\ell_1(\mu,h;s))$ (Left) and $s\in \R\longmapsto\Re(\mathcal{H}\ell_2(\mu,h;s))$ (Right). For both cases $\mu=0.4$ and $h\simeq 0.67\in\mathcal{S}_{\mu,2/5}\subset\mathcal{S}_\mu$.}}\label{Heun3}
	\end{center}
\end{figure}
 
Consider the even and odd functions defined, respectively, by
\begin{eqnarray*}
	\widetilde{{\rm cl}}_{h,\mu}&:&s\in \R\longmapsto \mathcal{H}\ell_1(\mu,h;\rm{sn}^2(s,\mu))\,\sqrt{1-\mu\,{\rm sn}^2(s,\mu)}\in\R\,,\\
   \widetilde{{\rm sl}}_{h,\mu}&:&s\in \R\longmapsto \mathcal{H}\ell_2(\mu,h;{\rm sn}^2(s,\mu))\,\sqrt{1-\mu\,\rm{sn}^2(s,\mu)}\,{\rm sn}(s,\mu)\in\R\,.
\end{eqnarray*}
The functions $\widetilde{{\rm cl}}_{h,\mu}$ and $\widetilde{  {\rm sl}}_{h,\mu}$ are continuous and periodic with least period $2K(\mu)$. These functions are real analytic on the open intervals $((2p+1)K(\mu),(2p+3)K(\mu))$, $p\in {\mathbb Z}$, and fail to be differentiable at the points $(2p+1)K(\mu)$, $p\in {\mathbb Z}$. The reason for this lack of regularity is that $\lvert {\rm sn}((2p+1)K(\mu),\mu)\rvert=1$ for every $p\in\mathbb{Z}$ and the functions $s\in\R\longmapsto \Re(\mathcal{H}\ell_1(\mu,h;s))$ and $s\in\R\longmapsto \Re(\mathcal{H}\ell_2(\mu,h;s))$ are continuous but not differentiable at $s=1$. Indeed, at the points $(2p+1)K(\mu)$, $p\in {\mathbb Z}$, the derivatives of $\widetilde{{\rm cl}}_{h,\mu}$ and $\widetilde{{\rm sl}}_{h,\mu}$ have a jump discontinuity (see Figure \ref{Heun4} for an example of a graph of the functions $\widetilde{{\rm cl}}_{h,\mu}$ and $\widetilde{{\rm sl}}_{h,\mu}$). 

\begin{figure}[h]
	\begin{center}
		\makebox[\textwidth][c]{
			\includegraphics[height=4cm,width=6cm]{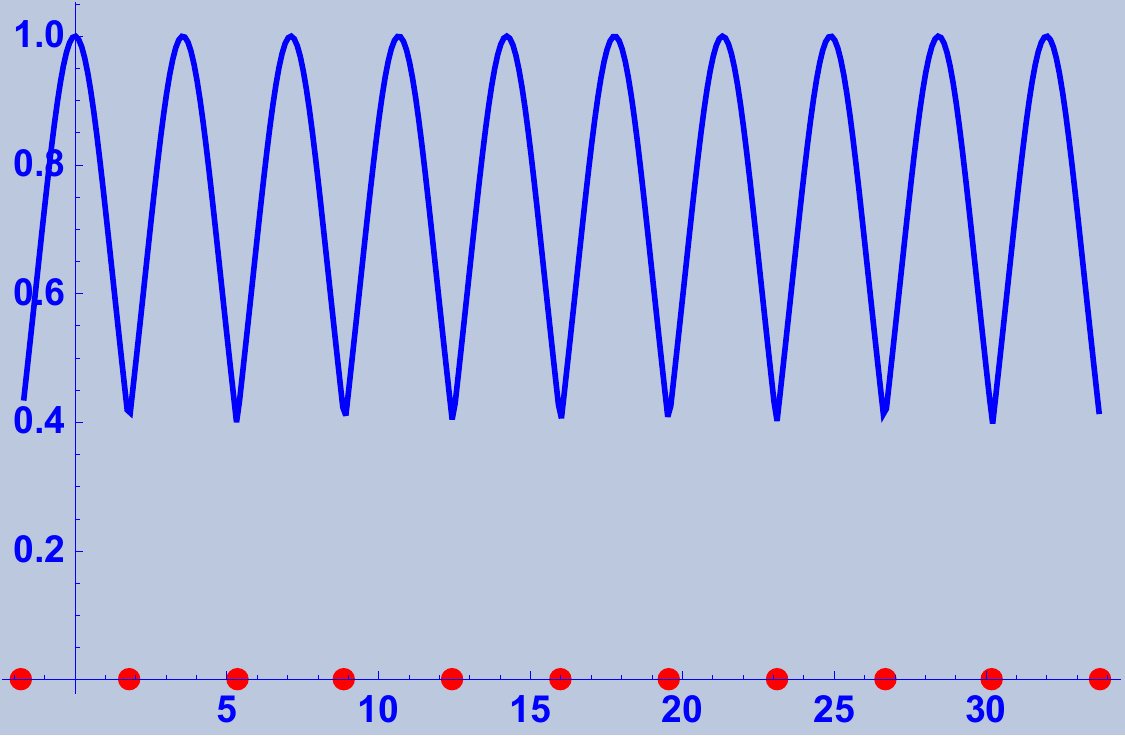}\quad\quad
			\includegraphics[height=4cm,width=6cm]{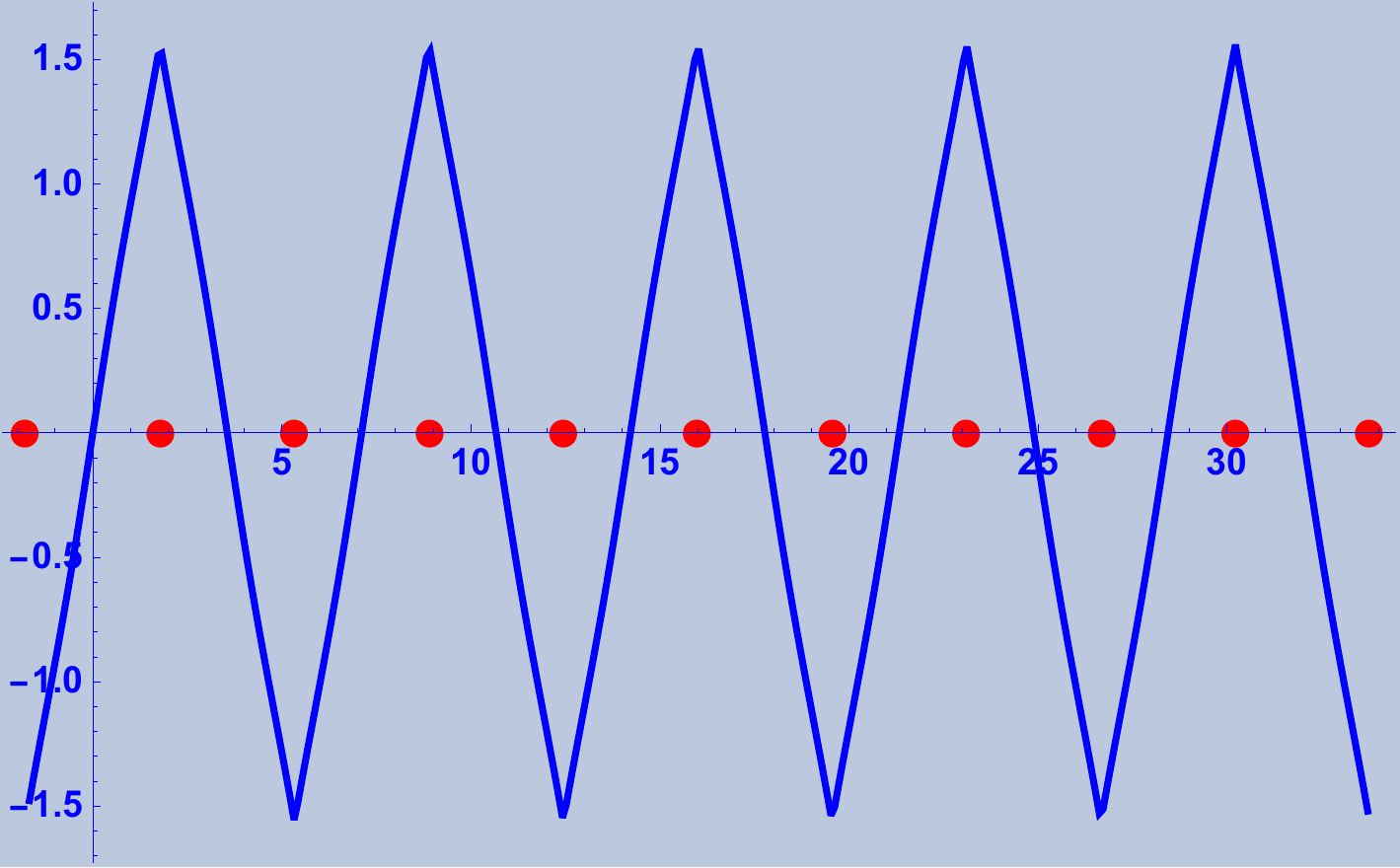}}
		\caption{\small{Left: The graph of the function $\widetilde{{\rm cl}}_{h,\mu}$. Right: The graph of the function $\widetilde{{\rm sl}}_{h,\mu}$. In both cases $\mu=0.4$ and $h\simeq 0.67$. The domain illustrated in the images is $[-K(\mu),19K(\mu)]$.}}\label{Heun4}
	\end{center}
\end{figure}

Using \eqref{Heun} and the first and second order ordinary differential equations satisfied by the Jacobi's ${\rm sn}$ function (see \cite{RW}, p. 560) it follows that the functions $\widetilde{{\rm cl}}_{h,\mu}$ and $\widetilde{{\rm sl}}_{h,\mu}$ solve the Lam\'e equation \eqref{Lame} with elliptic parameter $\mu$ and eigenvalue parameter $h$
on the intervals  $((2p+1)K(\mu),(2p+3)K(\mu))$, $p\in\mathbb{Z}$.  Since  $\widetilde{{\rm cl}}_{h,\mu}(0)=1$, $\widetilde{{\rm cl}}_{h,\mu}'(0)=0$, $\widetilde{{\rm sl}}_{h,\mu}(0)=0$, and $\widetilde{{\rm sl}}_{h,\mu}'(0)=1$, we have that these functions coincide with the fundamental solutions ${\rm cl}_{h,\mu}$ and ${\rm sl}_{h,\mu}$ of the Lam\'e equation \eqref{Lame} on the interval $[-K(\mu),K(\mu)]$. That is,
\begin{eqnarray*}
	\widetilde{{\rm cl}}_{h,\mu}|_{[-K(\mu),K(\mu)]} &=&  {\rm cl}_{h,\mu}|_{[-K(\mu),K(\mu)]}\,,\\  \widetilde{{\rm sl}}_{h,\mu}|_{[-K(\mu),K(\mu)]} &=&  {\rm sl}_{h,\mu}|_{[-K(\mu),K(\mu)]}\,.
\end{eqnarray*}
In order to extend these functions and build the fundamental solutions $ {\rm cl}_{h,\mu}$ and $ {\rm sl}_{h,\mu}$ on the whole real axis we compute
$$Q_{\pm,h,\mu}=\lim_{s\to \pm K(\mu)^{\mp}} 	\begin{pmatrix}  \widetilde{{\rm cl}}_{h,\mu}(s)& \widetilde{{\rm cl}}'_{h,\mu}(s)\\
\widetilde{{\rm sl}}_{h,\mu}(s)& \widetilde{{\rm sl}}'_{h,\mu}(s)
\end{pmatrix},
$$
and the monodromy matrix 
$$M_{\mu}(h)=Q_{+,h,\mu} Q_{-,h,\mu}^{-1}\in \SL\,.$$
Then,  on the intervals  $[-K(\mu)+pK(\mu),K(\mu)+pK(\mu)]$, $p\in {\mathbb Z}$, the fundamental solutions 
${\rm cl}_{h,\mu}$ and $ {\rm sl}_{h,\mu}$ of the Lam\'e equation \eqref{Lame} are given, respectively, by
\begin{eqnarray*}
  {\rm cl}_{h,\mu}(s)&=&\left(M_{\mu}(h)\right)^p_{1\,1} \widetilde{{\rm cl}}_{h,\mu}\left(s-pK(\mu)\right)+\left(M_{\mu}(h)\right)^p_{1\,2} \widetilde{{\rm sl}}_{h,\mu}\left(s-pK(\mu)\right),\label{fsol}\\
   {\rm sl}_{h,\mu}(s)&=&\left(M_{\mu}(h)\right)^p_{2\,1} \widetilde{{\rm cl}}_{h,\mu}\left(s-pK(\mu)\right)+\left(M_{\mu}(h)\right)^p_{2\,2} \widetilde{{\rm sl}}_{h,\mu}\left(s-pK(\mu)\right).\label{fsol2}
\end{eqnarray*}

\begin{remark} \emph{In the cases $q=0,1$, the fundamental solutions coincide, up to a constant multiplicative factor, with the Lam\'e functions $EC_1^r(s,\mu)$ and $ES_1^r(s,\mu)$ (see, for instance, \cite{V0}).}
\end{remark}

Let $m<n$ be two relatively prime natural numbers and consider that $q=m/n$ is the characteristic exponent of the Floquet eigenvalue $h\in\mathcal{S}_{\mu,q}\subset  \mathcal{S}_{\mu}$. Then, the monodromy $M_\mu(h)$ has order $2n$, if $m$ is odd, and order $4n$, if $m$ is even.  Consequently, the fundamental solutions are periodic with least period $2nK(\mu)$ (if $m$ is odd) and $4nK(\mu)$ (if $m$ is even).
   
In what follows we will illustrate with an specific example the procedure to build the fundamental periodic solutions of the Lam\'e equation \eqref{Lame}.
   
\begin{ex} Fix the elliptic parameter $\mu=0.4$ and consider the Floquet eigenvalue $h\simeq 0.67\in\mathcal{S}_{\mu,2/5}\subset\mathcal{S}_\mu$ with characteristic exponent $q=2/5$.
	
For these values, Figures \ref{Heun1} and \ref{Heun2} reproduce the graphs of the imaginary and real parts of the Heun functions $\mathcal{H}\ell_1(\mu,h;z)$ and $\mathcal{H}\ell_2(\mu,h;z)$. Figure \ref{Heun3} depicts the graphs of the functions $s\in \R\longmapsto \Re(\mathcal{H}\ell_1(\mu,h;s))$ and $s\in \R\longmapsto \Re(\mathcal{H}\ell_2(\mu,h;s))$. Observe that since $-1\le {\rm sn}(s,\mu)\le 1$ and 
${\rm sn}(s,\mu)=1$ if and only if $s=2pK(\mu)$ and ${\rm sn}(s,\mu)=-1$ if and only if $s=(2p-1)K(\mu)$, $p\in {\mathbb Z}$, only $\Re(\mathcal{H}\ell_1(\mu,h;s))|_{[0,1]} = \mathcal{H}\ell_1(\mu,h;s)|_{[0,1]}$ and $\Re(\mathcal{H}\ell_2(\mu,h;s))|_{[0,1]}=\mathcal{H}\ell_2(\mu,h;s)|_{[0,1]}$ are relevant for the construction of the fundamental solutions of the Lam\'e equation (these parts are represented in red in Figure \ref{Heun3}). Additionally, Figure \ref{Heun4} reproduces the graphs of the functions $  \widetilde{{\rm cl}}_{h,\mu}$ and $  \widetilde{{\rm sl}}_{h,\mu}$ on the interval $[-K(\mu),19K(\mu)]$. The monodromy matrix $M_{\mu}(h)\in \SL$ can be computed numerically obtaining
$$M_{\mu}(h)\simeq 	\begin{pmatrix} -0.309017&-0.331386\\
2.72947& -0.309017\end{pmatrix},
$$
which can be checked to have order $10$. 

Then, we can evaluate ${\rm cl}_{h,\mu}$ and ${\rm sl}_{h,\mu}$ by means of the explicit expression obtained above. Due to the order of the monodromy, the fundamental solutions are periodic with least period $20K(\mu)$ (recall that since $m=2$ is even the least period is $4nK(\mu)=20K(\mu)$).  

These functions have been illustrated employing the above explained method in red in Figure \ref{Heun5}. Moreover, they are compared with a numerical solution of the Lam\'e equation (in dashed black).
\end{ex}

\begin{figure}[h]
	\begin{center}
		\makebox[\textwidth][c]{
			\includegraphics[height=4cm,width=6cm]{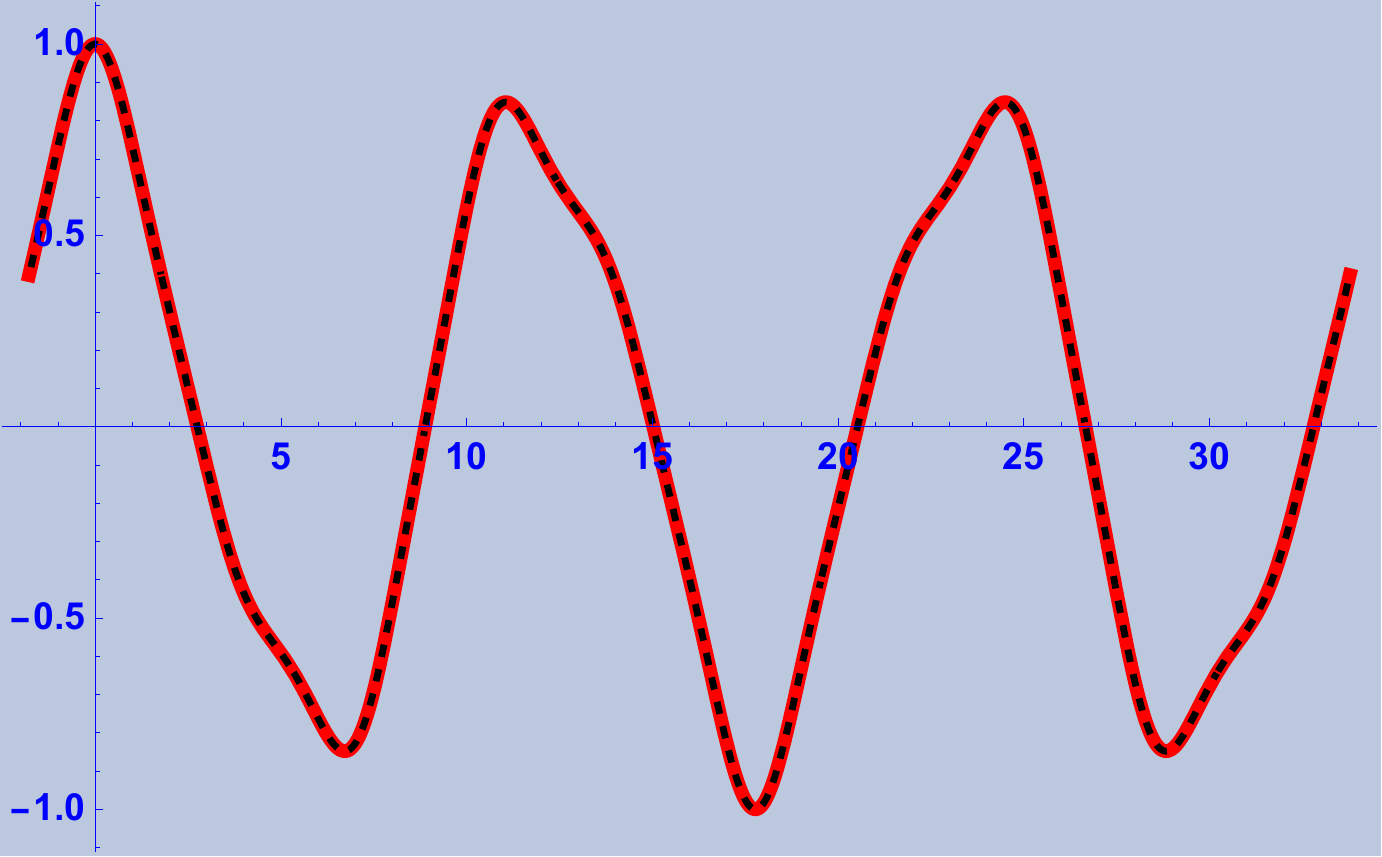}\quad\quad
			\includegraphics[height=4cm,width=6cm]{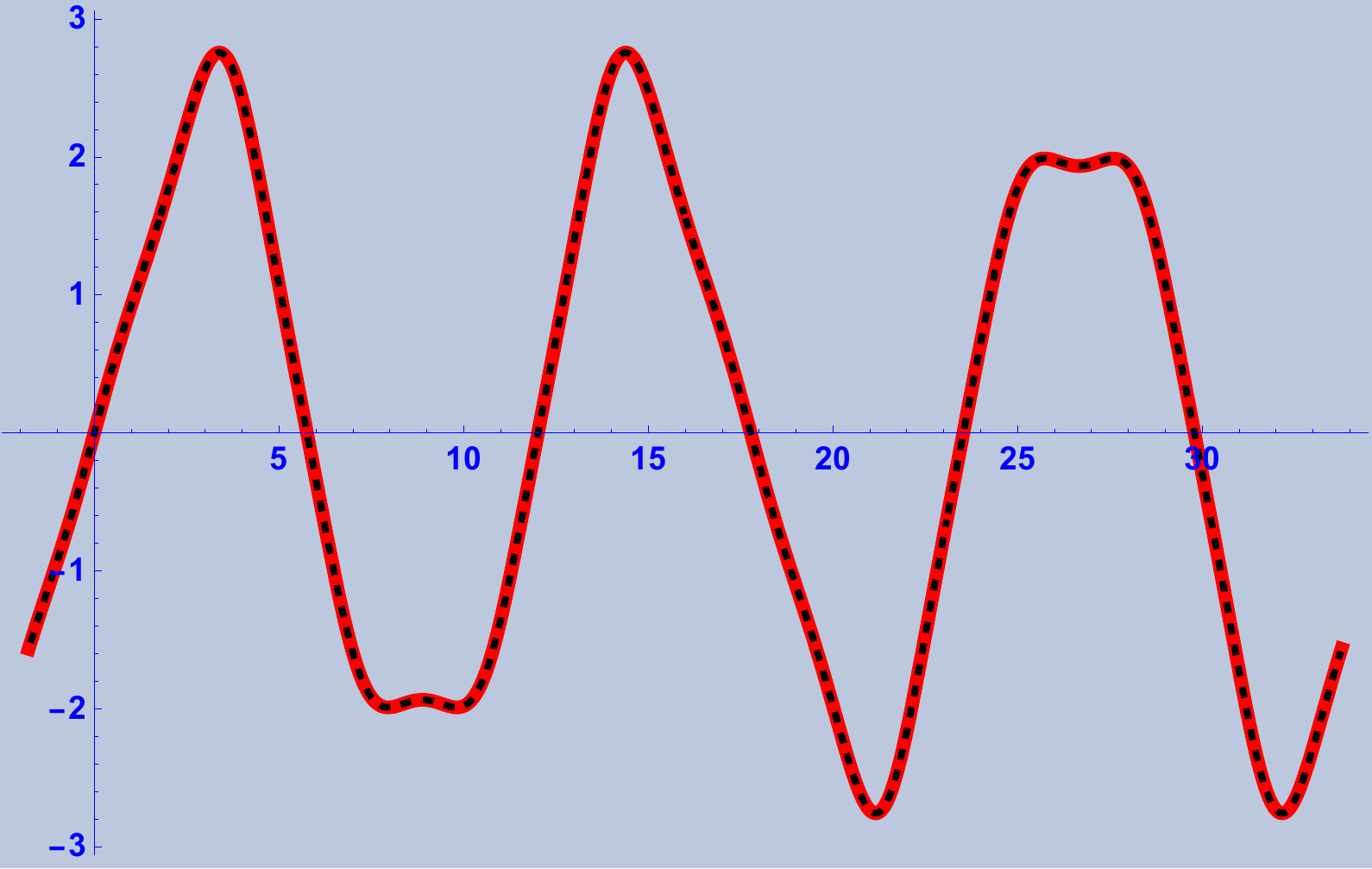}}
		\caption{\small{Left: The graph of the function ${\rm cl}_{h,\mu}$. Right: The graph of the function ${\rm sl}_{h,\mu}$. In both cases $\mu=0.4$ and $h\simeq 0.67$. The domain illustrated in the images is $[-K(\mu),19K(\mu)]$. Both graphs are compared with a numerical solution of the Lam\'e equation \eqref{Lame}, which is shown in dashed black.}}\label{Heun5}
	\end{center}
\end{figure}

Now that we know how to compute the periodic fundamental solutions of the Lam\'e equation \eqref{Lame}, we are in a position to study stationary solutions of the LIEN flow \eqref{LIEN2}.

\subsection{Stationary Solutions of the LIEN Flow}

\begin{defn} We say that a null curve $\gamma:J\subseteq\R\longrightarrow\AdS$ is a \emph{stationary curve} of the LIEN flow \eqref{LIEN2} if its evolution by the flow is given by
$$(s,t)\in J\times I\subseteq\R^2\longmapsto A(t)\gamma(s+2\ell t)B(t)^{-1}\in\AdS\,,$$
where $\ell\in\R$ and $A,B:I\subseteq\R\longrightarrow\SL$ are two maps such that $A(0)=B(0)={\rm Id}$. (Recall that, for convenience, we are assuming that $0\in I$).
\end{defn} 

From Theorem \ref{induced}, it follows that the induced evolution equation on the bending $\bending$ of $\gamma$ is the KdV equation \eqref{KdV2}. Since $\gamma$ is a stationary curve of the LIEN flow \eqref{LIEN2}, we then deduce that the bending $\bending$ of $\gamma$ is a solution of the third order differential equation
\begin{equation}\label{stationary}
	\bending'''+2\ell \bending'-6\bending\bending'=0\,,
\end{equation}
where $\left(\,\right)'$ denotes the derivative with respect to the proper time $s\in J$. The bending of its evolution by the flow is the traveling wave solution $\bending(s+2\ell t)$ of the KdV equation \eqref{KdV2}.

If the bending $\bending$ of $\gamma$ is a nonconstant function, equation \eqref{stationary} can be integrated twice obtaining a first order differential equation in separable variables which involves the square root of a polynomial of degree three. Hence, nonconstant solutions of \eqref{stationary} can be expressed in terms of elliptic functions. In particular, the periodic solutions are of the form
\begin{equation}\label{solution}
	\bending(s)=\frac{1}{h_--h_+}\left(4\,\mu\,{\rm sn}^2\left(\sqrt{\frac{2}{h_--h_+}} \,s,\mu\right)-h_--h_+\right),
\end{equation}
where $\mu\in(0,1)$ is a constant\footnote{The constant $\mu\in(0,1)$ of \eqref{solution} is \emph{a priori} unrelated to the elliptic parameter of the Lam\'e equation \eqref{Lame}. However, in Theorem \ref{stationarycurve} it is shown that both constants do coincide. Hence, for simplicity, we use the same letter for both of them.} and $h_->h_+$ are two real numbers satisfying
\begin{equation}\label{m}
	\ell(h_--h_+)+3(h_-+h_+)=4(1+\mu)\,.
\end{equation}

We next show how to construct stationary curves of the LIEN flow \eqref{LIEN2} from the fundamental solutions of the Lam\'e equation \eqref{Lame}. Recall that the procedure to construct these periodic fundamental solutions was described in Subsection 5.2.

\begin{thm}\label{stationarycurve}
	A null curve $\gamma:J=\R\longrightarrow\AdS$ with nonconstant periodic bending $\bending$ is a stationary curve of the LIEN flow \eqref{LIEN2} if and only if it is equivalent to the curve
	$$s\in\R\longmapsto \delta_{h_+,\mu}\left(\sqrt{\frac{2}{h_--h_+}} \,s\right)\delta_{h_-,\mu}\left(\sqrt{\frac{2}{h_--h_+}} \,s\right)^{-1}\in\AdS\,,$$
	where $h_->h_+$ are two real numbers, $\mu\in(0,1)$ is a constant and $\delta_{h,\mu}:\R\longrightarrow\SL$ is the map defined in \eqref{beta}. 
	
	Moreover, the null curve $\gamma$ is closed if and only if $h_+,h_-\in\mathcal{S}_{\mu}$. In other words, if and only if $h_+$ and $h_-$ are Floquet eigenvalues with elliptic parameter $\mu$.
\end{thm}
\begin{proof} Consider a stationary curve $\gamma$ of the LIEN flow \eqref{LIEN2} with nonconstant periodic bending $\bending$. Since $\gamma$ is stationary and its bending is periodic then $\bending$ is given by \eqref{solution}. 
	
Denote by $(\eta_+,\eta_-)$ the pair of cousins associated with the null curve $\gamma$. Without loss of generality we may assume that $\gamma$ is normalized so that the components of the spinor frame field $(F_+,F_-)$ along $\gamma$ satisfy
$$F_\pm(0)=\begin{pmatrix} 1/\sqrt{\sigma} & 0 \\ 0 & \sqrt{\sigma} \end{pmatrix},$$
respectively. Here, we are using the constant $\sigma=\sqrt{2/(h_--h_+)}$ to simplify the expression. 

From Theorem \ref{relation}, it follows that the central affine curvatures of the star-shaped curves $\eta_+$ and $\eta_-$ are, respectively, $\curvature_+=\bending+1$ and $\curvature_-=\bending-1$. Hence, we deduce from \eqref{canonicaldF} that $\eta_\pm$ satisfy
\begin{eqnarray*}
	\eta_+''(s)&=&\left(\bending(s)+1\right)\eta_+(s)\,,\\
	\eta_-''(s)&=&\left(\bending(s)-1\right)\eta_-(s)\,,
\end{eqnarray*}
respectively. Define the functions $\widetilde{\eta}_\pm(t)=\sqrt{\sigma}\,\eta_\pm(t/\sigma)$. From the differential equations satisfied by $\eta_\pm$ and the explicit expression of $\bending(s)$ given in \eqref{solution}, we see that the new functions $\widetilde{\eta}_\pm$ satisfy Lam\'e equations of order one \eqref{Lame} with elliptic parameter $\mu\in(0,1)$ and eigenvalue parameters $h_\pm$, respectively.

Moreover, due to our choice of normalization, the maps $\delta_{h_+,\mu}=(\widetilde{\eta}_+,\widetilde{\eta}_+')$ and $\delta_{h_-,\mu}=(\widetilde{\eta}_-,\widetilde{\eta}_-')$ satisfy the Cauchy problem \eqref{Cauchy}. We then conclude from Theorem \ref{relation}, that $\gamma$ is as in the statement. The converse, follows in the same way.

For the second assertion, observe that the null curve $\gamma$ is closed if and only if the maps $\delta_{h_\pm,\mu}$ are periodic with commensurable periods. On the other hand, these maps are periodic if and only if $h_\pm\in\mathcal{S}_{\mu,q_\pm}$, respectively, where $q_\pm\in[0,1]\cap\mathbb{Q}$. Let $q_\pm=m_\pm/n_\pm$ for relatively prime natural numbers $n_\pm$ and $m_\pm$. Hence, $\delta_{h_\pm,\mu}$ are periodic with least period $2n_\pm K(\mu)$ or $4n_\pm K(\mu)$ depending on whether $m_\pm$ are even or odd, respectively (see Subsection 5.2). Consequently, their periods are automatically commensurable.
\end{proof}

Let $\gamma:J=\R\longrightarrow\AdS$ be a stationary curve of the LIEN flow \eqref{LIEN2} with nonconstant periodic bending $\bending$. From Theorem \ref{stationarycurve}, we have that $\gamma$ depends on three real parameters, namely, $h_->h_+$ and $\mu\in(0,1)$.

\begin{defn} The constant $\mu\in(0,1)$ is called the \emph{elliptic parameter} and the real numbers $h_->h_+$ will be referred as to the \emph{Lam\'e eigenvalues}.
\end{defn}

\begin{figure}[h]
	\begin{center}
		\makebox[\textwidth][c]{
			\includegraphics[height=4cm,width=4cm]{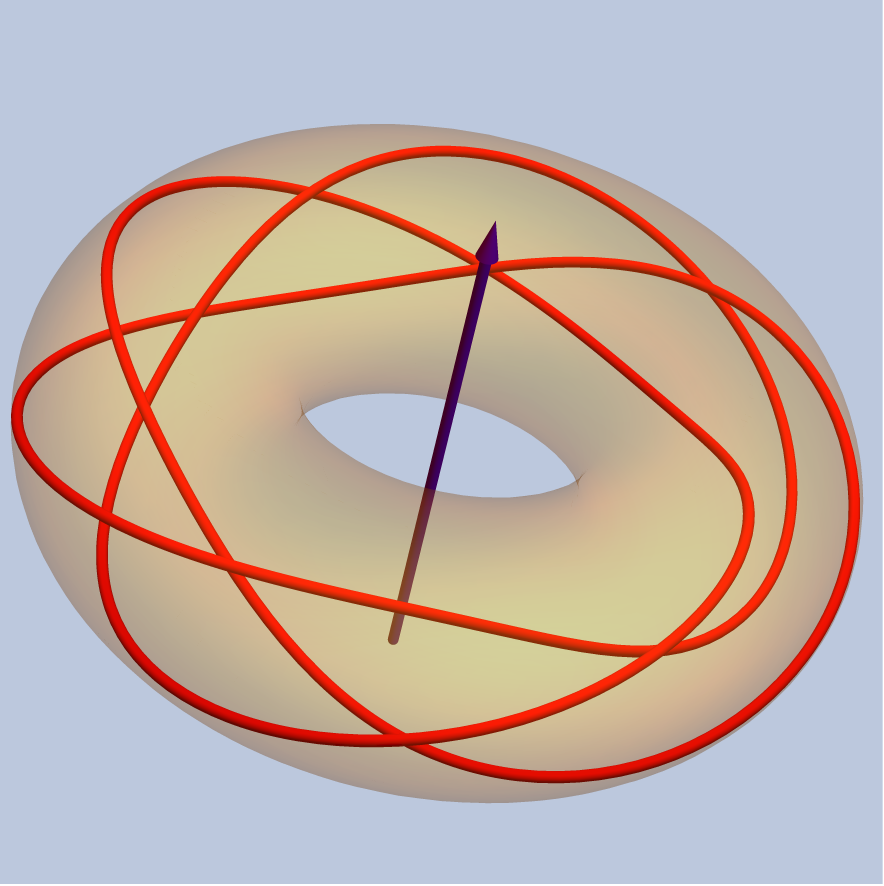}\quad
			\includegraphics[height=4cm,width=4cm]{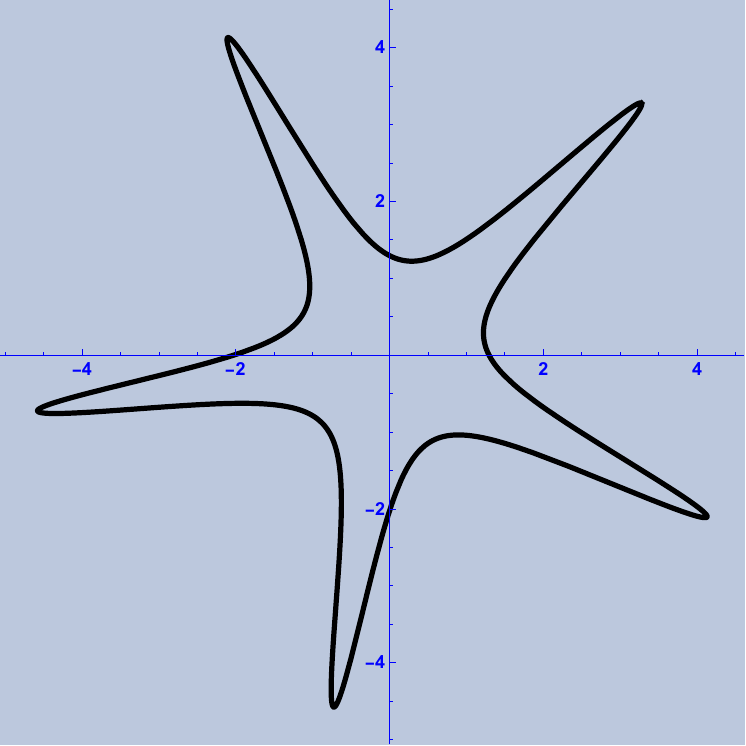}\quad
			\includegraphics[height=4cm,width=4cm]{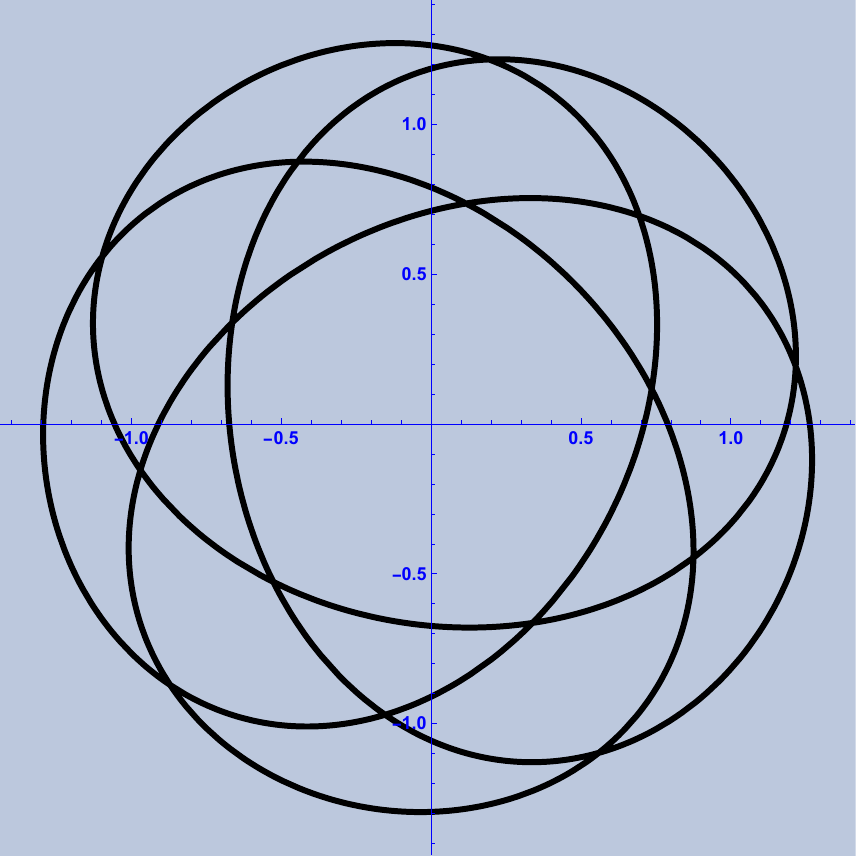}
		}
		\caption{\small{Left: The closed stationary curve of the LIEN flow \eqref{LIEN2} with elliptic parameter $\mu=0.9$ and Lam\'e eigenvalues $h_-\simeq 2.23>h_+\simeq 0.93\in\mathcal{S}_{\mu,2/5}\subset\mathcal{S}_\mu$. This curve represents a torus knot of type $(3,5)$. Center/Right: The pair of star-shaped cousins $(\eta_+,\eta_-)$ associated with the null curve.}} \label{simple}
	\end{center}
\end{figure}

In Figure \ref{simple} (Left) we show a closed stationary curve of the LIEN flow \eqref{LIEN2} obtained as in Theorem \ref{stationarycurve}. In order to produce the stationary curve, we first fix our elliptic parameter $\mu\in(0,1)$ and a desired rational number $q\in[0,1]\cap\mathbb{Q}$ (which \emph{a posteriori} will be closely related to the type of the torus knot represented by the curve). We then compute two Floquet eigenvalues $h_->h_+$ with characteristic exponent $q$ by numerically solving $\tau_\mu(h)=\cos(q\pi)$ (see Subsection 5.1). This implies that $h_+,h_-\in\mathcal{S}_{\mu,q}\subset\mathcal{S}_\mu$ and, hence, guarantees that the null curve $\gamma$ will be closed\footnote{One may as well fix two different rational numbers, say $q_+\neq q_-\in[0,1]\cap\mathbb{Q}$, and compute one Floquet eigenvalue $h_\pm$ for each characteristic exponent $q_\pm$. This also guarantees that the null curve $\gamma$ will be closed.}. Finally, we solve the two Lam\'e equations satisfied by the star-shaped cousins associated with the null curve (see Subsection 5.2) and apply Theorem \ref{relation} to obtain our stationary curve. In addition, we also illustrate in Figure \ref{simple} the pair of star-shaped cousins associated with this stationary curve. We observe that when the Lam\'e eigenvalue belongs to the interval $(\mu,1)$ (for which $q\in(0,1)$ necessarily) the star-shaped curve may be simple, while, in other cases, it has self-intersections (cf. Figure \ref{simple} Center and Right, respectively).

Other examples of closed stationary curves of the LIEN flow \eqref{LIEN2} (constructed as in Theorem \ref{stationarycurve} and following the procedure explained above) can be found in Figures \ref{cable} and \ref{examples}. In both cases, the stationary curves are cable knots.

\begin{remark}\label{cablerem} \emph{We recall here what a cable knot is (from a differential geometric viewpoint). Consider a smooth non-trivial knot of length $l$ parameterized by a curve $\alpha:\R\longrightarrow \R^3$. Without loss of generality we suppose that $\alpha$ is bi-regular and parameterized by the Euclidean arc length. Let $\{\vec{t},\vec{v},\vec{w}\}$ be an orthogonal frame field along $\alpha$ such that $\vec{t}$ is the unit tangent vector field. Let $r:\R\longrightarrow \R$ be a smooth positive periodic function with least period $l$. If ${\rm max}(r)$ is sufficiently small, the map
$$f_{\alpha}  : [(s,\theta)]\in \widetilde{ {\mathbb T}}^2=\frac{\R}{l{\mathbb Z}}\times \frac{\R}{2\pi{\mathbb Z}} \longmapsto  \alpha(s)+r(s)\cos(\theta)\vec{v}(s)+r(s)\sin(\theta)\vec{w}(s)\in\R^3\,,$$
is a smooth embedding of the torus $\widetilde{ {\mathbb T}}^2$. However,  this is not isotopic to the standard embedding. Fix a rational number $q=m/n$, $q\notin {\mathbb Z}$. Then 
$$\beta :[s]\in  \frac{\R}{n l{\mathbb Z}}\longmapsto f_{\alpha}\left(s,  \frac{2\pi}{l}qs\right)\in f_{\alpha}(\widetilde{ {\mathbb T}}^2)\,,$$
is a simple closed curve. The image of $\beta$ is a smooth cable knot of type $(m,n)$ with respect to the framing $\{\vec{t},\vec{v},\vec{w}\}$ and with the center line $\alpha$. The cabling depends both on $q=m/n$ and on the framing. We refer the reader to \cite{AD} for a topological description of satellite and cable knots.}
\end{remark}

\begin{figure}[h]
	\begin{center}
		\makebox[\textwidth][c]{
			\includegraphics[height=4cm,width=4cm]{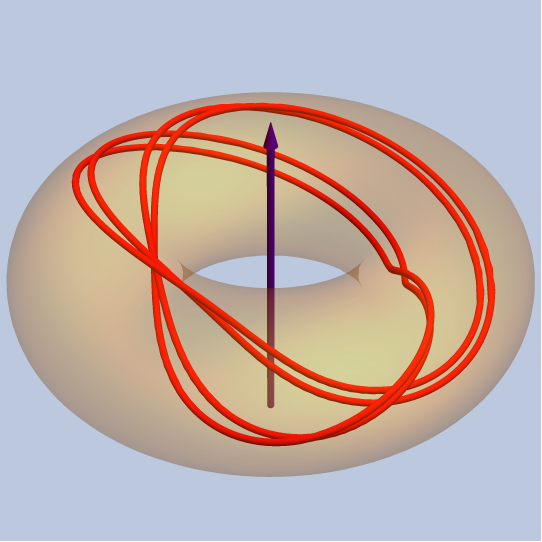}\quad
				\includegraphics[height=4cm,width=4cm]{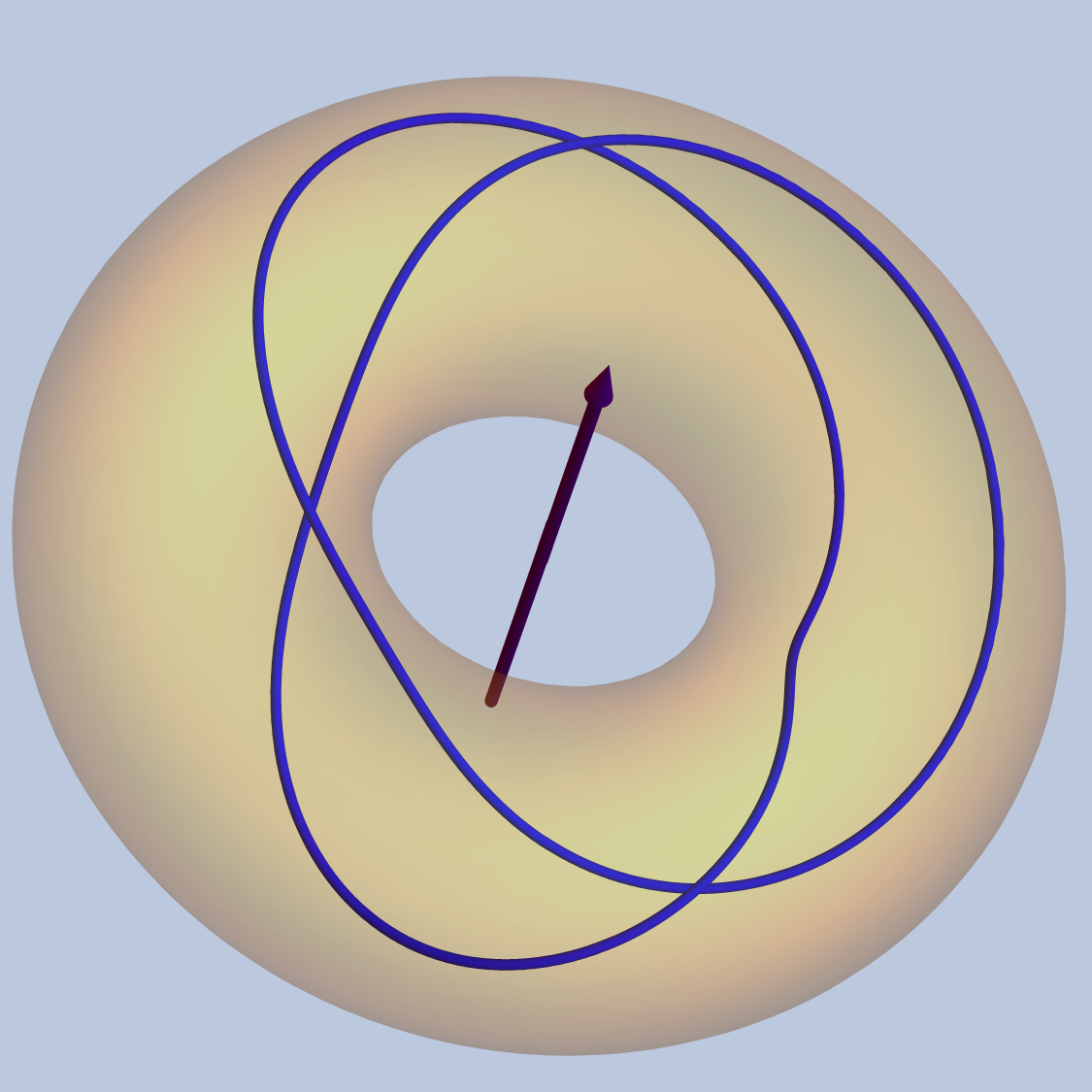}\quad
			\includegraphics[height=4cm,width=4cm]{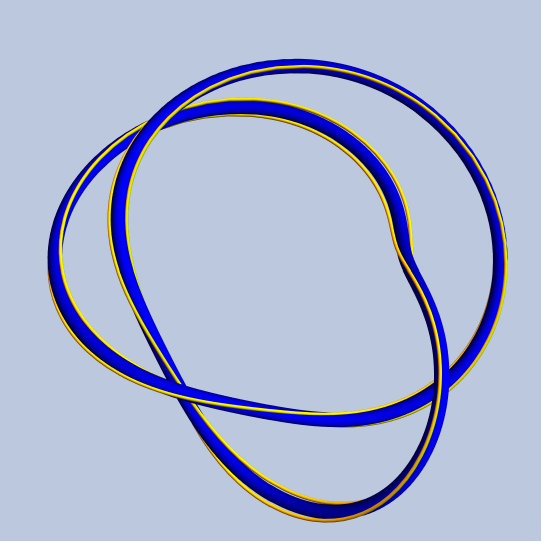}
		}
		\caption{\small{Left: The closed stationary curve of the LIEN flow \eqref{LIEN2} with elliptic parameter $\mu=0.6$ and Lam\'e eigenvalues 
		$h_-\simeq 65.59>h_+\simeq 3.29\in\mathcal{S}_{\mu,0}\subset\mathcal{S}_\mu$. This curve represents a cable knot. Center: The center line of the cabling, which is a trefoil torus knot. Right: An embedded tube $f_\alpha(\widetilde{\mathbb{T}}^2)$ (in the notation of Remark \ref{cablerem}) around the center line and the closed stationary curve traced on the tube (in yellow).}} \label{cable}
	\end{center}
\end{figure}

\begin{figure}[h]
	\begin{center}
		\makebox[\textwidth][c]{
			\includegraphics[height=4cm,width=4cm]{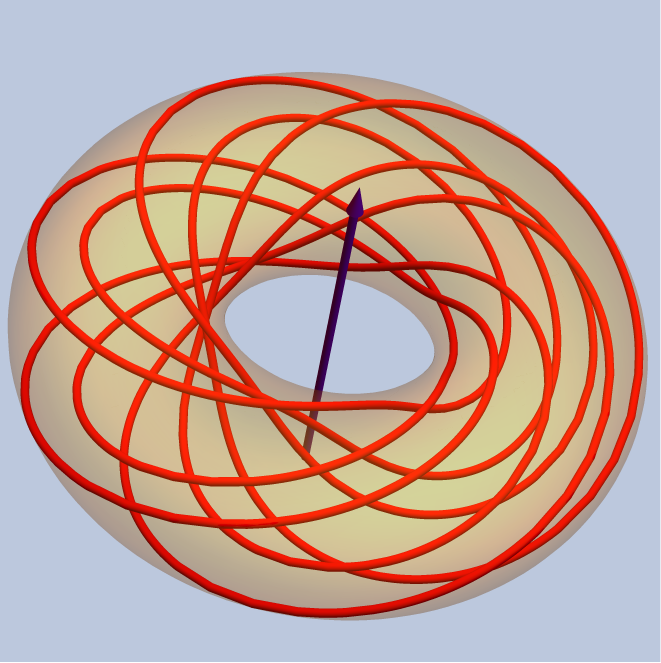}\quad
			\includegraphics[height=4cm,width=4cm]{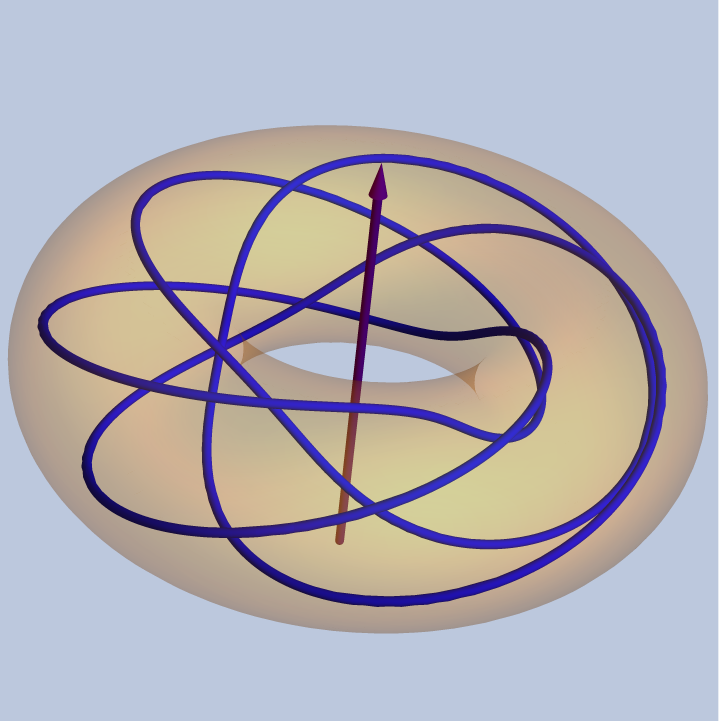}\quad
			\includegraphics[height=4cm,width=4cm]{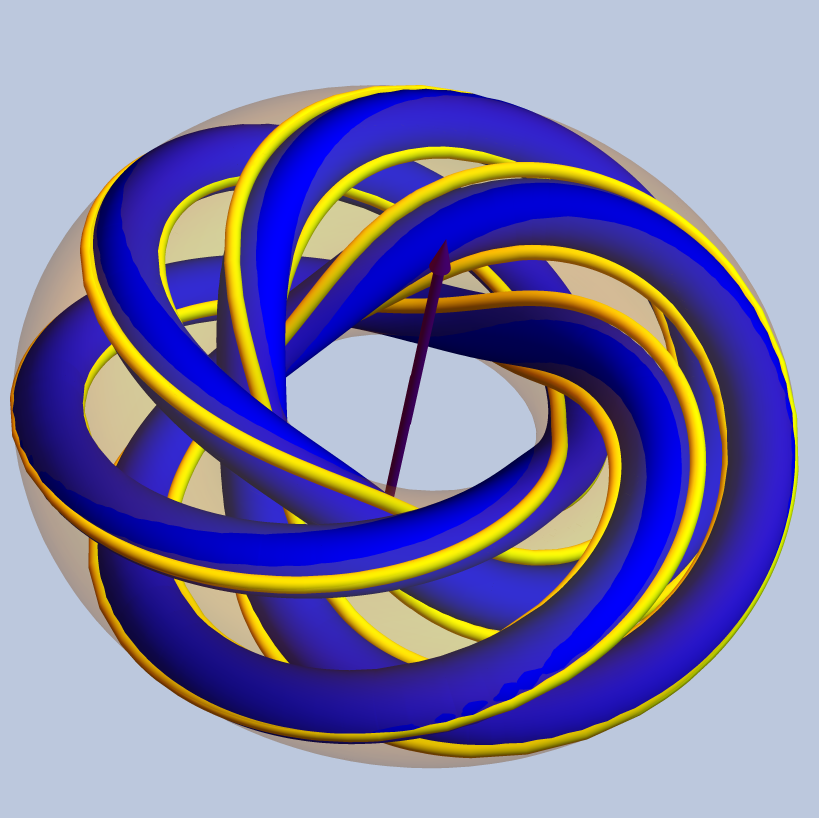}
		}
		\caption{\small{Left: The closed stationary curve of the LIEN flow \eqref{LIEN2} with elliptic parameter $\mu=0.9$ and Lam\'e eigenvalues $h_-\simeq 5.98>h_+\simeq 0.93\in\mathcal{S}_{\mu,2/5}\subset\mathcal{S}_\mu$. This curve represents a cable knot. Center: The center line of the cabling, which is a torus knot of type $(4,5)$.  Right: An embedded tube $f_\alpha(\widetilde{\mathbb{T}}^2)$ (in the notation of Remark \ref{cablerem}) around the center line and the closed stationary curve traced on the tube (in yellow).}} \label{examples}
	\end{center}
\end{figure}

The evolution by the LIEN flow \eqref{LIEN2} of a closed stationary curve can be explicitly described.

\begin{prop} Let $\gamma:J=\R\longrightarrow\AdS$ be a closed stationary curve of the LIEN flow \eqref{LIEN2} with elliptic parameter $\mu\in(0,1)$ and Lam\'e eigenvalues $h_->h_+\in\mathcal{S}_\mu$. Then, the evolution of $\gamma$ by the flow is
$$(s,t)\in \R\times I\subseteq\R^2\longmapsto {\rm Exp}(t\mathfrak{m}_+)\gamma(s+2\ell t){\rm Exp}(-t\mathfrak{m}_-)\in\AdS\,,$$
where $\ell\in\R$ is the constant satisfying \eqref{m} and $\mathfrak{m}_\pm\in {\mathfrak sl}(2,\R)$ are given by
	$$\mathfrak{m}_\pm=\begin{pmatrix} 0 & \frac{8\sqrt{2}(h_\pm-1)(\mu-h_\pm)}{(h_--h_+)^{3/2}} \\ \frac{8\sqrt{2}(h_\pm-1-\mu)}{(h_--h_+)^{3/2}} & 0 \end{pmatrix}.$$
\end{prop}
\begin{proof}
	Assume that $\gamma$ is a closed stationary curve of the LIEN flow \eqref{LIEN2} with $\mu\in(0,1)$ and $h_->h_+$ fixed and let $(F_+,F_-)$ be the spinor frame field along $\gamma$. Then, the bending $\bending$ of $\gamma$ satisfies \eqref{solution} (and, hence, \eqref{stationary} too) for the constant $\ell\in\R$ given by \eqref{m}.  
	
	We begin by proving the following conservation laws
	\begin{eqnarray}
		F_+\left(\mathcal{P}_1-2\ell\,\mathcal{K}_1\right)F_+^{-1}&=&\mathfrak{m}_+\,,\label{momenta+}\\
		F_-\left(\mathcal{P}_{-1}-2\ell\,\mathcal{K}_{-1}\right)F_-^{-1}&=&\mathfrak{m}_-\,,\label{momenta-}
	\end{eqnarray} 
	where $\mathcal{K}_\lambda$ and $\mathcal{P}_\lambda$, $\lambda=1,-1$, are defined in \eqref{KP} and $\mathfrak{m}_\pm$ are constant matrices. The proof of both conservation laws is analogous, hence, we will avoid explicitly writing the sub-indexes. The following computations are valid for both possible sub-indexes.
	
	Define the function $\mathcal{H}(s)=\mathcal{P}(s)+2\ell\,\mathcal{K}(s)$. Using the spinorial Frenet-type equations \eqref{dF+} or \eqref{dF-}, respectively, of $\gamma$, we have
	$$\left(F\,\mathcal{H}\,F^{-1}\right)'=F\left(\mathcal{H}'-\left[\mathcal{H},\mathcal{K}\right]\right)F^{-1}=0\,.$$
	The second equality follows from \eqref{stationary} and the following computation involving the definition of $\mathcal{H}$,
	$$\mathcal{H}'-\left[\mathcal{H},\mathcal{K}\right]=\mathcal{P}'+2\ell\,\mathcal{K}'-\left[\mathcal{P},\mathcal{K}\right]=\begin{pmatrix} 0 & -\bending'''-2\ell\bending'+6\bending\bending' \\ 0 & 0 \end{pmatrix}=0\,.$$
	Consequently, $F\mathcal{H}F^{-1}$ must be constant which shows the conservation laws \eqref{momenta+} and \eqref{momenta-}. Moreover, since the matrices $\mathfrak{m}_\pm$ are constant, it is enough to compute \eqref{momenta+}-\eqref{momenta-} at $s=0$. Using the same normalization for $F_\pm$ employed in the proof of Theorem \ref{stationarycurve} and the derivatives of $\bending$ obtained from \eqref{solution}, we then deduce that the matrices $\mathfrak{m}_\pm$ are as in the statement.
	
	We next show that $\widehat{\gamma}(s,t)={\rm Exp}(t\mathfrak{m}_+)\gamma(s+2\ell t){\rm Exp}(-t\mathfrak{m}_-)$ is a solution of the LIEN flow \eqref{LIEN2}. Observe that, by construction, the spinor frame field $(\widehat{F}_+,\widehat{F}_-)$ along $\widehat{\gamma}$ is given by
	\begin{eqnarray*}
		\widehat{F}_+(s,t)&=&{\rm Exp}(t\mathfrak{m}_+)F_+(s+2\ell t)\,,\\
		\widehat{F}_-(s,t)&=&{\rm Exp}(t\mathfrak{m}_-)F_-(s+2\ell t)\,.
	\end{eqnarray*}
	As in the previous part of the proof, we need to work with both possible sub-indexes, but again the computations are analogous. Thus, we will again avoid writing the sub-indexes and the computations will be valid for both cases. Differentiating the components of the spinor frame field along $\widehat{\gamma}$, we obtain
	\begin{eqnarray*}
		d\widehat{F}&=&{\rm Exp}(t\mathfrak{m})F'ds+\left(\mathfrak{m}{\rm Exp}(t\mathfrak{m})F+2\ell\,{\rm Exp}(t\mathfrak{m})F'\right)dt\\
		&=&{\rm Exp}(t\mathfrak{m})\left[F'ds+\left(\mathfrak{m}F+2\ell F'\right)dt\right].
	\end{eqnarray*}
Consequently, from the spinorial Frenet-type equations of $\gamma$ \eqref{dF+}-\eqref{dF-} and the conservation laws \eqref{momenta+}-\eqref{momenta-},
\begin{eqnarray*}
	\widehat{F}^{-1}d\widehat{F}&=&F^{-1}{\rm Exp}(-t\mathfrak{m})d\widehat{F}=F^{-1}F'ds+\left(F^{-1}\mathfrak{m}F+2\ell\,F^{-1}F'\right)dt\\
	&=&\mathcal{K}\,ds+\mathcal{P}\,dt\,.
\end{eqnarray*}
Therefore, using the expression of $\pi_s^*(\omega)$ given in Section 2, the Cartan frame field $\widehat{\mathcal{F}}$ along $\widehat{\gamma}$ evolves according to
$$\partial_t\widehat{\mathcal{F}}=\widehat{\mathcal{F}}\begin{pmatrix} 0 & -4\sqrt{2} & 0 & -2\sqrt{2}\bending \\ -2\sqrt{2}\bending & -2\bending' & \sqrt{2}\left(-4+2\bending^2-\bending''\right) & 0 \\ 0 & 2\sqrt{2}\bending & 0 & -\sqrt{2}\left(-4+2\bending^2-\bending''\right) \\ -4\sqrt{2} & 0 & -2\sqrt{2}\bending & 2\sqrt{2}\bending \end{pmatrix},$$
where $\bending\equiv\bending(s+2\ell t)$. Hence, $\partial_t\widehat{\gamma}=-2\sqrt{2}\left(\bending T-2B\right)$, which shows that $\widehat{\gamma}$ is a solution of the LIEN flow \eqref{LIEN2} with bending $\bending\equiv\bending(s+2\ell t)$ and initial condition the stationary curve $\gamma(s)$. This finishes the proof.
\end{proof}

\begin{remark} \emph{Equation \eqref{stationary} can be obtained as the Euler-Lagrange equation for the functional
	$$ \gamma\longmapsto \int_\gamma \left(\bending+2\ell \right)ds\,.$$
This variational problem has been studied in the last couple of decades by several authors (see for instance \cite{AGL,ABG, MN1,MN1Bis} and references therein).}
\end{remark}

\section{KKSH-Solutions} 

In this section (more precisely, in Subsection 6.2) we will investigate the solutions of the LIEN flow \eqref{LIEN2} arising from the 3-parameter family of solutions of the KdV equations studied by Kevrekidis, Khare, Saxena and Herring in \cite{KKSH} (see also \cite{AMP}). These functions depend on two elliptic parameters $\mu,\tau \in (0,1)$ and a homothetic parameter $h>0$. Unlike the stationary case, the Hill's equations that need to be solved for the construction of the corresponding evolution of null curves cannot be written in an exact form in terms of known special functions. Therefore our analysis is essentially based on the numerical solutions of such equations. For this purpose, we will use symbolic and numerical routines implemented in the software \emph{Mathematica 13.3}. We begin by recalling how the KKSH-solutions of the KdV equation \eqref{KdV2} are constructed (Subsection 6.1).

\subsection{KKSH-Solutions of the KdV} 

Let $\mu,\tau \in (0,1)$ and $h>0$ be three real numbers and define the function $\phi(s,t)=\sqrt[4]{\mu\,\tau\,}f_+(s,t)f_-(s,t)$, where $ f_\pm$ are the traveling waves given by
\begin{eqnarray*}
f_+(s,t)&=&{\rm sn}\left(h\left(s+h^2\left(1+\mu+3\sqrt{\mu/\tau}(1+\tau)\right)t\right),\mu\right),\\
f_-(s,t)&=&{\rm sn}\left(\sqrt[4]{\mu/\tau}\, h\left(s+h^2\left(\sqrt{\mu/\tau}(1+\tau)+3(1+\mu)\right)t\right),\tau\right).
\end{eqnarray*}
Then the function $u=-2\partial_s\left(\arctanh\,\phi\right)$ is a solution of the defocusing mKdV equation 
$$\partial_tu-6u^2\partial_su+\partial^3_su=0\,.$$ 
It then follows that, using the Miura transformation (\cite{Mi}), the function defined by
\begin{equation}\label{f}
	\bending=\partial_s u+u^2\,,
\end{equation}
is a solution of the KdV equation \eqref{KdV2}.  

\begin{defn}
The solution $\bending$ of the KdV given in \eqref{f} is called the \emph{KKSH-solution with elliptic parameters $\mu,\tau\in(0,1)$ and homothetic parameter $h>0$}.  
\end{defn}

\begin{remark} \emph{Observe that if $\mu=\tau$ the KKHS-solution $\bending$ is a traveling wave solution. Therefore, from now on we assume that $\mu\neq \tau$.}
\end{remark}  

The solution $\bending$ given in \eqref{f} is periodic in $s$, with least period $\rho=4m/h K(\mu)$, if and only if there exist relatively prime integers $m$ and $n$ such that
$$\sqrt[4]{\mu}\,m\,K(\mu)=\sqrt[4]{\tau}\,n\,K(\tau)\,.$$ 
This represents the implicit equation of a curve ${\mathcal C}_{m,n}$ of $(0,1)\times (0,1)$. It turns out that this curve is the graph of a function. In fact,
$g:\tau \in (0,1)\longmapsto \sqrt[4]{\tau}\,K(\tau)\in (0,\infty)$ is a strictly increasing diffeomorphism. Therefore, $\bending$ is periodic in $s$ if and only if $\tau = \tau_{m,n}(\mu)=g^{-1}\left(m/n\sqrt[4]{\mu}\,K(\mu)\right)$.   Then, ${\mathcal C}_{m,n}$ is the graph of the function $\tau_{m,n} :\mu\in (0,1)\longmapsto \tau_{m,n}(\mu)\in (0,1)$. 

\begin{remark} \emph{Note that $g^{-1}$ can be evaluated solving numerically the Cauchy problem
$$y'=\frac{4(y-1)y^{3/4}}{(1-y)K(y)-2E(y)}\,, \quad\quad\quad y\left(\frac{K(1/2)}{\sqrt[4]{2}}\right)=1/2\,,
$$
where $E$ is the complete elliptic integral of the second kind.}
\end{remark}

\begin{defn} The KKHS-solution periodic in $s$ with $\tau=\tau_{m,n}(\mu)$ will be denoted by $\bending_{m,n}$. This solution is called the \emph{KKSH-solution with elliptic parameter $\mu\in(0,1)$, homothetic parameter $h>0$ and quantum numbers $m$ and $n$}.
\end{defn}  

The solution $\bending$ given in \eqref{f} is periodic in $t$ if and only if there exist relatively prime integers $p$ and $r$ such that
$$\sqrt[4]{\mu/\tau}\left((1+\tau)\sqrt{\mu/\tau} + 3(1+\mu)\right)p\,K(\mu)=\left(1+\mu+3(1+\tau)\sqrt{\mu/\tau}\right)r\,K(\tau)\,.
$$
Denote by ${\mathcal D}_{p,r}\subset (0,1)\times (0,1)$ the curve defined by the above implicit equation. If $m/n$ and $p/r$ are rational numbers such that
${\mathcal D}_{p,r}\cap {\mathcal C}_{m,n}\neq \emptyset$, then for every $(\mu,\tau)\in  {\mathcal D}_{p,r}\cap {\mathcal C}_{m,n}$, the  KKSH-solution $\bending$ is doubly-periodic. 

\begin{remark} \emph{Our numerical experiments strongly support the hypothesis that the set $\{(m/n,p/r)\in {\mathbb Q}^2\,\vert\, {\mathcal D}_{p,r}\cap {\mathcal C}_{m,n}\neq \emptyset\}$ has infinite countably many elements. This would imply that there are infinite countably many 1-parameter families of KKSH-solutions that are doubly periodic (in which case the homothetic parameter is the free parameter of each one of these families).}
\end{remark}

\subsection{KKSH-Solutions of the LIEN Flow} 

Let $m$ and $n$ be two relatively prime integers and consider the KKSH-solution $\bending_{m,n}$, which is periodic in $s$, with elliptic parameter $\mu\in(0,1)$, homothetic parameter $h>0$, and quantum numbers $m$ and $n$. Denote by $\gamma_{m,n}$ the corresponding solution of the LIEN flow \eqref{LIEN2} and assume that $\gamma_{m,n}$ is normalized by $F_+(0,0)=F_-(0,0)={\rm Id}$, where $(F_+,F_-)$ is the spinor frame field along $\gamma_{m,n}$.

According to Theorem \ref{induced}, the solution $\gamma_{m,n}$ of the LIEN flow \eqref{LIEN2} can be constructed as $\gamma_{m,n}=E_1E_{-1}^{-1}$, where $E_\lambda$, $\lambda=-1,1$, are the extended frames of $\bending_{m,n}$ with spectral parameters $\lambda=-1,1$. Observe that from the proof of Theorem \ref{induced}, it follows that $(E_1,E_{-1})$ is the spinor frame field along $\gamma_{m,n}$. Hence, with the current notation, $E_1=F_+$ and $E_{-1}=F_-$.

Therefore, from a theoretical point of view, to evaluate the evolution, we need to obtain the extended frames $F_\pm$ of $\bending_{m,n}$, by solving \eqref{dE}. That is, in a first step we need to solve the linear systems
\begin{eqnarray}
	\frac{dA_+}{dt}&=&A_+\begin{pmatrix} -\partial_s \bending&-\partial_s^2 \bending+2\bending^2-2\bending-4\\
2\bending-4&\partial_s \bending\end{pmatrix}_{\lvert_{(0,t)}}\,,\label{system11}\\
\frac{dA_-}{dt}&=&A_-\begin{pmatrix} -\partial_s \bending&-\partial_s^2 \bending+2\bending^2+2\bending-4\\
	2\bending+4&\partial_s \bending\end{pmatrix}_{\lvert_{(0,t)}}\,,\label{system12}
\end{eqnarray}
with the initial conditions $A_+(0)=A_-(0)={\rm Id}$. (For simplicity, we are simply writing $\bending$ to denote $\bending_{m,n}$.) The second step is to solve, for every $t\in \R$, the linear systems
\begin{eqnarray}
	\frac{dF_+}{ds}&=&F_+ \begin{pmatrix} 0& \bending+ 1\\
1&0\end{pmatrix}_{\lvert_{(s,t)}}\,,\label{system21}\\
\frac{dF_-}{ds}&=&F_- \begin{pmatrix} 0& \bending- 1\\
	1&0\end{pmatrix}_{\lvert_{(s,t)}}\,,\label{system22}
\end{eqnarray}
with initial conditions $F_+(0,t)=A_+(t)$ and $F_-(0,t)=A_-(t)$, respectively. Then, $\gamma_{m,n}(s,t)=F_+(s,t)F_-^{-1}(s,t)$ it the normalized solution of the LIEN flow with bending $\bending_{m,n}$.   

\begin{remark}\emph{Note that \eqref{system21}-\eqref{system22} is equivalent to the four Hill's equations
\begin{equation}\label{Hill}
	x_{\pm}''=\left(\bending_{m,n}\pm 1\right)x_{\pm}\,, \quad\quad\quad y_{\pm}''=\left(\bending_{m,n}\pm 1\right)y_{\pm}\,.
\end{equation}
Even though we do not have exact solutions, the numerical integration does not presents problems. However, the system \eqref{system11}-\eqref{system12} may present critical issues from a numerical point of view. Nonetheless, if the two monodromy matrices are non-trivial and diagonalizable, the problem can be overcome by first solving \eqref{system21}-\eqref{system22} with initial conditions $F_{\pm}(0,t)={\rm Id}$ and then determining, just with algebraic manipulations, the solutions of \eqref{system11}-\eqref{system12}. In particular, it is essential to determine the type of orbit of the evolution (see Subsection \ref{OrbitType}).}
\end{remark}

Keeping in mind Proposition \ref{CauchyP}, in order to determine the type of orbit of $\gamma_{m,n}$ it suffices to consider the initial curve $\widehat{\gamma}_{m,n}(s)=\gamma_{m,n}(s,0)$. This curve and its spinor frame field can be evaluated solving numerically the Hill's equations (\ref{Hill}) with $t=0$ and initial conditions $x_{\pm}(0)=1, x'_{\pm}(0)=0$,  $y_{\pm}(0)=0$, and $y'_{\pm}(0)=0$. Let $(\widehat{F}_+,\widehat{F}_-)$ be the spinor frame field along $\widehat{\gamma}_{m,n}$ and consider the monodromy
$\widehat{M}=(\widehat{F}_+(\rho),\widehat{F}_{-}(\rho))$, where $\rho=4m/h\,K(\mu)$ is the least period of the bending of $\widehat{\gamma}_{m,n}$ (see Subsection 6.1). We next compute the invariant ${\rm I}_\pm={\rm tr}\widehat{F}_\pm(\rho)^2-4$. Then, $\widehat{F}_\pm(\rho)$ is elliptic if ${\rm I}_\pm<0$, hyperbolic if ${\rm I}_\pm>0$, and parabolic if ${\rm I}_\pm=0$ and $\widehat{F}_\pm(\rho)\neq \pm{\rm Id}$. In particular, if  ${\rm I}_\pm\neq 0$ the monodromy matrices are diagonalizable.
 
\begin{ex} Fix the real numbers $m=1$, $n=6$ and $h=2$, and consider the $1$-parameter family $\{\bending_{1,6}\}_\mu$ of KKSH-solutions with elliptic parameter $\mu\in(0,1)$, homothetic parameter $h=2$, and quantum numbers $m=1$ and $n=6$. Recall that these solutions are periodic in $s$ and, in this case, the least period of $\bending_{1,6}$ is $\rho=2K(\mu)$.
	
Since $\rho$ depends on $\mu$, the invariants ${\rm I}_\pm$ introduced above are functions ${\rm I}_{\pm}:\mu \in (0,1) \longmapsto {\rm I}_\pm(\mu)\in\R$. In Figure \ref{INIP} we represent the graphs of these functions. From these representations we deduce that the type of orbit of the KKSH-solutions $\gamma_{1,6}$ of the LIEN flow \eqref{LIEN2} with bending $\bending_{1,6}$ is $(H,E)$ (for every $\mu\in(0,1)$).

\begin{figure}[h]
	\begin{center}
		\makebox[\textwidth][c]{
			\includegraphics[height=4cm,width=6cm]{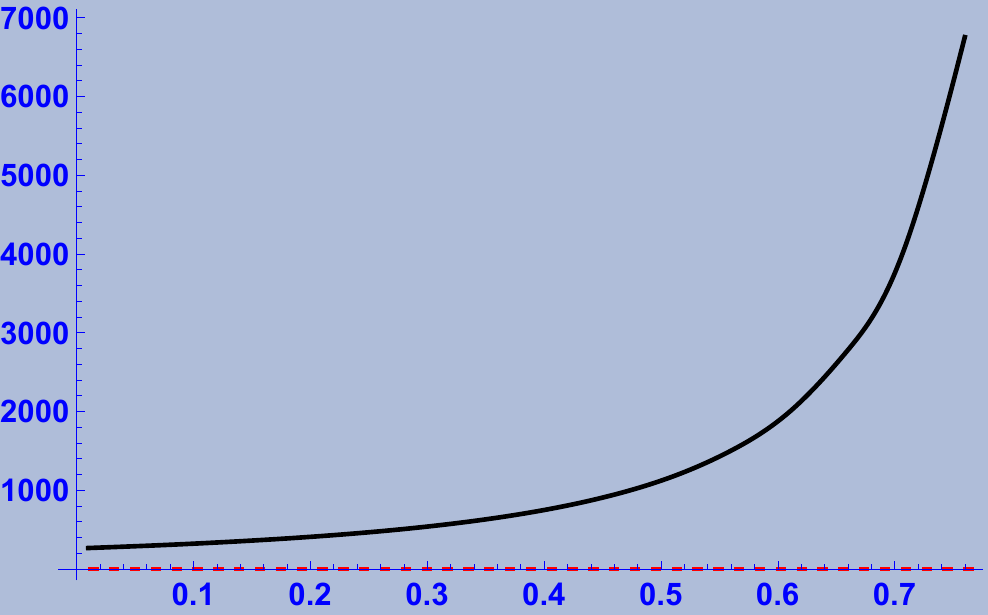}\quad\quad
			\includegraphics[height=4cm,width=6cm]{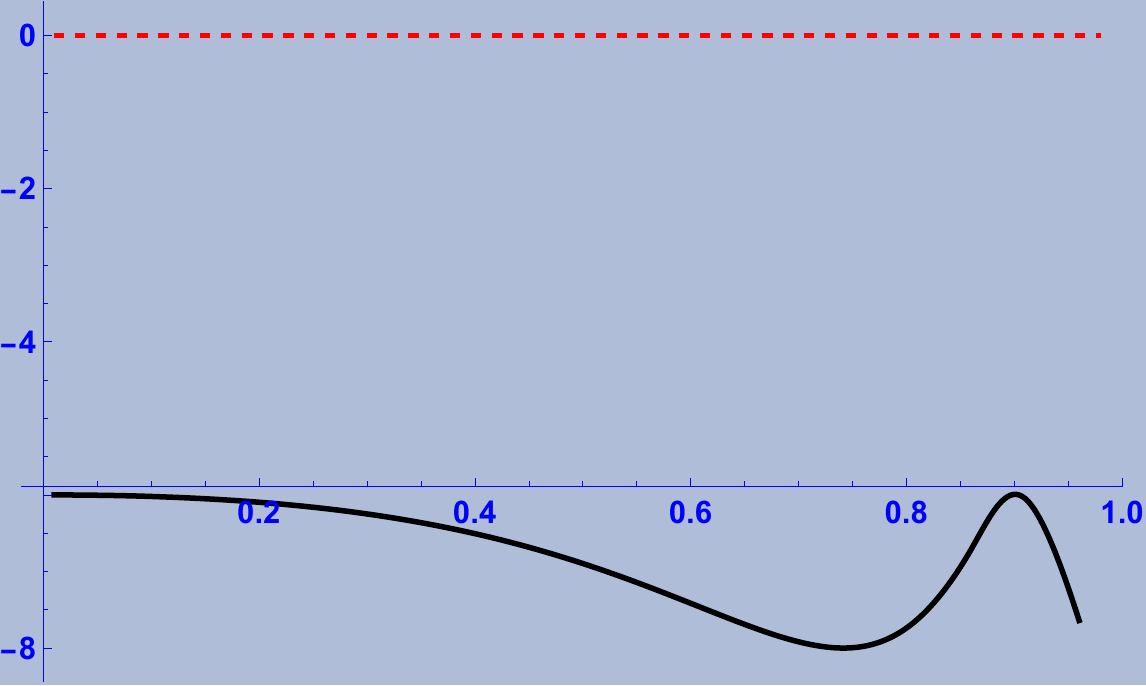}
		}
		\caption{\small{Left: The graph of the function ${\rm I}_+$. Right: The graph of the function ${\rm I}_-$. The domain illustrated in the images is $(0.01,0.98)$.}} \label{INIP}
	\end{center}
\end{figure}

Next, we find the elliptic parameter $\mu_*\in(0,1)$ such that the eigenvalues of the monodromy matrix $\widehat{F}_-(\rho)$ are $e^{i\pm 2\pi/3}$. To this end we evaluate the function
$${\mathtt p}_{2/3} : \mu\in (0,1)\longmapsto \frac{1}{2}{\mathfrak R}\left({\rm tr}\widehat{F}_-(\rho)+\sqrt{{\rm I}_-}\right)-\cos(2\pi/3)\in\R\,,$$
and compute numerically its unique zero, namely, $\mu_*\simeq 0.61500934$ (see Figure \ref{TAVOLA3}).

 \begin{figure}[h]
	\begin{center}
		\makebox[\textwidth][c]{
			\includegraphics[height=4cm,width=6cm]{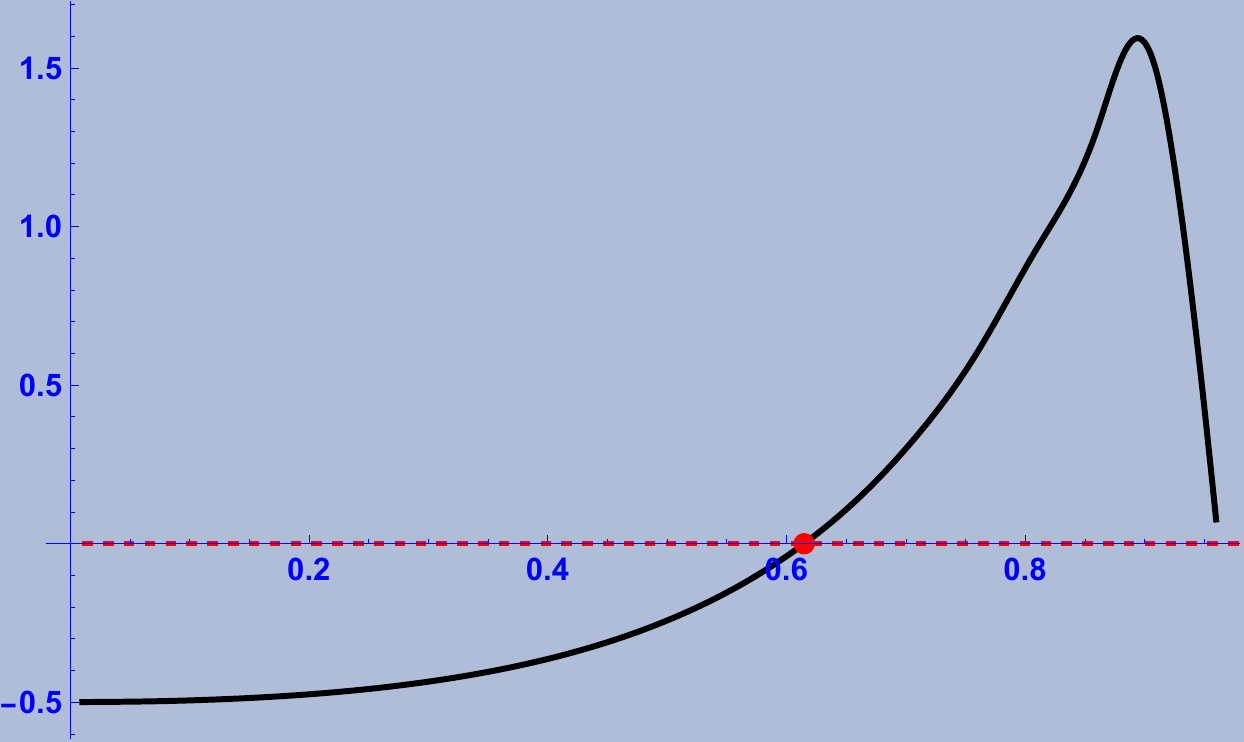}}
		\caption{\small{The graph of the function ${\mathtt p}_{2/3}$. The only zero of this function is $\mu_*\simeq 0.61500934$ (in red).}} \label{TAVOLA3}
	\end{center}
\end{figure}

We now fix the elliptic parameter $\mu=\mu_*$ and let $\bending_{1,6}$ be the corresponding KKSH-solution of the KdV equation (for this particular value $\mu_*\in(0,1)$). This function $\bending_{1,6}$ is periodic in $s$ with least period $\rho\simeq 3.93225$. In Figure \ref{KKSH} we show the graphs of the functions $\bending_{1,6}(s,t_j)$ over the domain $[0,\rho]$, for the values $t_0=0$, $t_1\simeq 0.537285$, $t_2= 2t_1$ and $t_3= 3t_1$.
 
\begin{figure}[h]
 	\begin{center}
 		\makebox[\textwidth][c]{
 			\includegraphics[height=4cm,width=6cm]{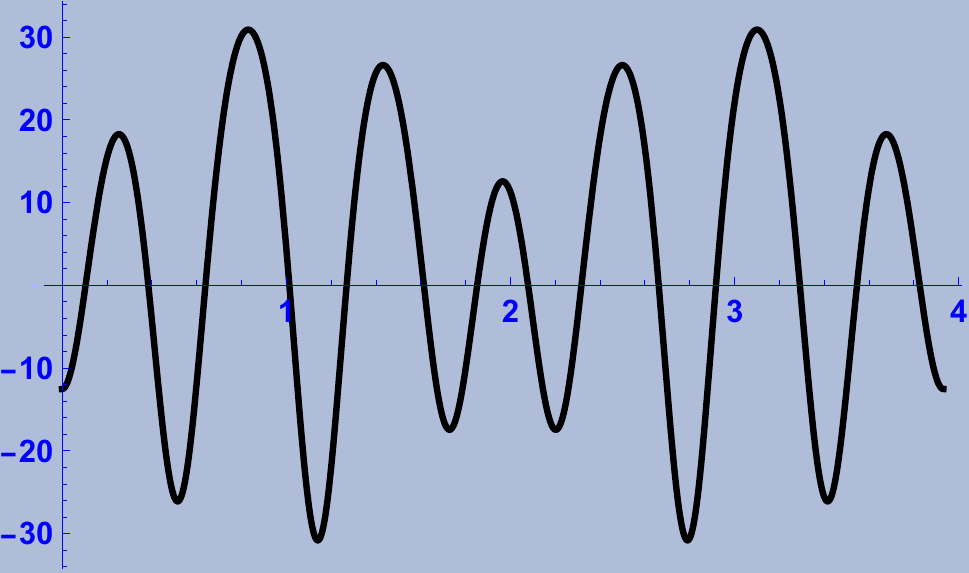}\quad\quad
 			\includegraphics[height=4cm,width=6cm]{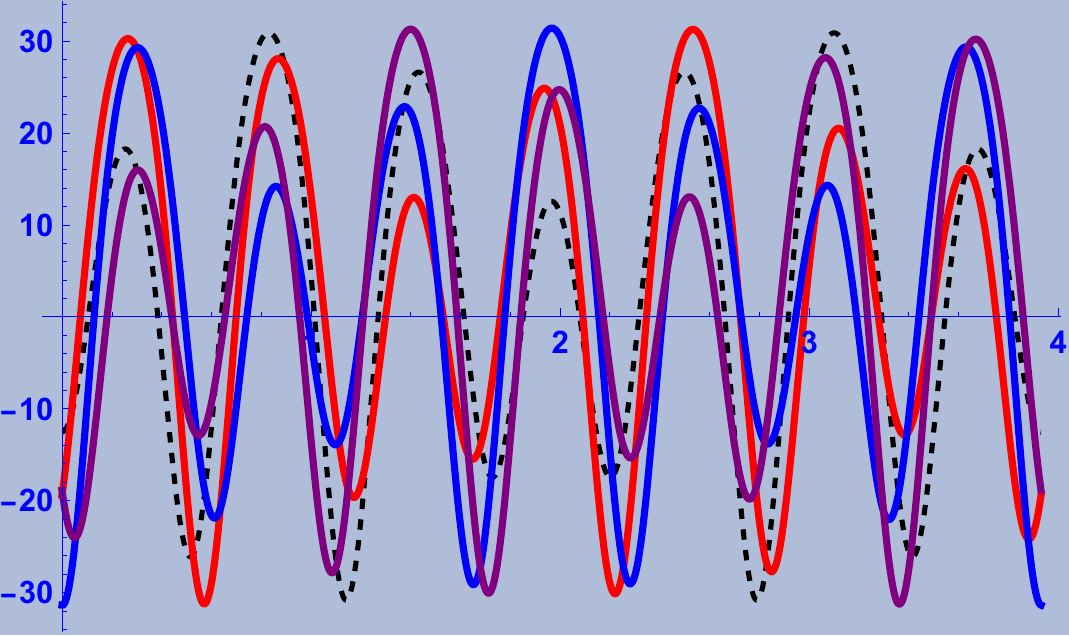}
 		}
 		\caption{\small{Left: The graph of the function $\bending_{1,6}(s,0)$. Right: The graphs of the functions $\bending_{1,6}(s,t_j)$, for $j=1,2,3$, (in red, blue, and purple, respectively.)}} \label{KKSH}
 	\end{center}
 \end{figure} 

Denote by $\gamma_{1,6}$ the solution of the LIEN flow \eqref{LIEN2} with bending $\bending_{1,6}$ (for the fixed value $\mu=\mu_*$ of the elliptic parameter) and by $(\eta_+,\eta_-)$ the pair of star-shaped cousins associated with $\gamma_{1,6}$. Figures \ref{EV1}, \ref{EV2}, \ref{EV3} and \ref{EV4} depict the null curves $\gamma_{1,6}(s,t_j)$ and the associated pair of cousins. The part of the null curves colored in red are the arcs $\gamma_{1,6}(s,t_j)$ where $s\in[-\rho/2,\rho/2]$, while the points colored in cyan are $\gamma_{1,6}(-\rho/2,t_j)$ and $\gamma_{1,6}(\rho/2,t_j)$. The parts colored in purple are $\gamma_{1,6}(s,t_j)$ with $s\in(-\infty, -\rho/2]$ and $s\in[\rho/2,\infty)$. When $s>\rho/2$, the curve $\gamma_{1,6}(s,t_j)$ is trapped into a small tubular neighborhood of an asymptotic null curve ${\mathcal C}^+$ of the ideal boundary (see Remark \ref{CE}). Similarly, when  $s<-\rho/2$,   $\gamma_{1,6}(s,t_j)$ is trapped into a small tubular neighborhood of a second asymptotic null curve ${\mathcal C}^-$ of the ideal boundary. On the other hand, the arcs $\gamma_{1,6}(s,t)$ where $s\in[-\rho/2,\rho/2]$ undergo non-trivial deformations. The trajectories of the evolving curves $\gamma_{1,6}$ have a stabilizer of infinite order, spanned by the monodromy $\widehat{M}=(\widehat{F}_+(\rho),\widehat{F}_-(\rho))$, where $\widehat{F}_-(\rho)$ is a clockwise rotation of $2\pi/3$ around the origin and $\widehat{F}_+(\rho)$ is the hyperbolic element
$$\widehat{F}_+(\rho)\simeq \begin{pmatrix}32.13972944617541&31.723219516279162\\32.5231&32.1327 \end{pmatrix}\in \SL\,,
$$
with eigenvalues $\zeta_1\simeq 64.26$ and $\zeta_2=1/\varrho_1$.
\end{ex}

 \begin{figure}[h]
	\begin{center}
		\makebox[\textwidth][c]{
			\includegraphics[height=4cm,width=4cm]{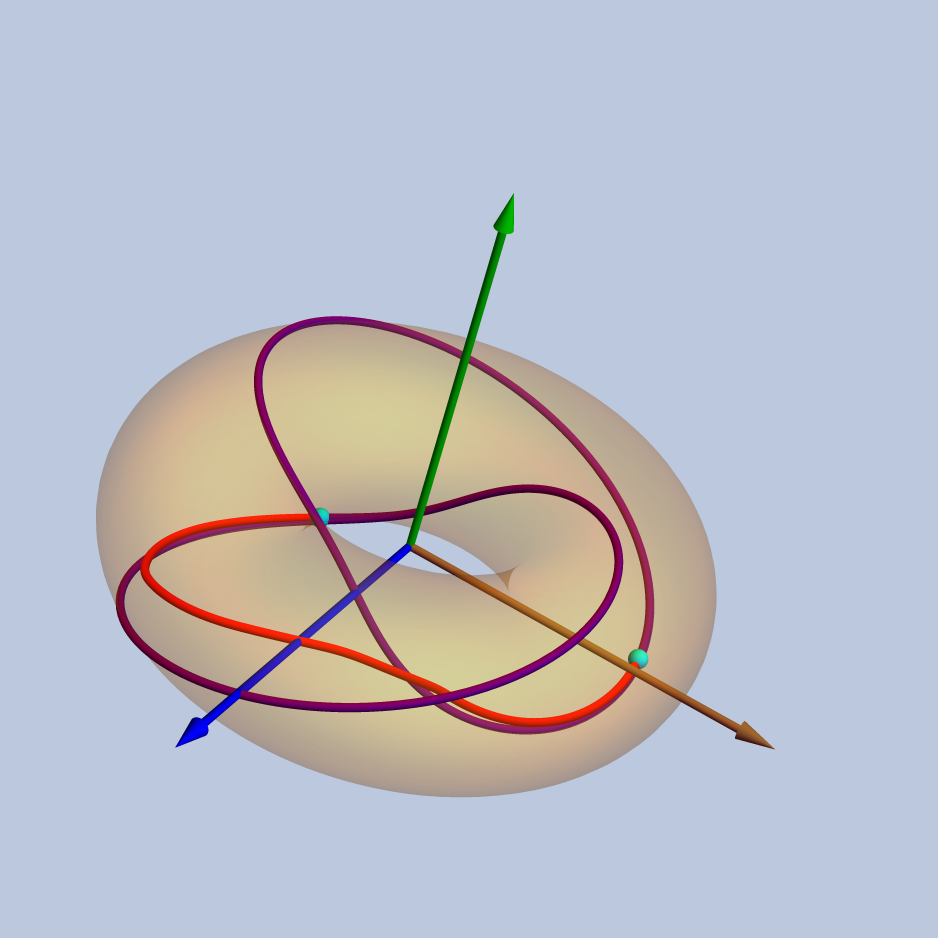}\quad\quad\quad
			\includegraphics[height=4cm,width=4cm]{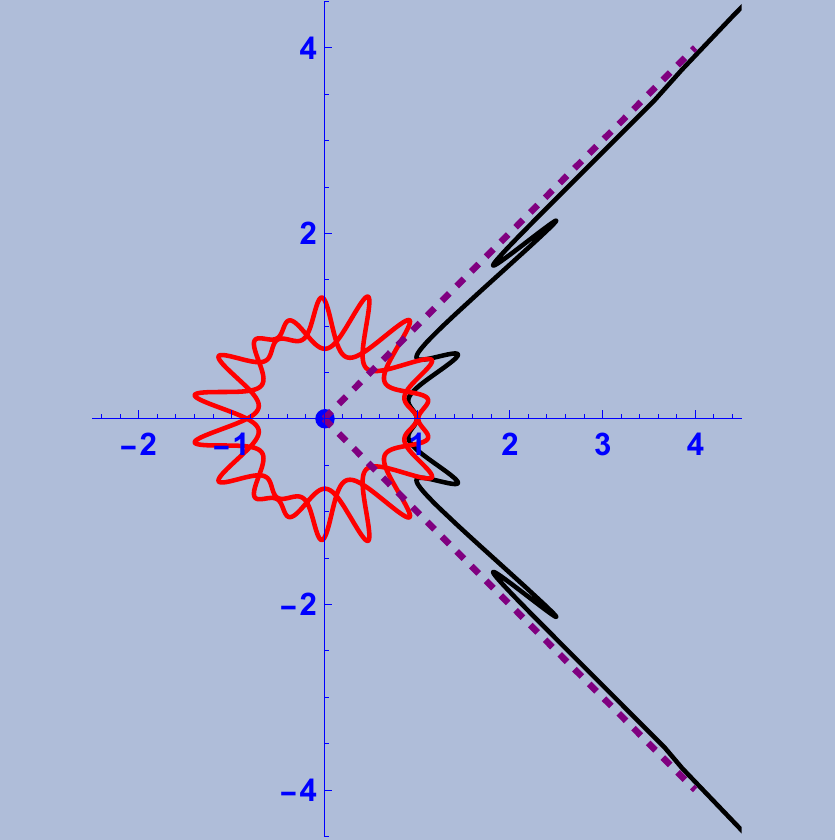}
		}
		\caption{\small{Left: The null curve $\gamma_{1,6}(s,t_0)$. Right: The pair of star-shaped cousins $(\eta_+,\eta_-)$ in black and red, respectively.}} \label{EV1}
	\end{center}
\end{figure}

 \begin{figure}[h]
	\begin{center}
		\makebox[\textwidth][c]{
			\includegraphics[height=4cm,width=4cm]{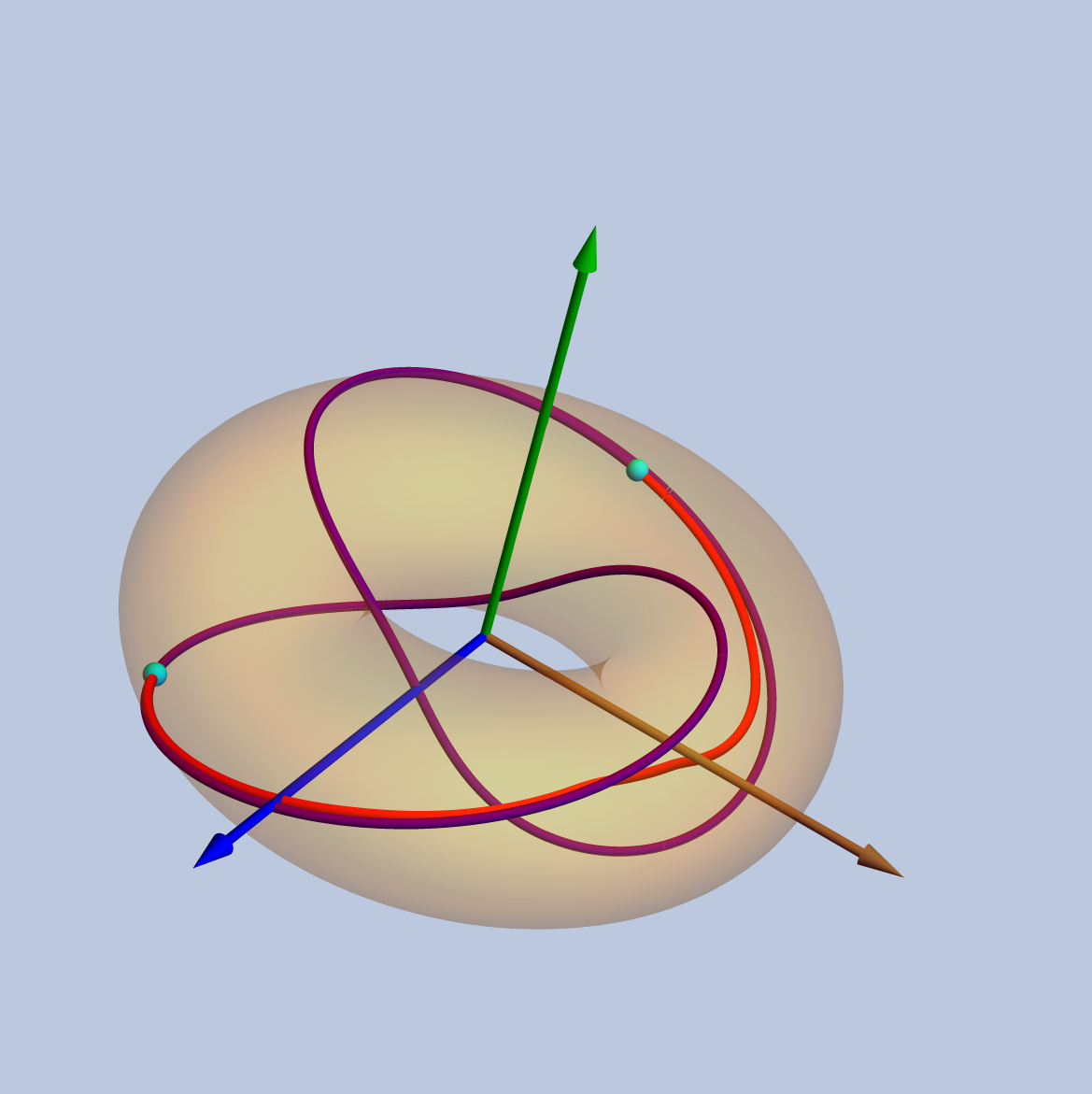}\quad\quad\quad
			\includegraphics[height=4cm,width=4cm]{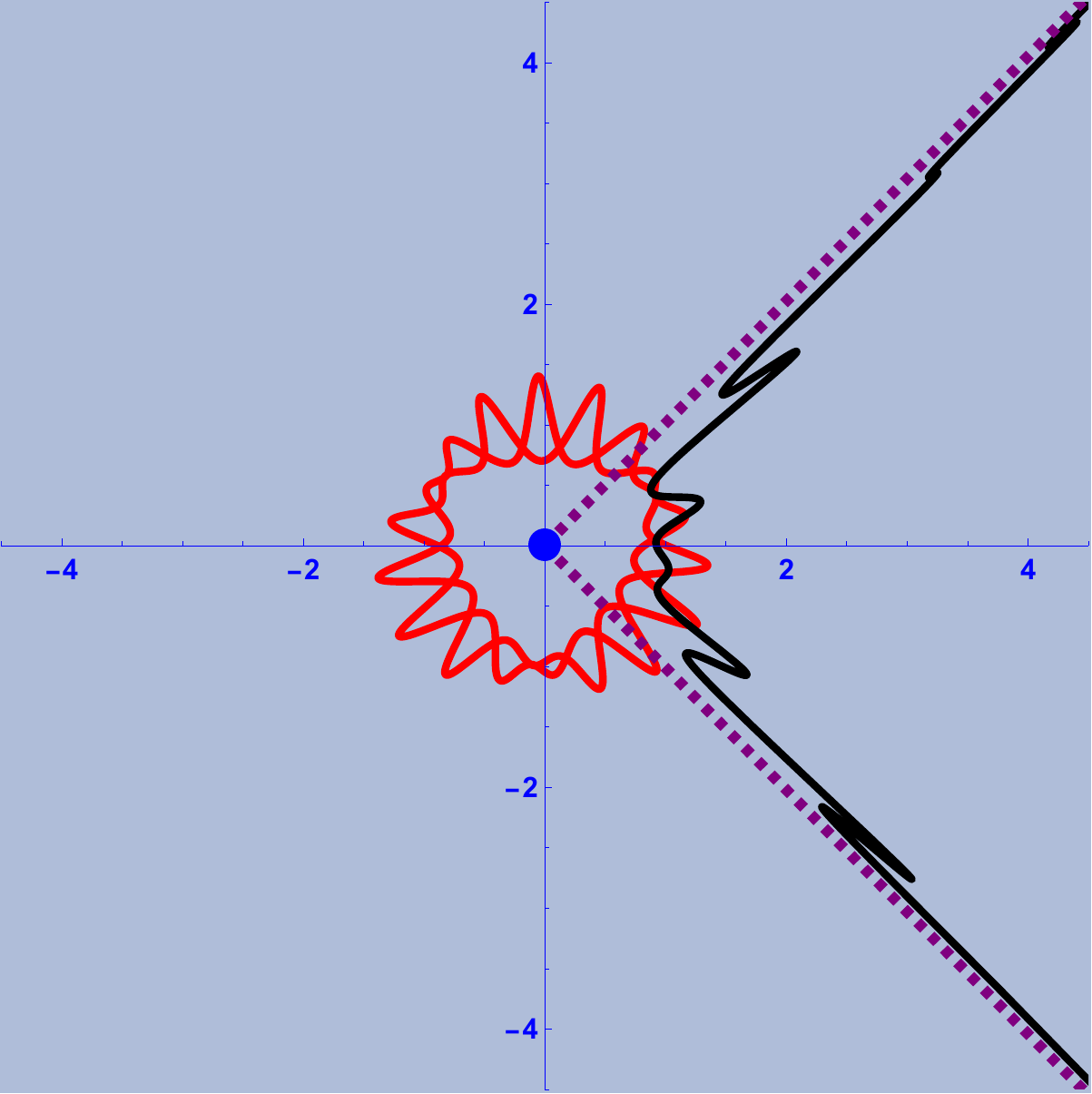}
		}
		\caption{\small{Left: The null curve $\gamma_{1,6}(s,t_1)$. Right: The pair of star-shaped cousins $(\eta_+,\eta_-)$ in black and red, respectively.}} \label{EV2}
	\end{center}
\end{figure}

 \begin{figure}[h]
	\begin{center}
		\makebox[\textwidth][c]{
			\includegraphics[height=4cm,width=4cm]{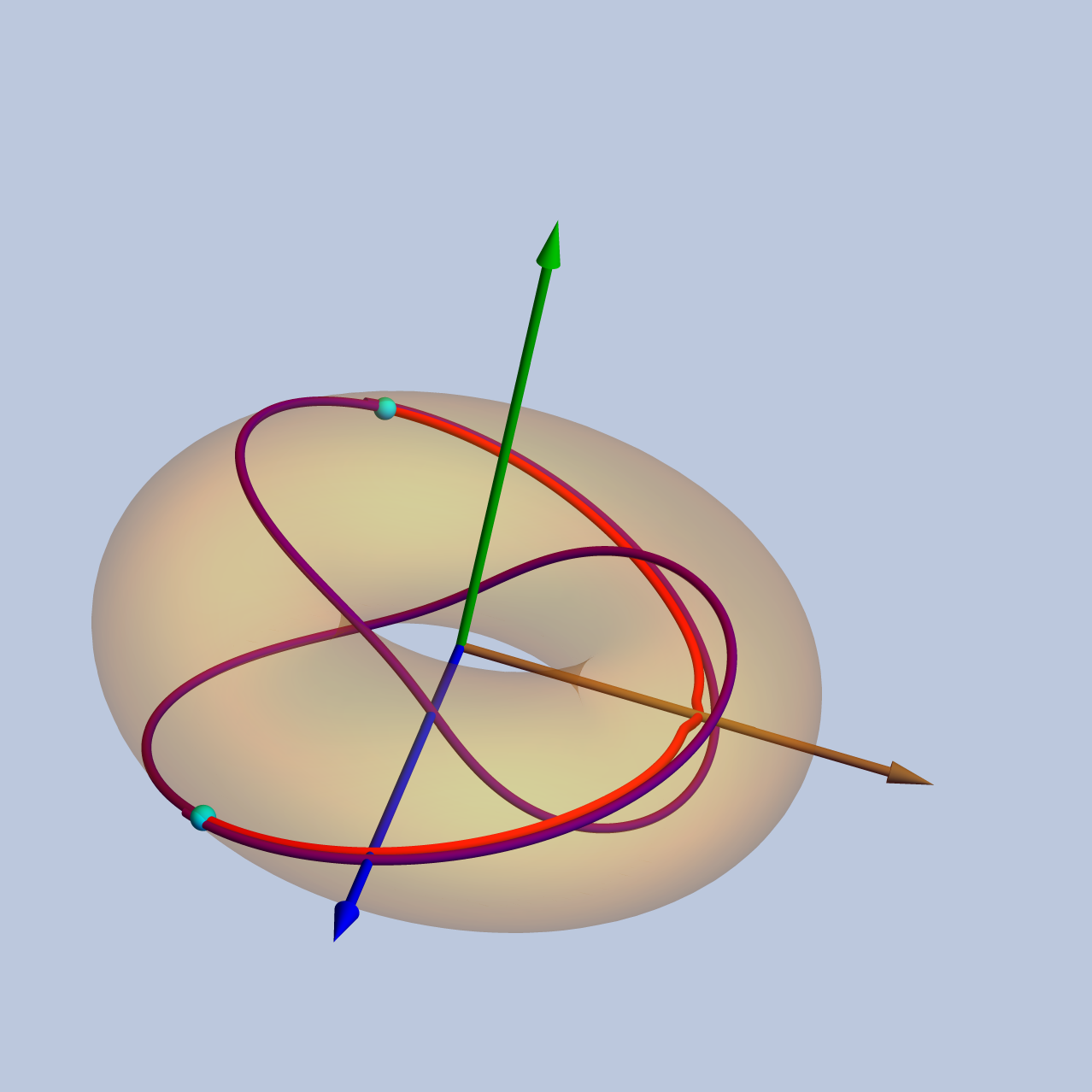}\quad\quad\quad
			\includegraphics[height=4cm,width=4cm]{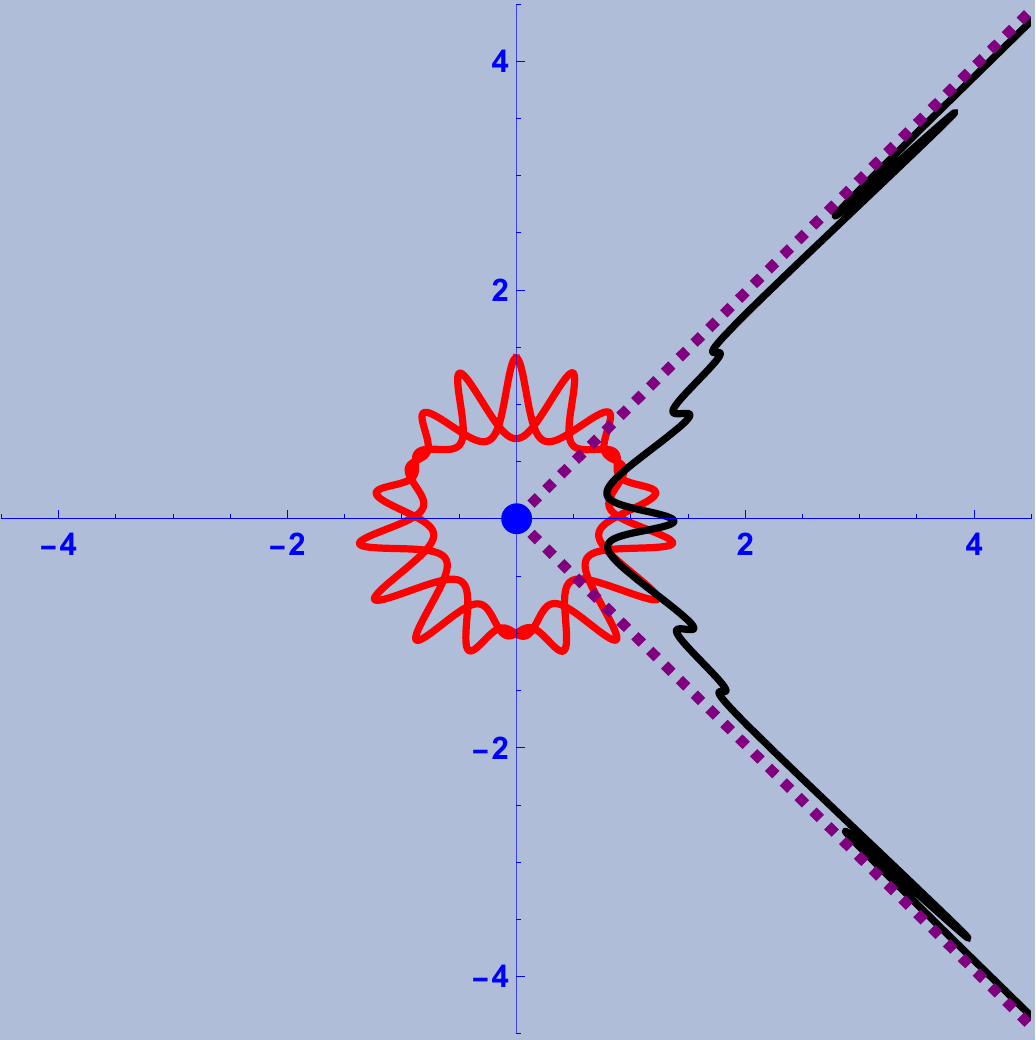}
		}
		\caption{\small{Left: The null curve $\gamma_{1,6}(s,t_2)$. Right: The pair of star-shaped cousins $(\eta_+,\eta_-)$ in black and red, respectively.}} \label{EV3}
	\end{center}
\end{figure}

 \begin{figure}[h]
	\begin{center}
		\makebox[\textwidth][c]{
			\includegraphics[height=4cm,width=4cm]{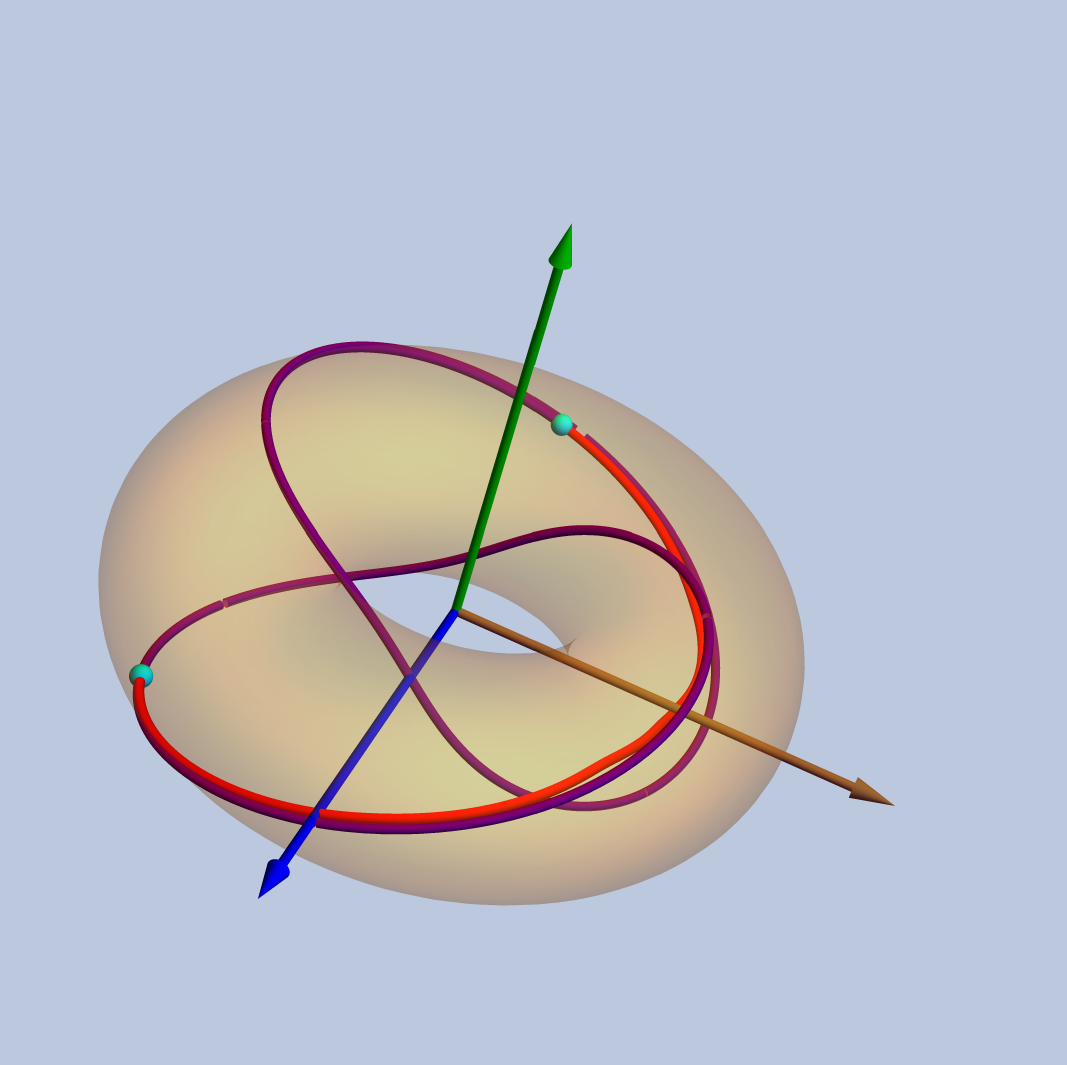}\quad\quad\quad
			\includegraphics[height=4cm,width=4cm]{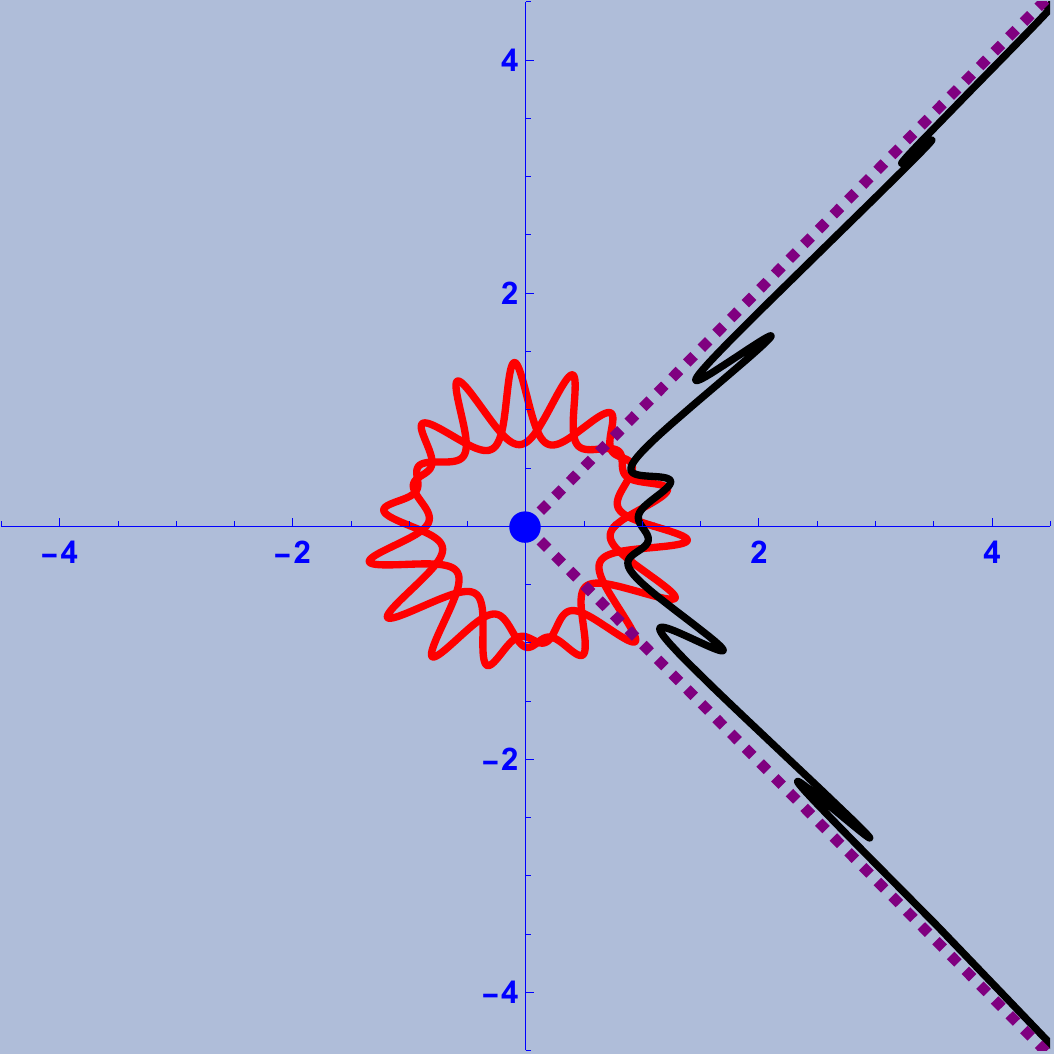}
		}
		\caption{\small{Left: The null curve $\gamma_{1,6}(s,t_3)$. Right: The pair of star-shaped cousins $(\eta_+,\eta_-)$ in black and red, respectively.}} \label{EV4}
	\end{center}
\end{figure} 

\section{Discussion}

To summarize, we have shown that in the context of Lorentzian geometry in $\AdS$, there is a flow (the LIEN flow) on null curves that induces the bending evolution by the KdV equation. Using the specific geometry of $\AdS$ we have been able to relate the LIEN flow to the Pinkall flows for star-shaped curves in the centro-affine plane. 

We have carried out a detailed analysis of curves which are stationary (ie. whose flows are congruent with respect to the restricted automorphism group of $\AdS$ to the initial curve), identifying closure conditions and obtaining a complete description of periodic stationary curves via an analysis of the Floquet spectrum of the Lam\'e equation of order one. We have also showed that the periodic stationary curves can be explicitly described in terms of Jacobi's elliptic functions and special types of Heun functions.

Next, we have investigated a $3$-parameter family of solutions of the LIEN flow. Contrary to the case of stationary solutions, in this case the Hill's equation arising when studying the corresponding evolution of null curves cannot be explicitly solved in terms of known special functions. However, we have carried out an analysis of these solutions based on the numerical solutions of such Hill's equation.

These results naturally suggest further questions and directions for research. For instance, one can investigate if the finite-gap solutions of the Pinkall flows obtained by Calini, Ivey and Mar\'i-Beffa (\cite{CIB2}) can be used to obtain finite gap solutions of the LIEN flow.

In addition, in Appendix B we will see that, in fact, there are flows (higher order LIEN flows) on null curves that induce bending evolution by any integrable PDE in the KdV hierarchy. This hierarchy allows a bi-Hamiltonian formulation and, hence, it is reasonable to ask if it is possible to extend to the context of null curves in $\AdS$ the bi-Hamiltonian structure on the space of star-shaped curves or on the space of null curves in Lorentz-Minkowski $(2+1)$-space studied by Terng and Wu (\cite{TW1}), Tabachnikov (\cite{T}), and Amor, Gim\'enez and Lucas (\cite{AGL}).

Finally, as a completely integrable PDE the KdV equation has a rich structure which also includes a B\"{a}cklund transformation (\cite{TU}) which generates new solutions from old. In a future paper, the authors will use the Tabachnikov's transformation on star-shaped curves (\cite{T}) to find a geometric transformation on null curves that corresponds to the B\"{a}cklund transformation for the respective bendings.

\section*{Appendix A. Null Curves with Constant Bending}

Let $\gamma:J=\R\longrightarrow\AdS$ be a null curve with constant bending $\bending$. Although our curves are contained in the interior of the torical model of $\AdS$, they may approach asymptotically the ideal boundary of this model, that is, $\partial\mathbb{T}$ (see Remark \ref{CE}).

Let $(\eta_+,\eta_-)$ be the pair of cousins associated with the null curve $\gamma$ with constant bending $\bending$. From Theorem \ref{relation}, it follows that the curvatures $\curvature_\pm$ of $\eta_\pm$ are also constant and so, we can explicitly solve the spinorial Frenet-type equations of $\gamma$, \eqref{dF+} and \eqref{dF-}. Consequently, we can obtain the explicit parameterizations of $\eta_\pm$ by taking the first column vector of $F_\pm$, respectively. Moreover, depending on the value of $\bending$ and, hence, on the values of $\curvature_\pm$, we have different possible types for the star-shaped curves $\eta_\pm$ as well as different possible combinations for the pair of cousins associated with $\gamma$. Indeed, there are five possible cases:
\begin{enumerate}
	\item Case $\bending<-1$. This case was explicitly described in Example \ref{example}. We recall here that the pair of cousins is composed by two ellipses while the null curve $\gamma$ itself is the orbit of a $1$-parameter group of isometries of type $(E,E)$. Hence, the components of the spinor frame field $F_\pm$ of $\gamma$ are periodic. Moreover, when their least periods are commensurable the null curve $\gamma$ is closed. See Figures \ref{F1} and \ref{F2}.
	\item Case $\bending=-1$. In this case, the pair of cousins associated with $\gamma$ consists of a line $\eta_+$ and an ellipse $\eta_-$. The null curve $\gamma$ is the orbit of a $1$-parameter group of isometries of type $(P,E)$. Consequently, the null curve $\gamma$ cannot be closed but, in the torical model for $\AdS$, it tends asymptotically to a null curve of the ideal boundary $\partial\mathbb{T}$. See Figure \ref{F3}.
	\begin{figure}[h]
		\begin{center}
			\makebox[\textwidth][c]{
				\includegraphics[height=4cm,width=4cm]{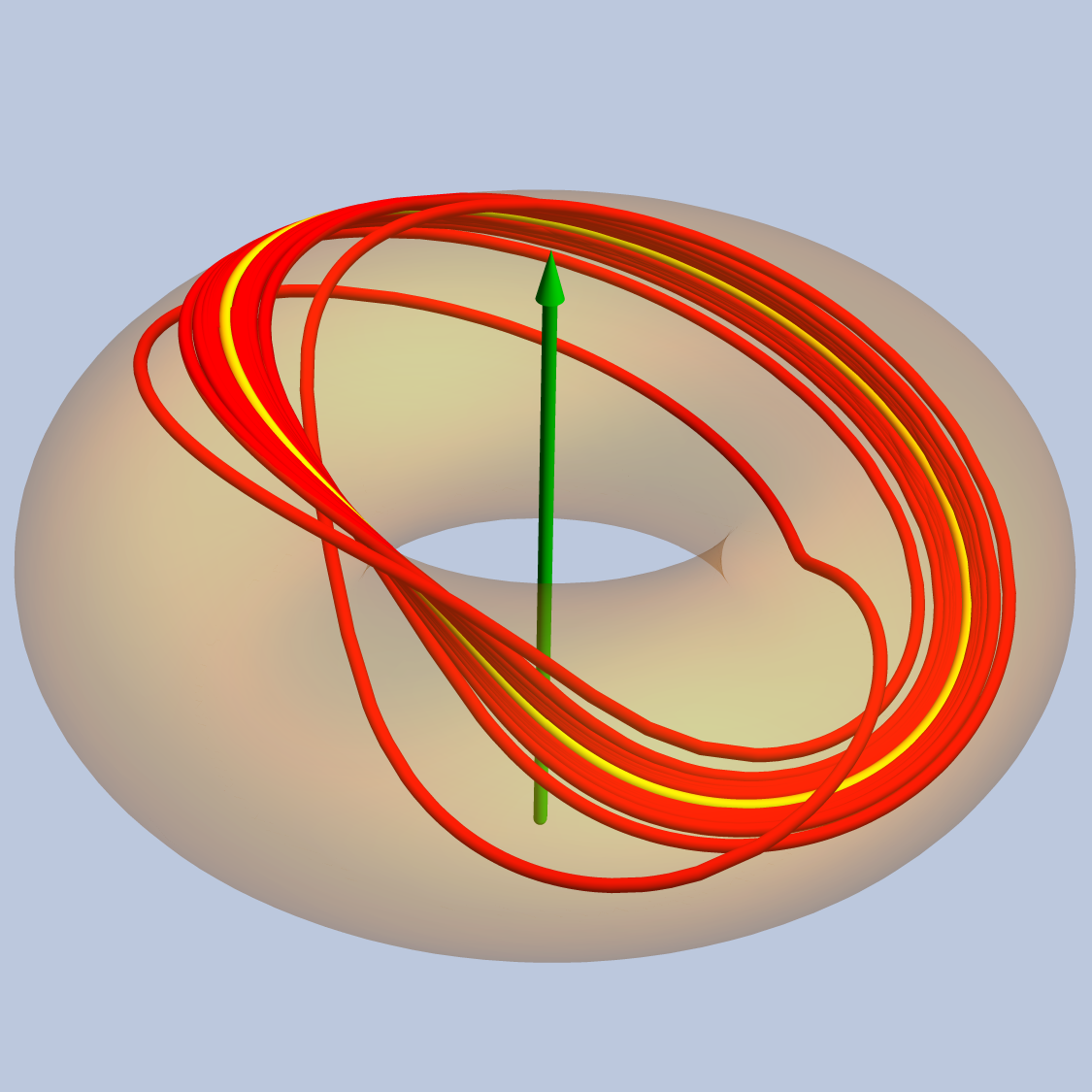}\quad\quad\quad
				\includegraphics[height=4cm,width=4cm]{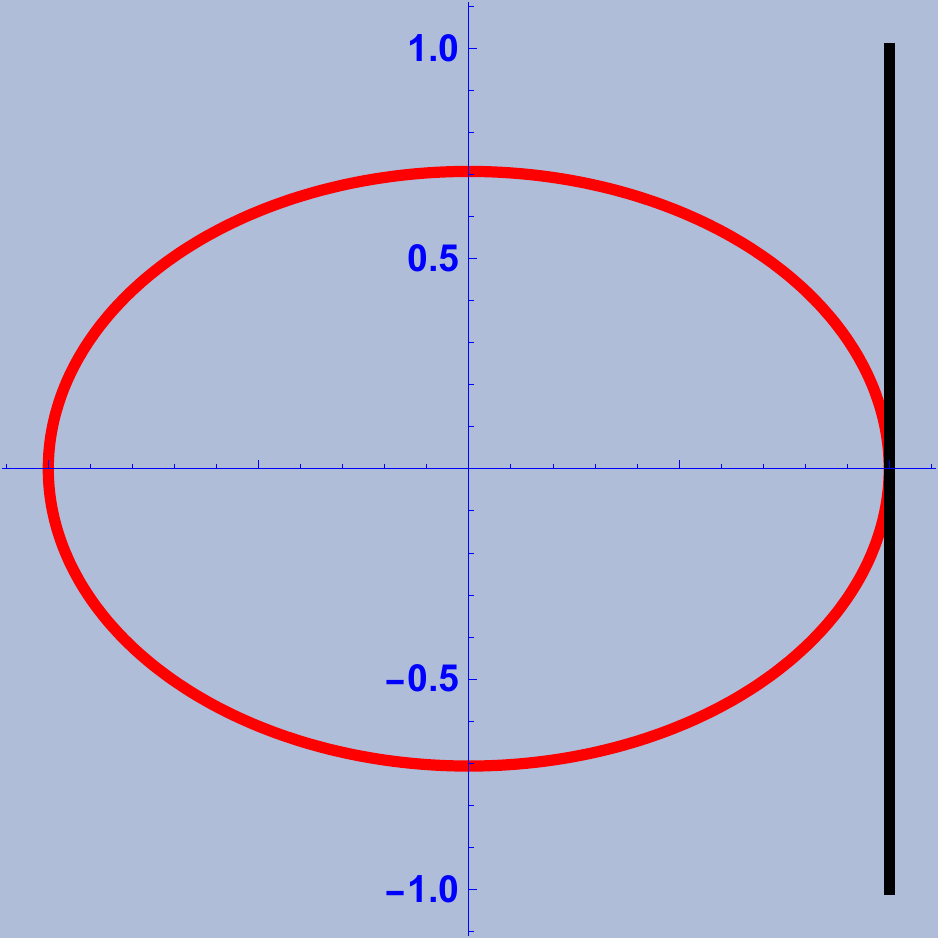}
			}
			\caption{\small{Left: The null curve $\gamma$ with constant bending $\bending=-1$ in the torical model for $\AdS$ (in red) and the limiting null curve (in yellow). Right: The pair of star-shaped cousins $(\eta_+,\eta_-)$ in black and red, respectively.}} \label{F3}
		\end{center}
	\end{figure}
	\item Case $-1<\bending<1$. The pair of cousins consists of a branch of an hyperbola $\eta_+$ and an ellipse $\eta_-$. The null curve $\gamma$ is the orbit of a $1$-parameter group of isometries of type $(H,E)$ and so it cannot be closed. It tends asymptotically to two null curves of the boundary $\partial\mathbb{T}$. Unlike in the case $\bending=-1$, the limiting curves are different. See Figure \ref{F4}.
	\begin{figure}[h]
		\begin{center}
			\makebox[\textwidth][c]{
				\includegraphics[height=4cm,width=4cm]{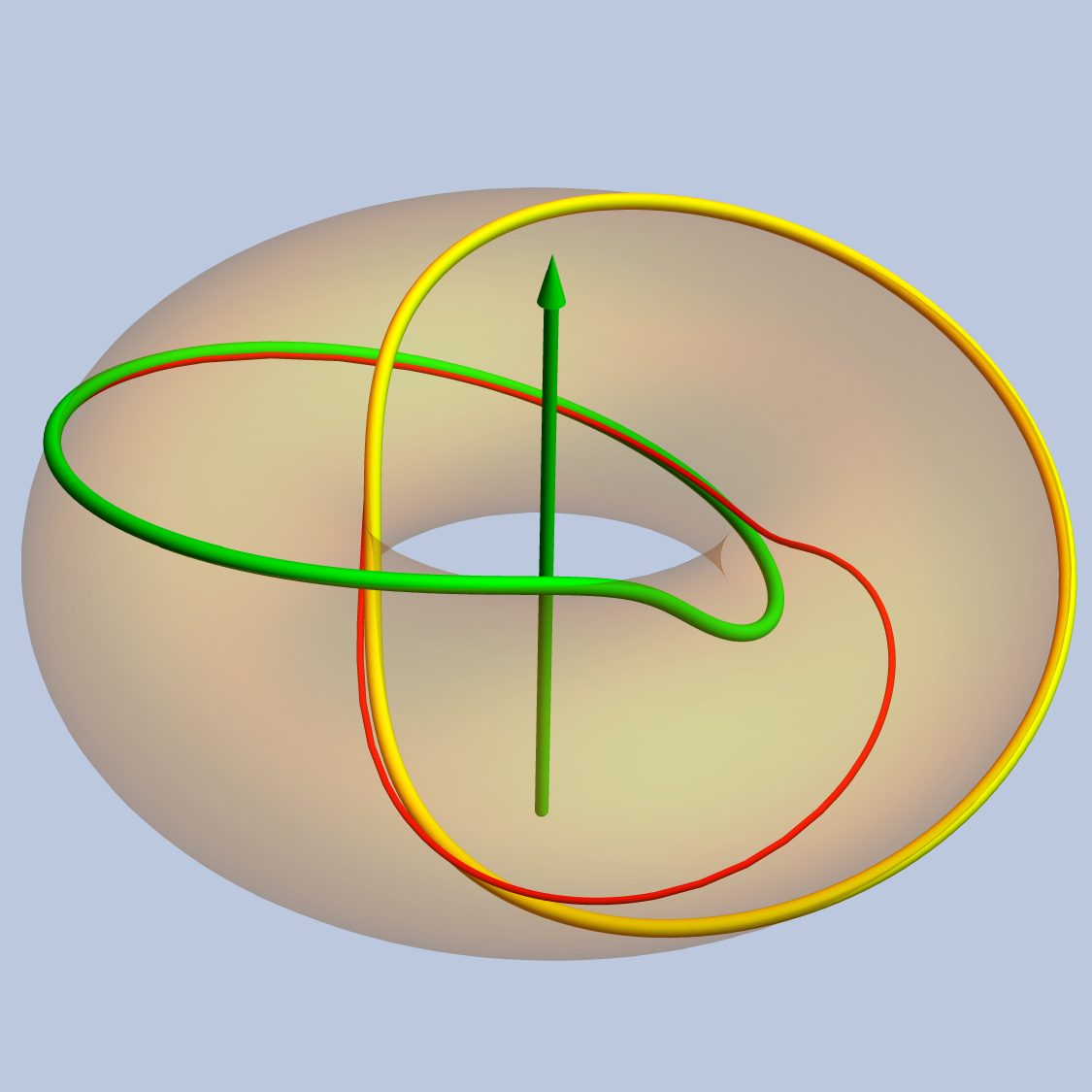}\quad\quad\quad
				\includegraphics[height=4cm,width=4cm]{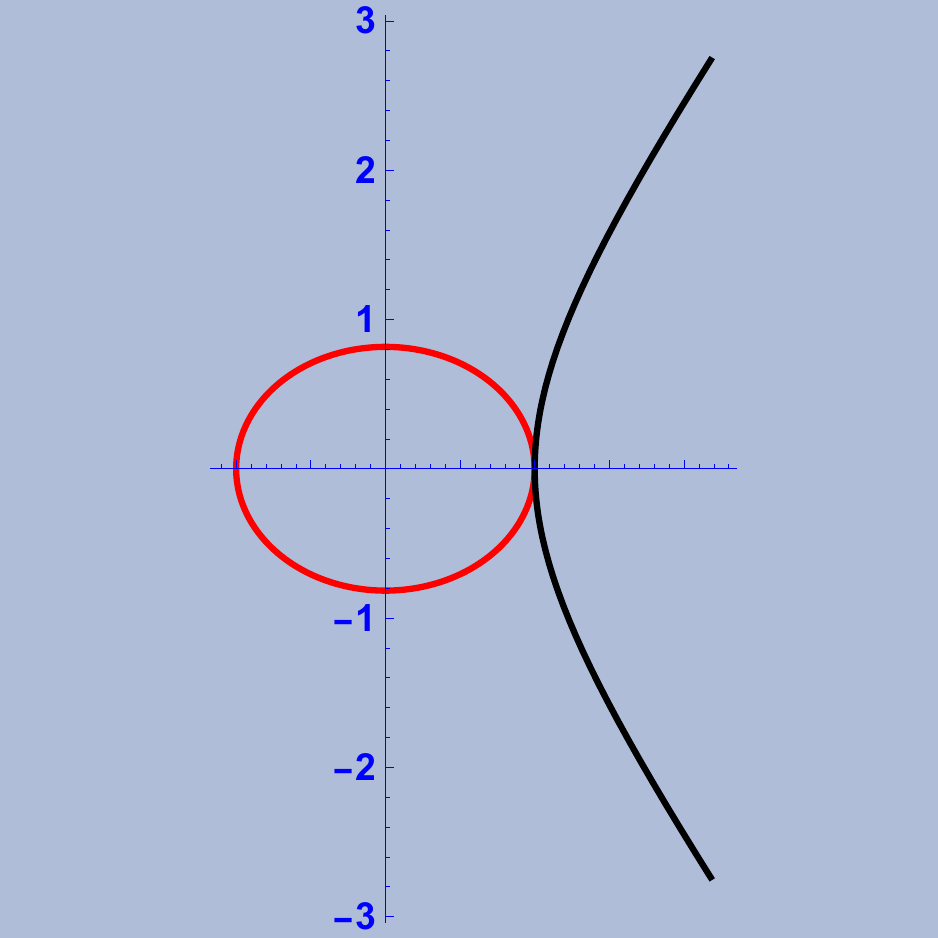}
			}
			\caption{\small{Left: The null curve $\gamma$ with constant bending $\bending=-1/2$ in the torical model for $\AdS$ (in red) and the two limiting null curves (in yellow and green, respectively). Right: The pair of star-shaped cousins $(\eta_+,\eta_-)$ in black and red, respectively.}} \label{F4}
		\end{center}
	\end{figure}
	\item Case $\bending=1$. The pair of cousins consists of a branch of an hyperbola $\eta_+$ and a line $\eta_-$. The null curve $\gamma$ is the orbit of a $1$-parameter group of isometries of type $(H,P)$. This non-closed null curve tends asymptotically to two different points of the ideal boundary. See Figure \ref{F5}.
	\begin{figure}[h]
		\begin{center}
			\makebox[\textwidth][c]{
				\includegraphics[height=4cm,width=4cm]{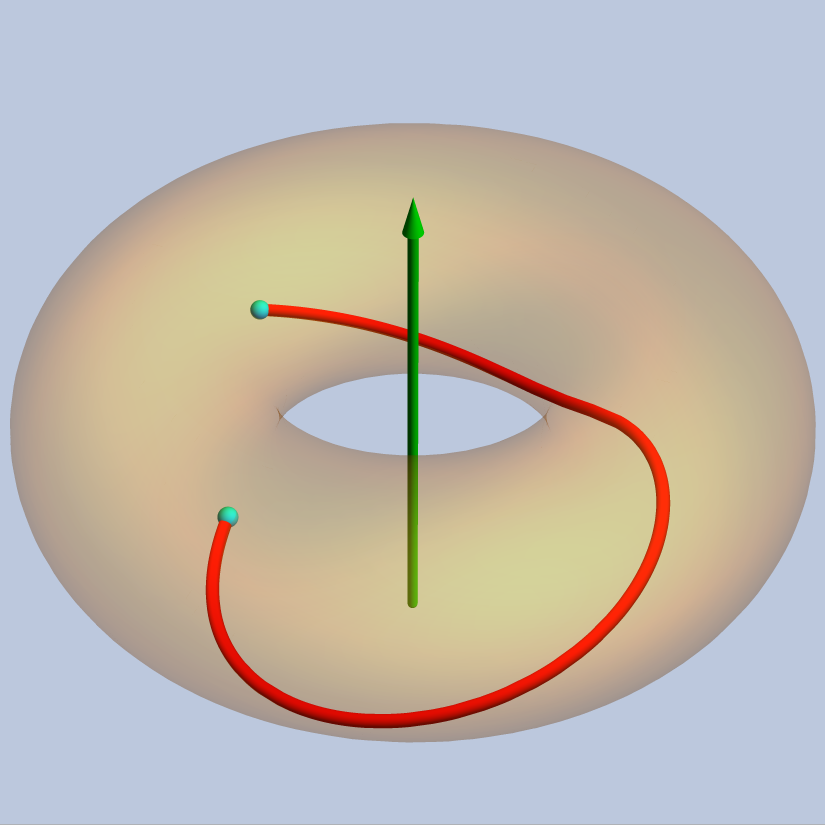}\quad\quad\quad
				\includegraphics[height=4cm,width=4cm]{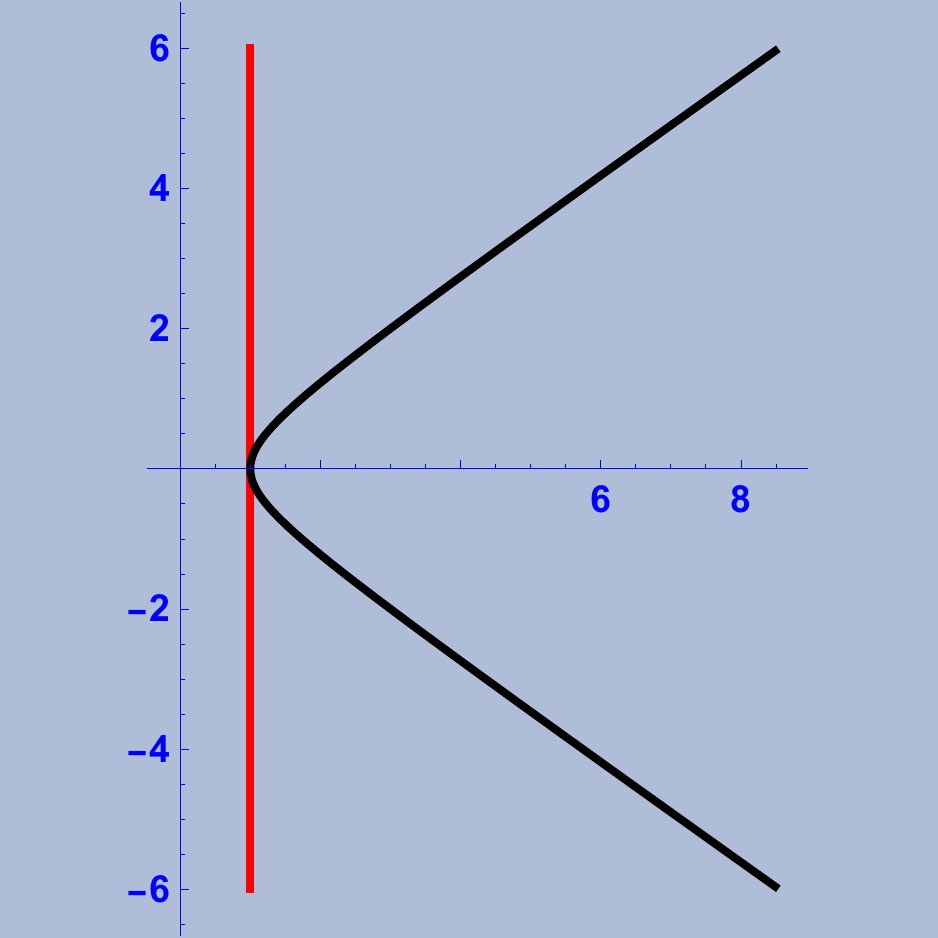}
			}
			\caption{\small{Left: The null curve $\gamma$ with constant bending $\bending=1$ in the torical model for $\AdS$ and the two limiting points. Right: The pair of star-shaped cousins $(\eta_+,\eta_-)$ in black and red, respectively.}} \label{F5}
		\end{center}
	\end{figure}
	\item Case $\bending>1$. In this case both $\eta_+$ and $\eta_-$ are two branches of hyperbolas and the null curve $\gamma$ is the orbit of a $1$-parameter group of isometries of type $(H,H)$. It tends asymptotically to two different points of the ideal boundary, as in the case $\bending=1$. See Figure \ref{F6}.
	\begin{figure}[h]
		\begin{center}
			\makebox[\textwidth][c]{
				\includegraphics[height=4cm,width=4cm]{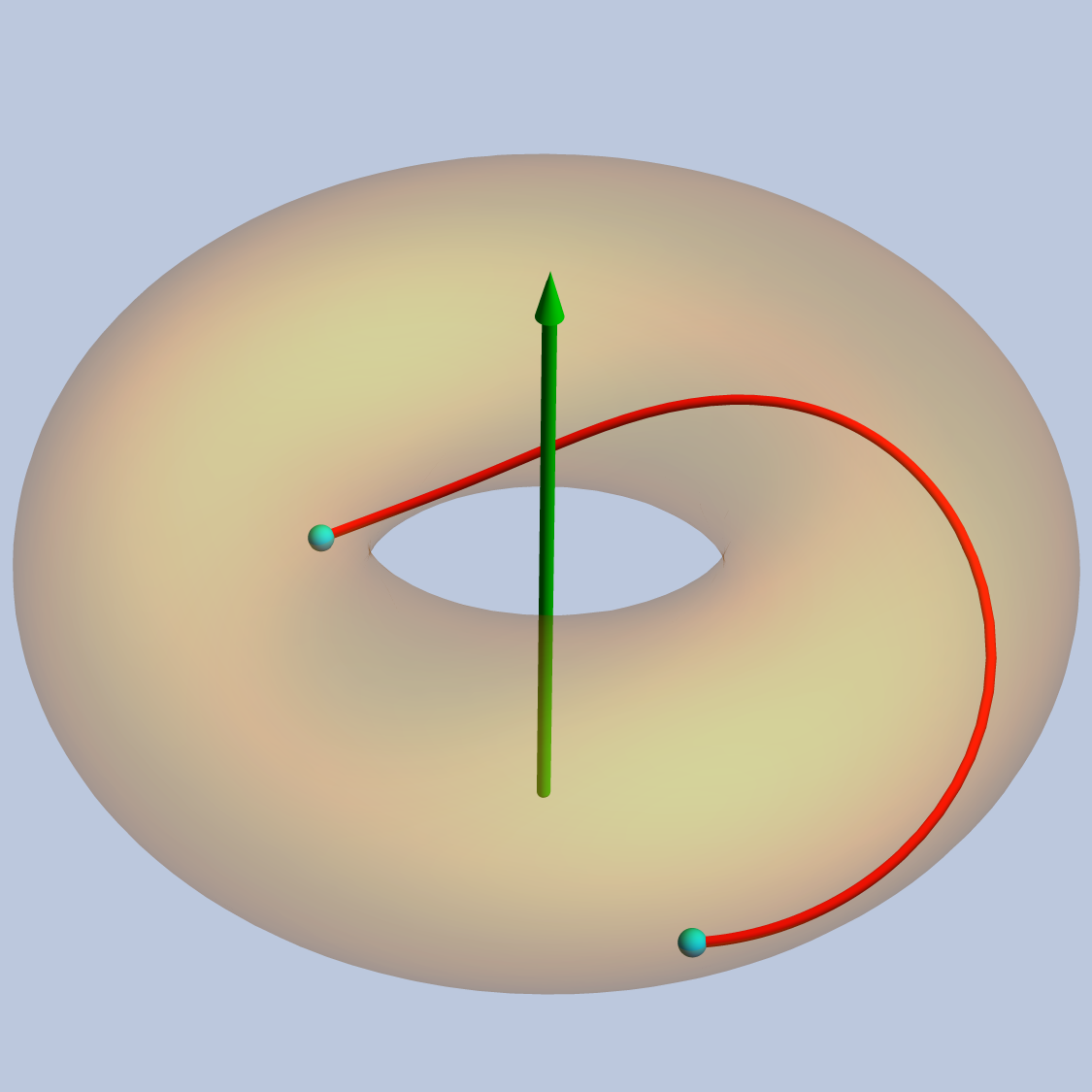}\quad\quad\quad
				\includegraphics[height=4cm,width=4cm]{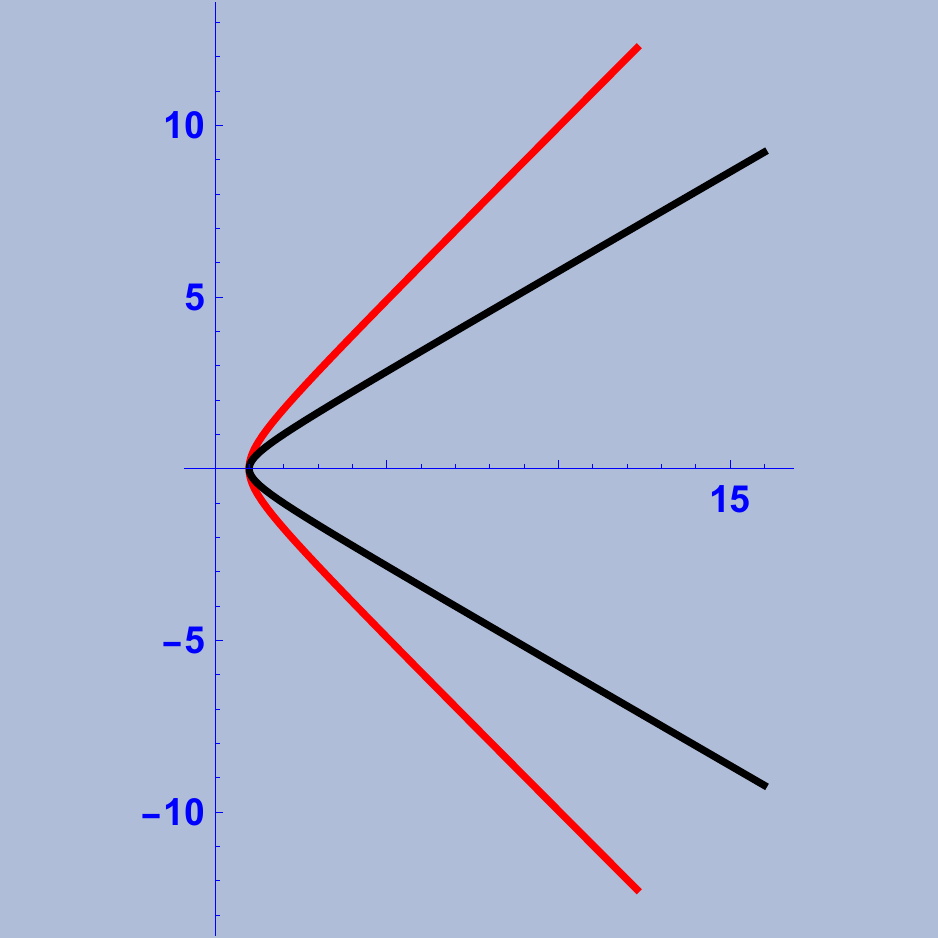}
			}
			\caption{\small{Left: The null curve $\gamma$ with constant bending $\bending=2$ in the torical model for $\AdS$ and the two limiting points. Right: The pair of star-shaped cousins $(\eta_+,\eta_-)$.}} \label{F6}
		\end{center}
	\end{figure}
\end{enumerate}

\section*{Appendix B. Construction of the Higher Order Flows}

\subsection{Notations}

The space of $n$-th order jets of functions $f\in\mathcal{C}^\infty(\mathbb{R},\mathbb{R})$ is denoted by $\mathfrak{J}^n(\mathbb{R},\mathbb{R})$. The independent variable is denoted by $s$, while the dependent variable and its virtual derivatives up to order $n$ are denoted by $u$ and $u_{(1)},...,u_{(n)}$, respectively. The projective limit of the natural sequence
$$\cdots\rightarrow\mathfrak{J}^n(\mathbb{R},\mathbb{R})\rightarrow\mathfrak{J}^{n-1}(\mathbb{R},\mathbb{R})\rightarrow\cdots\rightarrow\mathfrak{J}^1(\mathbb{R},\mathbb{R})\rightarrow\mathfrak{J}^0(\mathbb{R},\mathbb{R})\rightarrow\mathbb{R}\,,$$
is the \emph{infinite jet space}, denoted by $\mathfrak{J}(\mathbb{R},\mathbb{R})$. If $f$ is a smooth function of the variables $s$ and $t$, we put
$$j_s(f):(s,t)\in \R^2\longmapsto (s,f(s,t),\partial_sf(s,t),\dots ,\partial_s^n f(s,t),\dots) \in \mathfrak{J}(\mathbb{R},\mathbb{R})\,.$$

A function $\mathfrak{p}:\mathfrak{J}(\mathbb{R},\mathbb{R})\longrightarrow\mathbb{R}$ is a \emph{polynomial differential function} of order $n$ if there exists a polynomial $p\in\mathbb{R}[x_0,...,x_n]$ such that $\mathfrak{p}(\boldsymbol{u})=p(u,u_{(1)},...,u_{(n)})$ for every $\boldsymbol{u}=(s,u,u_{(1)},...,u_{(n)},...)\in\mathfrak{J}(\mathbb{R},\mathbb{R})$. The set of all polynomial differential functions is denoted by $\mathfrak{J}[\boldsymbol{u}]$.

The \emph{total derivative} $D$ and \emph{variational derivative} $\mathcal{E}$ of a polynomial differential function $\mathfrak{p}\in\mathfrak{J}[\boldsymbol{u}]$ are defined by
$$D(\mathfrak{p})\lvert_{\boldsymbol{u}}=\sum_{i=0}^\infty \frac{\partial p}{\partial x_i}\lvert_{\boldsymbol{u}} u_{(i+1)}\,, \quad\quad\quad \mathcal{E}(\mathfrak{p})\lvert_{\boldsymbol{u}}=\sum_{i=0}^\infty (-1)^iD^i\left(\frac{\partial p}{\partial x_i}\right)\lvert_{\boldsymbol{u}}\,,$$
respectively.

A polynomial differential function $\mathfrak{p}\in\mathfrak{J}[\boldsymbol{u}]$ is a \emph{total divergence} if there exists $\mathfrak{q}\in\mathfrak{J}[\boldsymbol{u}]$ such that $\mathfrak{p}=D(\mathfrak{q})$. We say that $\mathfrak{q}$ is a \emph{primitive} of $\mathfrak{p}$. The primitive is unique up to an additive constant.  If  $\mathfrak{p}$ is a total divergence we denote by $D^{-1}({\mathfrak p})$ the unique primitive of $\mathfrak{p}$ such that 
$D^{-1}({\mathfrak p})|_{\bf{0}}=0$. 

We next recall three basic properties:
\begin{enumerate}
	\item The polynomial differential function $\mathfrak{p}\in\mathfrak{J}[\boldsymbol{u}]$ is a total divergence if and only if $\mathcal{E}(\mathfrak{p})=0$.
	\item For every $\mathfrak{p}\in\mathfrak{J}[\boldsymbol{u}]$, $u_{(1)}\mathcal{E}(\mathfrak{p})$ is a total divergence.
	\item For every $\mathfrak{p}\in\mathfrak{J}[\boldsymbol{u}]$, $uD\mathcal{E}(\mathfrak{p})$ is a total divergence.
\end{enumerate}

Let ${\mathcal D}: {\mathfrak{J}[\boldsymbol{u}] }\longrightarrow\mathfrak{J}[\boldsymbol{u}]$ be the linear operator defined by
\begin{equation}\label{D}
	{\mathcal D}(\mathfrak{p})=D^3(\mathfrak{p})-4uD(\mathfrak{p})-2u_{(1)}\mathfrak{p}\,.
\end{equation}
Denote by $\mathfrak{P}[\boldsymbol{u}]$ the kernel of  $\mathcal{E}\circ{\mathcal D}$. For each $\mathfrak{p}\in \mathfrak{P}[\boldsymbol{u}]$, ${\mathcal D}(\mathfrak{p})$ belongs to ${\rm Ker}(\mathcal{E})$. Thus $D^{-1}{\mathcal D} : \mathfrak{P}[\boldsymbol{u}]\longrightarrow \mathfrak{J}[\boldsymbol{u}]$ is a well defined linear operator.
 
\subsection{The KdV Differential Polynomials and the KdV Hierarchy} 

The KdV polynomial differential functions ${\mathfrak p}_n\in {\mathfrak{J}[\boldsymbol{u}] }$ can be defined by the Lenard recursive formula
\begin{equation}\label{Leneard}
	{\mathfrak p}_0 = 1\,,\quad\quad\quad {\mathfrak p}_1 = u\,,\quad\quad\quad D({\mathfrak p}_n) = {\mathcal D} ({\mathfrak p}_{n-1})\,,\quad n\ge 2\,,
\end{equation}
where $D$ is the total derivative and $\mathcal{D}$ is the linear operator defined in \eqref{D}. From the definition \eqref{Leneard} it follows that ${\mathfrak p}_n$ is a polynomial differential function of order $2(n-1)$. The $n$-th KdV equation is the evolution equation of order $2n+1$ given by
\begin{equation}\label{KdVn}
\partial_t f + D({\mathfrak p}_{n+1})|_{j_s (f)}=0\,.
\end{equation}
In particular, the $0$th KdV equation is the wave equation $\partial_tf+\partial_sf=0$, the $1$st KdV equation is, precisely, \eqref{KdV2}, and the $2$nd KdV equation is
$$\partial_sf+30f^2\partial_sf-20\partial_sf\partial_s^2f-10f\partial_s^3f+\partial_s^5f=0\,.$$

Define a family of polynomial differential functions ${\mathfrak h}_n\in {\mathfrak{J}[\boldsymbol{u}] }$ by
\begin{equation}\label{clKdVn}
	{\mathfrak h}_n({\boldsymbol u})=\int_0^1 { \mathfrak p}_{n}|_{\epsilon {\boldsymbol u} }u\,d\epsilon\,.
\end{equation}
Then, ${\mathfrak p}_{n} = {\mathcal E}({\mathfrak h}_n)$ and (\ref{KdVn}) can be rewritten in the bi-Hamiltonian form
\begin{equation}\label{KdVHam}
 \partial_t f + D {\mathcal E}({\mathfrak h}_{n+1})|_{j_s (f)}=0\,,\quad\quad\quad   \partial_t f + {\mathcal D} {\mathcal E}({\mathfrak h}_{n})|_{j_s (f)}=0\,.
\end{equation}
The relevance of the Hamiltonian formulation of the KdV hierarchy stems from the fact that (\ref{KdVHam}) implies that
the functionals $f\to \int {\mathfrak h}_n|_{j_s(f)}ds$ are conservation laws of the KdV equation. Using $\eqref{clKdVn}$, the conservation laws can be explicitly computed. 

\begin{remark} \emph{The existence of infinite countably many conservation laws in involution for the KdV equation was the starting point of the Lax's proof of the Cauchy problem for rapidly decaying initial data (\cite{L0}) and for the implementation of the inverse scattering method (\cite{GGKM}).  In addition to that, this fact was used again in \cite{L1,L2} to prove the existence, for every $n\in {\mathbb N}$, of $n$-dimensional tori in the space of periodic functions $f:\R\longrightarrow \R$ which are invariant by the KdV flow. Moreover, the time-evolution of $f\in {\mathbb T}^n$ is almost-periodic in time.}
\end{remark}

\subsection{Higher Order LIEN Flows}  

Let $\{{\mathfrak r}_n\}$, $\{{\mathfrak q}_n\}$, $\{{\mathfrak a}_n\}$, and $\{{\mathfrak b}_n\}$ be the sequences of polynomial differential functions defined by
\begin{eqnarray}
	&{\mathfrak r}_0&=2{\mathfrak p}_0\,,\quad\quad\quad {\mathfrak r}_1=2{\mathfrak p}_1-4\,,\quad\quad\quad  {\mathfrak r}_n=2 {\mathfrak p}_n+4 {\mathfrak r}_{n-1}\,,\label{PLIENDF}\\
	&{\mathfrak q}_0&=2{\mathfrak p}_0\,,\quad\quad\quad {\mathfrak q}_1=2{\mathfrak p}_1+4\,,\quad\quad\quad  {\mathfrak p}_n=2 {\mathfrak p}_n-4 {\mathfrak r}_{n-1}\,,\label{PLIENDF2}
\end{eqnarray}
and by
\begin{equation}\label{LIENDF} 
	{\mathfrak a}_n={\mathfrak r}_n+{\mathfrak q}_n\,,\quad\quad\quad {\mathfrak b}_n={\mathfrak r}_n-{\mathfrak q}_n\,.
\end{equation}
The $n$-th LIEN flow is the evolution equation for null curves in $\AdS$ given by
\begin{equation}\label{LIENn}
	\partial_t\gamma=\frac{1}{\sqrt{2}}\left({\mathfrak a}_n+u{\mathfrak b}_n-\frac{1}{2}D^2{\mathfrak b}_n\right)|_{j_s(\kappa)} \,T+\frac{1}{2}D{\mathfrak b}_n|_{j_s(\kappa)} \,N+\frac{1}{\sqrt{2}}{\mathfrak b}_n|_{j_s(\kappa)} \, B\,.
\end{equation}
The $0$th LIEN flow is the trivial flow $\partial_t\gamma = 2\sqrt{2}\,T$, the $1$st LIEN flow gives back (\ref{LIEN2}). Explicit expressions of the higher order flows can be obtained with a simple recursive formula. For instance, the $2$nd LIEN flow is given by
$$\partial_t\gamma= -2\sqrt{2}\left(\partial^2_s\bending -\bending^2+8\right)T+8\partial_s\bending N+8\sqrt{2}\bending B\,.
$$

We next prove the analogue of Theorem \ref{induced} for higher order LIEN flows.

\begin{thm}\label{inducedn} Let $\gamma:J\times I\subseteq\R^2\longrightarrow\AdS$ be a solution of the $n$-th LIEN flow \eqref{LIENn}, then the bending $\bending(s,t)$ of $\gamma(s,t)$ evolves according to the n-th KdV equation \eqref{KdVn}.  
	
Conversely, if $\bending:J\times I\subseteq\R^2\longrightarrow\R$ is a smooth solution of the n-th KdV equation \eqref{KdVn}, then 
there exist a solution of the $n$-th LIEN flow \eqref{LIENn} with bending $\bending$. Moreover, any other solution of the $n$-th LIEN flow with bending $\bending$ is congruent to $\gamma$.
\end{thm}
\begin{proof} We begin by proving a Lax pair formulation of the $n$-th KdV equation \eqref{KdVn} which is suitable for our purposes. Fix $n\ge 0$ and consider the polynomial differential functions
\begin{equation}\label{pt}
\begin{cases}
{\mathfrak x}^2_1 &=\frac{1}{\sqrt{2}}( {\mathfrak a}_n+u {\mathfrak b}_n-\frac{1}{2}D^2 {\mathfrak b}_n),\\
{\mathfrak x}^3_1 &=\frac{1}{2}D {\mathfrak b}_n,\\
{\mathfrak x}^4_1 &=\frac{1}{\sqrt{2}} {\mathfrak b}_n,\\
{\mathfrak x}^2_2 &=-\frac{1}{2}D {\mathfrak a}_n,\\
{\mathfrak x}^3_2 &=-\frac{1}{\sqrt{2}} {\mathfrak a}_n,\\
{\mathfrak x}^2_3 &=-\frac{1}{\sqrt{2}} ({\mathfrak b}_n+u{\mathfrak a}_n-\frac{1}{2}D^2{\mathfrak a}_n).
\end{cases}
\end{equation}
Let ${\mathfrak K}$, ${\mathfrak P}_n$ be the ${\mathfrak g}$-valued polynomial differential functions
\begin{equation}\label{MPDF}
	{\mathfrak K}=\begin{pmatrix} 0 & 0 & 0 & \sqrt{2} \\ \sqrt{2} & 0 & \sqrt{2}\,u & 0 \\ 0 & \sqrt{2} & 0 & -\sqrt{2}\,\,u \\ 0 & 0 & -\sqrt{2} & 0 \end{pmatrix},\quad
{\mathfrak P}_n=\begin{pmatrix} 0 & {\mathfrak x}^4_1 & {\mathfrak x}^3_1 & {\mathfrak x}^2_1\\
 {\mathfrak x}^2_1,&{\mathfrak x}^2_2& {\mathfrak x}^2_3&0\\
  {\mathfrak x}^3_1& {\mathfrak x}^3_2&0&- {\mathfrak x}^2_3\\
   {\mathfrak x}^4_1&0&- {\mathfrak x}^3_2& {\mathfrak x}^2_2
 \end{pmatrix}.
\end{equation}
Recall that $\mathfrak{g}$ is the Lie algebra $\mathfrak{g}=\{X\in\R^{2,2}\,\lvert\, X^tg+gX=0\}$ (see Section 2).

Using (\ref{PLIENDF}), (\ref{LIENDF}) and taking into account the Lenard recursive formula (\ref{Leneard}) it follows that a smooth function $\bending:J\times I\subseteq\R^2\longrightarrow \R$ is a solution of the $n$-th KdV equation \eqref{KdVn} if and only if the exterior differential $1$-form $\Gamma={\mathfrak K}|_{j_s(\bending)}ds+{\mathfrak P}_n|_{j_s(\bending)}dt$ satisfies the Maurer-Cartan equation $d\Gamma+\Gamma\wedge \Gamma=0$. 

Let $\gamma:J\times I\subseteq\R^2\longrightarrow\AdS$ be a solution of the $n$-th LIEN flow \eqref{LIENn}. Then the Cartan frame field ${\mathcal F}$ is a solution of the linear system $d{\mathcal F} = {\mathcal F}\,\Gamma$.  Then, from what was said above, it follows that the bending is a solution of the $n$-th KdV equation \eqref{KdVn}.  Conversely, suppose that $\bending$ is a solution of the $n$-th KdV equation.  Consider the 1-form $\Gamma$.  Since $d\Gamma+\Gamma\wedge \Gamma=0$, there exist a smooth map ${\mathcal F}:J\times I\subseteq\R^2\longrightarrow \Cartan$ (unique up to the action of the restricted automorphism group) such that $d{\mathcal F}={\mathcal F}\,\Gamma$. Let $\gamma:J\times I\subseteq\R^2\longrightarrow \AdS$ be the first column vector of $ {\mathcal F}$.  Then $\gamma$ is a variation of null curves parameterized by the proper time whose bending evolves according to the $n$-th KdV equation \eqref{KdVn} with Cartan frame field $ {\mathcal F}$.  From (\ref{pt}) and (\ref{MPDF}) it follows that $\gamma$ is a solution of the $n$-th LIEN flow.\end{proof}

Since the bending of a solution of the LIEN flow evolves with the KdV equation, we deduce the following result.

\begin{cor} The functionals $\gamma\longmapsto \int_{\gamma}{\mathfrak h}_n|_{j_s(\bending)}ds$ are constant along the solutions of the LIEN flow.
\end{cor}

\bibliographystyle{amsalpha}

\begin{thebibliography}{AA}

\bibitem{AD} C. C. Adams, \textit{The Knot Book: An Elementary Introduction to the Mathematical Theory of Knots}, American Mathematical Society, 2004.

\bibitem{AGL} J. Amor, A. Gim\'enez and P. Lucas, Hamiltonian Structure for Null Curve Evolution, \emph{Nonlinearity} \textbf{27} (2014), 2627--2641.

\bibitem{ABG} J. Arroyo,  M. Barros and O. J. Garay,  Model of Relativistic Particle with Curvature and Torsion Revisited, \emph{Gen. Relativ. Gravit.} \textbf{36-6} (2004), 1441--1451.

\bibitem{AMP} M. A. Alejo, C. M\~{u}noz and J. M. Palacios, On the Variational Structure of Breather Solutions II; Periodic mKdV-Equation,  \emph{Electron. J. Differ. Equ.} \textbf{56} (2017), 1--26.

\bibitem{BFJL} M. Barros, A. Ferr\'andez, M. A. Javaloyes and P. Lucas, Relativistic Particles with Rigidity and Torsion in $D=3$ Spacetimes,  \emph{Class. Quantum Grav.} \textbf{22} (2005), 489--513.

\bibitem{BDFP} F. Burstall,  M. Donaldson, F. Pedit and U. Pinkall,  Isothermic Submanifolds of Symmetric R-spaces,  \emph{Crelle} \textbf{660} (2011), 191--243.




\bibitem{CIB1} A. Calini, T. Ivey and G. Mar\'i-Beffa, Integrable Flows for Starlike Curves in Centroaffine Space, \emph{SIGMA} \textbf{9} (2013), 022.

\bibitem{CIB2} A. Calini, T. Ivey and G. Mar\'i-Beffa, Remarks on KdV-Type Flows on Star-Shaped Curves, \emph{Physica D} \textbf{238} (2009), 788--797.

\bibitem{CIM} A. Calini, T. Ivey and E.Musso, mKdV-related Flows for Legendrian Curves in the Pseudo-Hermitian 3-Sphere, arXiv:2308.10125v1 [math.DG] 19 Aug 2023.

\bibitem{ChM} S. S. Chern and J. Moser, Real Hypersurfaces in Complex Manifolds, \emph{Acta Math.} \textbf{133} (1974), 219--271.

\bibitem{CP} S. S. Chern and C. K. Peng, Lie  Groups and KdV Equation, in  \textit{A Mathematician and his Mathematical Work, Selected Papers of S. S. Chern}, World Scientific Series in 20th Century Mathematics, World Sceintific (1996).

\bibitem{CQ1} K. S. Chou and C. Z. Qu, Integrable Equations and Motions of Plane Curve, \emph{Proc. IAMM NAS Ukraine} \textbf{43} (2002), 281--290.

\bibitem{CQ} K. S. Chou and C. Z. Qu, The KdV Equation and Motion of Plane Curves, \emph{J. Phys. Soc. Japan} \textbf{70} (2001), 1912--1916.

\bibitem{CQ3} K. S. Chou and C. Z. Qu, Integrable Equations Arising from Motions of Plane Curves, \emph{Phys. D} (2002), 639–-33.

\bibitem{CQ4} K. S. Chou and C. Z. Qu, Integrable Equations Arising from Motions of Plane Curves: II,  \emph{J. Nonlinear Sci.} \textbf{13} (2003), 487–-517.

\bibitem{Di} L. A. Dickey,  \textit{Soliton Equations and Hamiltonian System}, Advanced Series in Mathematical Physics, Vol. 12, World Scientific Publishing, 1991.

\bibitem{DMN} A. Dzhalilov, E. Musso and L. Nicolodi, Conformal Geometry of Timelike Curves in the (1+2)-Einstein
Universe, \emph{Nonlinear Anal.} \textbf{143} (2016), 224--255.


\bibitem{FK1} A. Fujioka and T. Kurose, Hamiltonian Formalism for the Higher KdV Flows on the Space of Closed Complex Equicentroaffine Curves, \emph{Int. J. Geom. Methods Mod. Phys.} \textbf{7} (2010), 165--175.

\bibitem{FK2} A. Fujioka and T. Kurose, Multi-Hamiltonian Structures on Space of Closed Equi-Centroaffine Plane Curves Associated to Higher KdV Flows, \emph{preprint}, arXiv:1310.1688 [math.DG].


\bibitem{GGKM} C. S. Gardner, J. M. Greene, M. D. Kruskal and R. M. Miura, Korteweg–de Vries Equation and Generalizations:
VI. Methods for Exact Solutions, \emph{Commun. Pure Appl. Math.} \textbf{27} (1974), 97--133.

\bibitem{KKSH} P. G. Kevrekidis, A. Khare, A. Saxena and G. Herring, On Some Classes of mKdV Periodic Solutions, 
\emph{J. Phys. A: Math. Gen.} {\textbf 37} (2004), 10959--10965.

\bibitem{KV} D. J. Korteweg and G. De Vries, On the Change of Form of Long Waves Advancing in a Rectangular Canal, and on a New Type of Long Stationary Waves, \emph{Philos. Mag.} \textbf{39-240} (1895), 422--443.

\bibitem{L0} P. D. Lax, Integrals of Nonlinear Equations of Evolution and Solitary Waves, \emph{Comm. Pure Appl. Math.} \textbf{21} (1968), 467--490.

\bibitem{L2} P. D. Lax, Periodic Solutions of the KdV Equation, \emph{Commun. Pure Appl. Math.} \textbf{28} (1975), 141--188.

\bibitem{L1} P. D. Lax, Almost Periodic Solutions of the KdV Equation, \emph{SIAM Review} \textbf{18} (1976), 351--375.


\bibitem{Ma} J. Maldacena, The Large N Limit of Superconformal Field Theories and Supergravity, \emph{Adv. Theor.} \textbf{2-4} (1998), 231--252.

\bibitem{Mi} R. Miura, Korteweg-de Vries Equation and Generalizations. I. A Remarkable Explicit Nonlinear Transformation, J. Math. Phys. 9 (1968), 1202–1204.


\bibitem{MN1} E. Musso and L. Nicolodi, Closed Trajectories of a Particle Model on Null Curves in Anti-de Sitter $3$-Space, \emph{Class. Quantum Grav.} \textbf{24} (2007), 5401--5411.

\bibitem{MN1Bis} E. Musso and L. Nicolodi, Reduction for Constrained Variational Problems on 3-Dimensional Null Curves, \emph{SIAM J. Control Optim.} \textbf{47} (2008), 1399--1414.

\bibitem{MN2} E. Musso and L. Nicolodi, Hamiltonian Flows on Null Curves, \emph{Nonlinearity} \textbf{23} (2010), 2117--2129.




\bibitem{OLBC} F. W. J. Olver, D. W. Lozier, R. F. Boisvert and C. W. Clark (Eds), \emph{NIST Handbook of Mathematical Functions}, Cambridge University Press, Cambridge, 2010.

\bibitem{P} U. Pinkall, Hamiltonian Flows on the Space of Star-Shaped Curves, \emph{Results in Mathematics} \textbf{27} (1995), 328--332.

\bibitem{Ro} A. Ronveaux (Ed.), \textit{Heun's Differential Equations}, New York: The Clarendon Press Oxford University Press, 1995.

\bibitem{RW} W. P. Reinhardt and P. L. Walker, ``Jacobian Elliptic Functions'', in Olver, Frank W. J.; Lozier, Daniel M.; Boisvert, Ronald F.; Clark, Charles W. (eds.), NIST Handbook of Mathematical Functions, Cambridge University Press, ISBN 978-0-521-19225-5.

\bibitem{SK} B. D. Sleeman and V. B. Kuznetzov, (2010), ``Heun functions'', in Olver, Frank W. J.; Lozier, Daniel M.; Boisvert, Ronald F.; Clark, Charles W. (eds.), NIST Handbook of Mathematical Functions, Cambridge University Press, ISBN 978-0-521-19225-5.

\bibitem{T} S. Tabachnikov, On Centro Affine Curves and B\"{a}cklund Transformations of the KdV Equation, \emph{Arnold Math. J.} \textbf{4} (2018), 445--458.

\bibitem{TU} C. L. Terng and K. K. Uhlenbeck, B\"{a}cklund Transformations and Loop Group Actions, \emph{Commun. Pure Appl. Math.} \textbf{53-1} (2000), 1--75.

\bibitem{TW1} C. L. Terng and Z. Wu, Central Affine Curve Flow on the Plane, \emph{J. Fixed Point Theory Appl.} \textbf{14} (2013), 375--396.


\bibitem{V0} H. Volkmer, (2010) ``Lam\'e functions'',  in Olver, Frank W. J.; Lozier, Daniel M.; Boisvert, Ronald F.; Clark, Charles W. (eds.), NIST Handbook of Mathematical Functions, Cambridge University Press, ISBN 978-0-521-19225-5.

\bibitem{V} H. Volkmer, Eigenvalue Problem for Lam\'e's Differential Equation, \emph{SIGMA} \textbf{14} (2018), 131.


\bibitem{Web} S. M. Webster, Pseudo-Hermitian Structures in a Real Hypersurface, \emph{J. Diff. Geom.} \textbf{3} (1978), 25--41.

\bibitem{Wi} E. Witten, Anti-de Sitter Space and Holography, \emph{Adv. Theore.} \textbf{2-2} (1998), 253--291.

\bibitem{ZF} V. E. Zakharov and L. D. Faddeev, Korteweg-De Vries Equation: A Completely Integrable Hamiltonian System, \emph{Funct. Anal. Appl.} \textbf{5} (1971), 280--287.

\end{thebibliography}

\end{document}